\documentclass[11pt, dina4,  amssymb]{article}
\usepackage{amssymb, amsmath,amsthm}
\usepackage[titletoc, title]{appendix}
\usepackage{graphicx}
\usepackage{diagbox}
\usepackage{cite}
\usepackage{multirow}
\usepackage{epstopdf}
\usepackage{color}
\usepackage[margin=0.8in]{geometry}
\usepackage[normalem]{ulem}

\newcommand{\bxi}{\textit{{\boldmath$\xi$}}}

\newcommand{\bu}{{\bf u} }
\newcommand{\p}{\partial}

\newcommand{\Og}{\Omega}
\newcommand{\fl}[2]{\frac{#1}{#2}}
\newcommand{\dt}{\delta}
\newcommand{\gm}{\gamma}
\newcommand{\nn}{\nonumber}
\newcommand{\ap}{\alpha}

\newcommand{\veps}{\varepsilon}

\newcommand{\Dt}{\Delta}

\newcommand{\be}{\begin{equation}}
\newcommand{\ee}{\end{equation}}
\newcommand{\ba}{\begin{array}}
\newcommand{\ea}{\end{array}}
\newcommand{\bea}{\begin{eqnarray}}
\newcommand{\eea}{\end{eqnarray}}
\newcommand{\beas}{\begin{eqnarray*}}
\newcommand{\eeas}{\end{eqnarray*}}
\newtheorem{remark}{Remark}[section]
\newtheorem{lemma}{Lemma}[section]
\newtheorem{theorem}{Theorem}[section]
\newtheorem{corollary}{Corollary}[section]

\newcommand{\bx}{{\bf x} }
\newcommand{\bm}{{\bf m}}
\newcommand{\bk}{{\bf k}}

\definecolor{ForestGreen}{rgb}{0.0, 0.5, 0.0}

\title{Numerical approximations for the tempered fractional Laplacian: Error analysis and applications}
\author{
Siwei Duo\thanks{Department of Mathematics, University of South Carolina, Columbia, SC 29208 (Email: duo@mail.sc.edu)} \  \  and 
Yanzhi Zhang\thanks{Department of Mathematics and Statistics, Missouri University of Science and Technology, Rolla, MO 65409-0020 (Email: zhangyanz@mst.edu) }}
\date{ }

\begin{document}
\maketitle
\vspace{-5mm}
\begin{abstract}
In this paper, we propose an accurate finite difference method to discretize the $d$-dimensional (for $d \ge 1$) tempered integral fractional Laplacian and apply it to study the tempered effects on the solution of problems arising in various applications.  
Compared to other existing methods, our method has higher accuracy and simpler implementation.
Our numerical method has an accuracy of ${\mathcal O}(h^\veps)$, for $u \in C^{0, \,\ap + \veps} (\bar{\Og})$ if $\ap < 1$ (or $u \in C^{1, \,\ap - 1 + \veps} (\bar{\Og})$ if $\ap \ge 1$) with $\veps > 0$,  suggesting the minimum consistency conditions. 
The  accuracy  can be improved to ${\mathcal O}(h^2)$, for $u \in C^{2, \,\ap + \veps} (\bar{\Og})$ if $\ap < 1$ (or $u \in C^{3, \,\ap - 1 + \veps} (\bar{\Og})$ if $\ap \ge 1$). 
Numerical experiments confirm our analytical results and provide insights in solving the tempered fractional Poisson problem. 
It suggests that to achieve the second order of accuracy, our method only requires the solution $u \in C^{1,1}(\bar{\Og})$ for any $\ap \in (0, 2)$. 
Moreover, if the solution of  tempered fractional Poisson problems satisfies $u \in C^{p, s}(\bar{\Og})$ for $p = 0, 1$ and $s\in(0, 1]$, our method has the accuracy of ${\mathcal O}(h^{p+s})$.
Since our method yields a (multilevel) Toeplitz {stiffness} matrix, one can design fast algorithms via the fast Fourier transform for efficient simulations. 
Finally, we apply it together with fast algorithms to study the tempered effects on the solutions of various  tempered fractional PDEs, including the Allen--Cahn equation and Gray--Scott equations. 

\smallskip
{\bf Key words. } Tempered integral fractional Laplacian, finite difference methods, error estimates,   fractional Allen--Cahn equation,  fractional Gray--Scott equations.
\end{abstract}

\section{Introduction}
\setcounter{equation}{0}
\label{section1}

The anomalous diffusion of L\'evy motion, in contrast to the normal diffusion of Brownian motion,  has gained a lot of attention in the last couple of decades \cite{Laskin2000, Carr2003, Cartea2007, Javanainen2013, Duo2016, Kirkpatrick2016}. 
Recently, the coexistence and transition of anomalous to normal diffusion was observed in many fields, ranging from biology \cite{Javanainen2013}, finance \cite{Carr2002,Carr2003}, turbulence \cite{Dubrulle1998}, to geophysics \cite{Meerschaert2008, Zhang2012}.  
To model such a phenomenon, several approaches were proposed in the literature, such as truncating a stable L\'evy process \cite{Mantegna1994,Koponen1995}, adding a high-order power law factor \cite{Sokolov2004}, including a nonlinear friction term \cite{Chechkin2005}, and exponentially tempering a stable  L\'evy process \cite{Cartea2007, Rosinski2007}. 
In the tempered models \cite{Cartea2007, Rosinski2007}, a damping term is introduced to exponentially temper the power-law decay of the L{\'e}vy process. 
Hence, they can capture the transition phenomena of anomalous diffusion in the early stage and then normal diffusion in the late stage.  
It shows that exponential tempering offers many technical advantages over other approaches \cite{Cartea2007, Rosinski2007, Meerschaert2008, Baeumer2010, Zhang2012}. 
However, the current mathematical and numerical studies of the tempered models still remain limited. 
In this study,  we propose  an efficient and accurate finite difference method to discretize the $d$-dimensional ($d \ge 1$) tempered  fractional Laplacian. 

The tempered fractional Laplacian $(-\Dt + \lambda)^{\fl{\ap}{2}}$ is defined in hypersingular integral form \cite{Meerschaert2008, Zhang0017-2, Sun0018}:
\bea\label{TFL}
(-\Dt + \lambda)^{\fl{\ap}{2}} u(\bx) = c_d^{\ap, \lambda}\,{\rm P. V.} \int_{{\mathbb R}^d}\fl{u({\bf x}) - u(\bf y)}{e^{\lambda|{\bx} - {\bf y}|}|{\bf x} - {\bf y}|^{d+\ap}} d{\bf y}, \qquad\mbox{for} \ \  \ap \in  (0, 2),\quad \lambda \ge 0,
\eea
for $d = 1, 2$, or $3$, where P.V. stands for the principal value, and $|\bx- {\bf y}|$ denotes the Euclidean distance
between points $\bx$ and ${\bf y}$.
The normalization constant $c_d^{\ap, \lambda}$ is defined as
\begin{eqnarray*}
	c_d^{\alpha,\lambda}= \fl{1}{2 \sqrt{\pi^d}}\left\{
	\begin{array}{ll}
\displaystyle
		2^\ap\alpha \Gamma\Big(\fl{d+\alpha}{2}\Big)/\Gamma\Big(1-\fl{\alpha}{2}\Big), &\quad \text{if} \ \ \lambda = 0 \ \ \text{or} \ \ \alpha = 1, \\ 
\displaystyle
		\Gamma\Big(\fl{d}{2}\Big)/|\Gamma(-\alpha)|,& \quad \text{otherwise}	\end{array}
	\right.
\end{eqnarray*}
with $\Gamma(\cdot)$ being the Gamma function. 
Probabilistically, the operator (\ref{TFL}) represents an infinitesimal generator of a tempered symmetric $\alpha$-stable L{\'e}vy process \cite{Meerschaert2008, Baeumer2010, Rosinski2007}.
As shown in \cite{Meerschaert2008, Rosinski2007},  the tempered $\alpha$-stable L{\'e}vy process can approximate a traditional $\alpha$-stable L\'evy process over a short distance, while over a long distance it behaves like  Brownian motion. 
Hence, the tempered fractional Laplacian $(-\Dt + \lambda)^{\fl{\ap}{2}}$ couples normal and anomalous diffusion in a seamless way. 
In the special case of $\lambda = 0$, the operator (\ref{TFL}) collapses to the fractional Laplacian $(-\Dt)^{\fl{\ap}{2}}$, which has been extensively studied (see \cite{Acosta2017, Duo2018, Duo2019, Minden0018, Tang0019} and references therein). 
Note that many other tempered fractional derivatives exist in the literature, such as the tempered Riemann--Liouville derivatives \cite{Cartea2007, Baeumer2010}, tempered Caputo derivatives \cite{Cartea2007}, and tempered Riesz derivatives, but {\it in this study we focus on the tempered integral fractional Laplacian  (\ref{TFL})}. 

So far,  numerical methods for the tempered fractional Laplacian $(-\Dt + \lambda)^{\fl{\ap}{2}}$ still remain very limited. 
In one-dimensional  (i.e., $d = 1$) cases, a finite difference collocation method is presented in \cite{Zhang0017-2}  to solve the tempered fractional Poisson equation, while later a Riesz basis Galerkin method is proposed  in \cite{Zhang0017}. 
Recently, a finite difference method based on the bilinear interpolation is proposed in \cite{Sun0018} to discretize the two-dimensional (i.e., $d = 2$) operator (\ref{TFL}).  
To the best of our knowledge, the current numerical methods for the tempered  fractional Laplacian suffer two main limitations:  low-dimensional (i.e., $d = 1, 2$) discretization and  $\alpha$-dependent accuracy, e.g.,  $\mathcal{O}(h^{2-\alpha})$, with $h$ a small mesh size. 
In this paper, we propose accurate finite difference methods to discretize the general $d$-dimensional ($d\geq 1$) tempered integral fractional Laplacian (\ref{TFL}). 
The main contributions of this study include: 
\vspace{-2mm}
\begin{itemize}\itemsep -2pt
\item[(i)]  New finite difference methods are proposed to discretize $d$-dimensional (for $d= 1, 2, 3$) tempered fractional Laplacian $(-\Dt+\lambda)^{\fl{\ap}{2}}$. 
In contrast to other methods, our schemes take similar framework for any dimension $d \ge 1$, making both error estimates and computer implementation much easier. 
\item[(ii)] Error analysis in this study provides a tighter consistency condition.
Moreover, we prove the second-order accuracy with much less regularity requirements. 
In the special case of $\lambda = 0$, our analysis improves the error estimates in \cite{Duo2018} for the fractional Laplacian $(-\Dt)^{\fl{\ap}{2}}$. 
\item[(iii)] Our method can achieve the accuracy of ${\mathcal O}(h^2)$ for any $\ap\in (0,  2)$, in contrast to an $\ap$-dependent accuracy ${\mathcal O}(h^{2-\ap})$ of other existing methods.
Moreover, it always yields a (multilevel) Toeplitz stiffness matrix for $d \ge 1$, enabling efficient implementations via the fast Fourier transform (FFT).
\end{itemize}
\vspace{-2mm}
This paper is organized as follows. 
In Sec. \ref{section2}, we first introduce the general framework of our method and then present the detailed schemes for one-, two-, and three-dimensional cases.
In Sec.  \ref{section3},  numerical  analysis is presented to study the local truncation errors. 
In Sec.  \ref{section4},  we present numerical experiments  to test the performance of our method in approximating the operator and in solving fractional Poisson problems.  
Two tempered fractional problems, i.e., Allen--Cahn equation and  Gray--Scott equations, are presented in Sec. \ref{section5} to study the tempered effects of the operator.
Finally, we make conclusions in Sec. \ref{section6}. 

\section{Numerical methods}
\setcounter{equation}{0}
\label{section2}

The main numerical challenges in discretizing the tempered fractional Laplacian $(-\Dt + \lambda)^{\fl{\ap}{2}}$ are from its nonlocality and strong singular kernel function. 
As mentioned previously,  the existing numerical methods are limited to the one- and two-dimensional cases \cite{Zhang0017-2, Sun0018}, and no reports can be found for three-dimensional tempered fractional Laplacian. 
Moreover, these methods have an $\ap$-dependent (i.e., rate of ($2-\ap$)) accuracy. 
In this section, we propose a new and accurate finite difference method for any $d$-dimensional ($d \ge 1$) tempered fractional Laplacian  (\ref{TFL}). 
%

Let $\Og \subset {\mathbb R}^d$ be a bounded domain, and $\Og^c = {\mathbb R}^d\backslash\Og$ denotes its complement. 
Here, we consider the tempered fractional Laplacian (\ref{TFL}) on domain $\Og$ (i.e., $\bx \in \Og$) with extended homogeneous Dirichlet boundary conditions on $\Og^c$  (i.e., $u(\bx) = 0$ for $\bx\in\Og^c$).
Let $\xi^{(i)}=  |x_i - y_i|$ for $1 \le i\le d$, and define a new vector $\bxi = \big(\xi^{(1)}, \, \xi^{(2)}, \, \cdots, \, \xi^{(d)}\big)$. 
For an integer $M > 0$,  we denote the index sets 
\beas
\varkappa_M = \{(m_1, \, m_2, \, \cdots, \, m_d) \ | \ 0 \le m_i \le M, \ \mbox{for} \ 1 \le i \le d \},\qquad
\widetilde{\varkappa}_M = \varkappa_M\backslash (0,\, 0, \, \cdots\, 0).
\eeas
Then, the tempered fractional Laplacian (\ref{TFL}) can be reformulated into a weighted integral as:
\bea \label{TfL2}
(-\Dt + \lambda)^{\fl{\ap}{2}}u(\bx) = -c_d^{\ap, \lambda} \int_{{\mathbb R}_+^d} \varphi_{d, \gm}(\bx, \bxi)\,w_{\lambda, \gm}(\bxi)\,d\bxi, 
\eea
where ${\mathbb R}_+^d = [0, \infty)^d$, the weight function $w_{\lambda,\gm} = |\bxi|^{\gm-(d+\ap)}\exp(-\lambda|\bxi|)$, and 
\bea \label{Psi}
\varphi_{d,\gm}(\bx,\bxi) = \fl{1}{|\bxi|^\gm}\Big(\sum_{{\bf m} \in \varkappa_1} u(\bx + (-1)^\bm \circ \bxi) - 2^d u(\bx)\Big)
\eea
with $\gm \in (\ap, 2]$ being a splitting parameter,  the vector  $(-1)^{\bf m} = \big((-1)^{m_1}, \, (-1)^{m_2}, \, \cdots, \, (-1)^{m_d} \big)$, and ${\bf a} \circ {\bf b}$ denoting the Hadamard product of ${\bf a}$ and ${\bf b}$.  
In other words, we split the strong singular kernel function of the tempered fractional Laplacian  (\ref{TFL}) into two weaker parts, i.e., $|\bxi|^{-\gm}$ in function $\varphi_{d, \gm}$ and $|\bxi|^{\gm-(d+\ap)}$ in the weight $w_{\lambda, \gm}$.
Note that one key idea that distinguishes our method  from other existing  methods is to split the kernel function and  approximate the resulting integral by the composite weighted trapezoidal rules.  
Here, how to choose the splitting parameter $\gm$ and where to include the damping term $e^{-\lambda|\bxi|}$ play a crucial role in determining the accuracy of the finite difference methods; see more discussion  in  Remarks \ref{remark2-1} and \ref{remark3-1}. 

Denote $\Og = (a_1, b_1)\, \times\, \cdots\, \times (a_d, b_d)$. 
Choose $L = \max_{1 \leq  i  \leq d} \{b_i - a_i\}$ and define mesh size $h = L/N$ with integer $N > 0$. 
We then define the grid points  $\bxi_{{\bf k}} = \big(\xi^{(1)}_{k_1},\, \xi^{(2)}_{k_2}, \, \cdots,\, \xi^{(d)}_{k_d}\big)$ for ${\bf k} \in \varkappa_N$, with  $\xi^{(i)}_{k_i} = k_ih$. For  each $\bk \in \varkappa_{N-1}$, define the element
 ${I}_{\bk} = [k_1 h, \, (k_1+1)h]\, \times \, [k_2 h, \, (k_2+1)h] \,\times\, \cdots \,  \times [k_d h, \,(k_d+1)h]$. 
Then, the tempered fractional Laplacian in (\ref{TfL2}) can be rewritten as: 
\bea\label{split1}
(-\Dt + \lambda)^{\fl{\ap}{2}}u(\bx) = -c_d^{\ap, \lambda}\bigg(\sum _{\bk\in\varkappa_{N-1}} \int _{{I}_{\bf k}} \varphi_{d,\gm}(\bx,\bxi)w_{\lambda,\gm}(\bxi) d\bxi + \int _{D_d}\varphi_{d,\gm}(\bx,\bxi)w_{\lambda,\gm}(\bxi) d\bxi\bigg),
\eea
where $D_d = {\mathbb R}_+^d\backslash[0, L]^d$. 
It is easy to verify that if ${\bf x}\in \Omega$ and $\bxi \in D_d$,  the point $(\bx + (-1)^\bm \circ \bxi)\in \Omega^c$ for $\bm \in \varkappa_1$, and thus the extended homogeneous Dirichlet boundary conditions imply that $u(\bx + (-1)^\bm\circ \bxi) = 0$.
This fact simplifies $\varphi_{d, \gm}(\bx, \bxi) = -2^d |\bxi|^{-\gm} u(\bx)$ and reduces the second integral in (\ref{split1})  to: 
\begin{eqnarray}\label{Int2}
\int_{D_d}\varphi_{d,\gm}(\bx,\bxi) w_{\lambda, \gm}(\bxi)\,d\bxi 
= -\fl{u({\bf x}) }{2^d}\int_{D_d} e^{-\lambda|\bxi|}|\textit{{\boldmath$\xi$}}|^{-(d+\alpha)} d\bxi.
\eea
For the integrals over element $I_\bk$, we divide our discussion into two parts. 
For $|{\bk}| \neq 0$ (i.e., $\bk \in \widetilde{\varkappa}_{N-1}$), we apply the $d$-dimensional weighted trapezoidal rule and obtain: 
\begin{eqnarray}\label{I-2D}
\int _{I_\bk}\varphi_{d,\gm}(\bx,\bxi) w_{\lambda, \gm}(\bxi)\,d\bxi
\approx \fl{1}{2^d}\Big(\sum_{{\bf n} \in \varkappa_1} \varphi_{d,\gm}\big(\bx, \bxi_{\bk + {\bf n}}\big)\Big)\int _{I_{\bk}}e^{-\lambda|\bxi|} |\bxi|^{\gm -(d+\ap)}\,d\bxi.
\end{eqnarray}

For $|\bk|  = 0$,  the integral over $I_{\bf 0}$ can be similarly approximated in form of (\ref{I-2D}), but 
the term $\varphi_{d, \gm}(\bx, \bxi_{\bf 0})$ should be replaced with $\lim_{\bxi\to {\bf 0}}\varphi_{d, \gm}(\bx, \bxi)$, since a singularity occurs at $\bxi_{\bf 0} = {\bf 0}$.  
Assuming that this limit exists, then it depends on the choice of splitting parameter $\gm \in (\ap, 2]$. 
If $\gm = 2$, it is approximated as: 
\beas
\label{lim2}
\lim_{\bxi \rightarrow {\bf 0}}\varphi_{d,2}(\bx,\bxi) \approx 
\Big(\fl{2^d}{d}-1\Big)\sum_{|{\bf n}| =1}\varphi_{d,2}(\bx,\bxi_{\bf n})
- \sum_{\substack{{\bf n} \in \widetilde{\varkappa}_1, \,|{\bf n}| \neq 1}} \varphi_{d,2}(\bx,\bxi_{\bf n}),  
\eeas
while for any $\gm \in (\ap,  2)$, this limit is zero, as 
$\displaystyle \lim_{\bxi \rightarrow {\bf 0}}\varphi_{d,\gm}(\bx,\bxi) = \lim_{\bxi \rightarrow {\bf 0}}(\varphi_{d,2}(\bx,\bxi) |\bxi|^{2-\gamma})$. 
Combining the above limits and  (\ref{I-2D})  with $|\bk| = 0$ yields the approximation: 
\bea\label{I00-final}
\int _{I_{{\bf 0}}}
\varphi_{d,\gm}(\bx,\bxi) w_{\lambda,\gm}(\bxi)\,d\bxi \approx 
\fl{1}{2^d}\bigg(\sum_{{\bf n} \in \widetilde{\varkappa}_1} c_{\bf n}^\gm\, \varphi_{d,\gm}\big(\bx, \bxi_{\bf n}\big)\bigg) \int _{I_{{\bf 0}}}e^{-\lambda|\bxi|} |\bxi|^{\gm-(d+\ap)}\,d\bxi,   
\eea
where  $c_{\bf n}^\gm = 1$ for $\gm \in (\ap, 2)$ and any ${\bf n}\in\widetilde{\varkappa}_1$; if $\gm = 2$,   $c_{\bf n}^\gm = 2^d/d$ for $|{\bf n}| = 1$, and $c_{\bf n}^\gm = 0$ for $|{\bf n}| \neq  1$. 

Combining (\ref{split1})--(\ref{I00-final}) and rearranging  terms, we obtain the following approximation to the $d$-dimensional tempered fractional Laplacian $(-\Dt + \lambda)^{\fl{\ap}{2}}$:
\bea \label{scheme0}  
&&(-\Delta+\lambda)_{h,\gm}^{\fl{\alpha}{2}} u(\bx) = 
-\frac{c_{d}^{\ap,\lambda}}{2^d}\bigg[\,\sum_{\bk\in\widetilde{\varkappa}_{N-1}}
\varphi_{d,\gm}\big({\bf x},\,\bxi_{\bk}\big)\int _{T_{\bk}}e^{-\lambda|\bxi|}|\bxi|^{\gm-(d+\ap)}d\bxi - 4^du(\bx)\int_{D_d} e^{-\lambda|\bxi|}|\bxi|^{-(d+\ap)} d\bxi\nn\\
&&\hspace{1cm}+\left\lfloor\fl{\gm}{2}\right\rfloor\bigg(\Big(\frac{2^d}{d}-1\Big)\sum_{|{\bf k}| = 1}\varphi_{d, \gm}(\bx,\bxi_{\bf k})
- \sum_{{\bf k}\in\widetilde{\varkappa}_1, |{\bf k}|\neq 1} \varphi_{d, \gm}(\bx,\bxi_{\bf k})\bigg)\int _{I_{\bf 0}} e^{-\lambda|\bxi|}|\bxi|^{\gm-(d+\ap)}d\bxi\bigg],
\eea
for $\bx \in\Og$, where $\lfloor \cdot \rfloor$ represents the floor function, and  
${T}_{\bk} = \big(\bigcup_{{\bf n}\in {\varkappa}_1}{I}_{\bk-{\bf n}} \big) \bigcap (0, L)^d$ (for $\bk\in\varkappa_{N-1}$)
denotes as the collection of elements associating with $\bxi_{\bk}$, i.e., elements that have $\bxi_{\bk}$ as a vertex.

\begin{remark}[Effect of the damping term] \label{remark2-1}
We emphasize that in order to obtain the optimal accuracy for all $\ap \in (0, 2)$,  the damping term $e^{-\lambda |\bxi|}$ must be included in the weight function as written in (\ref{TfL2}) and then eventually retained in the integral. 
\end{remark}

\begin{figure}[htb!]
\centerline{
(a) \includegraphics[height=4.66cm,width=6.460cm]{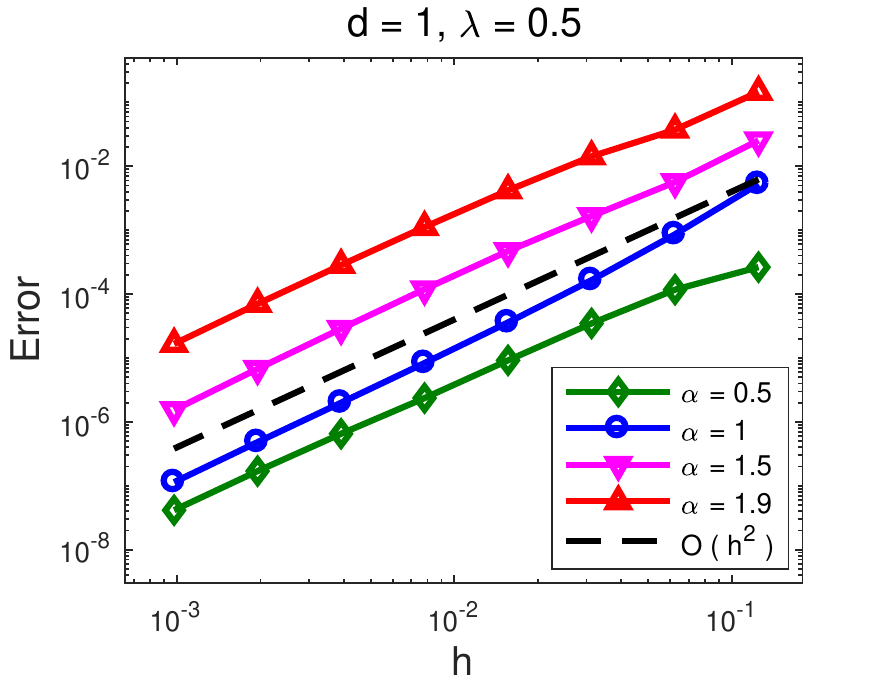}
(b) \includegraphics[height=4.66cm,width=6.460cm]{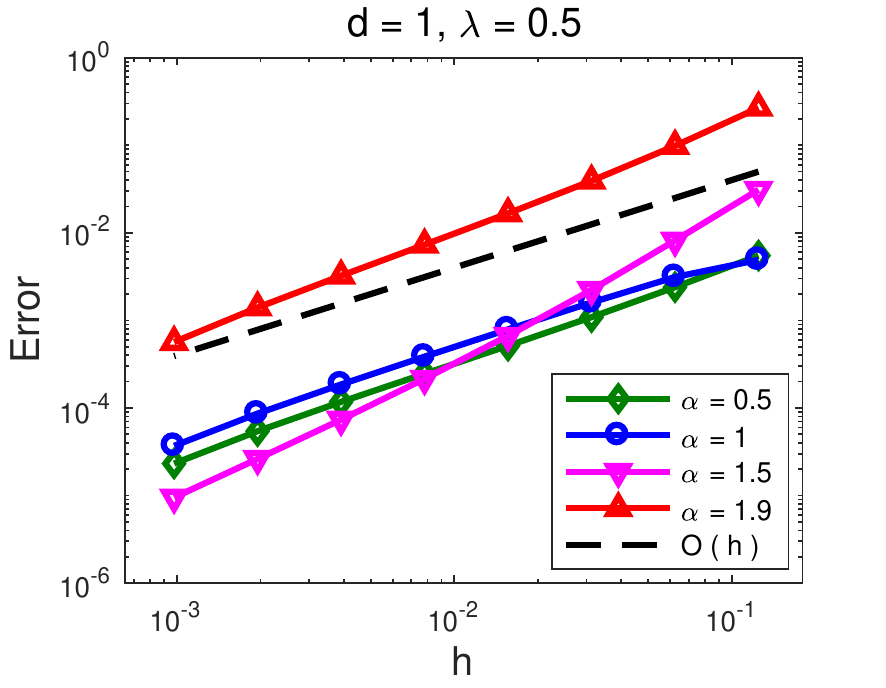}}
\vspace{-1mm}
\centerline{
(c) \includegraphics[height=4.66cm,width=6.460cm]{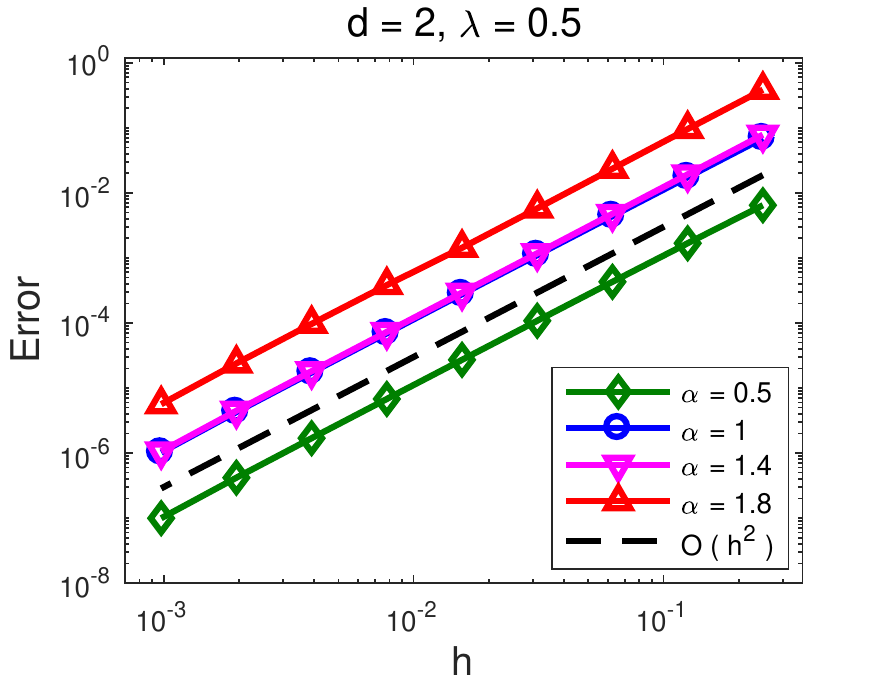}
(d) \includegraphics[height=4.66cm,width=6.460cm]{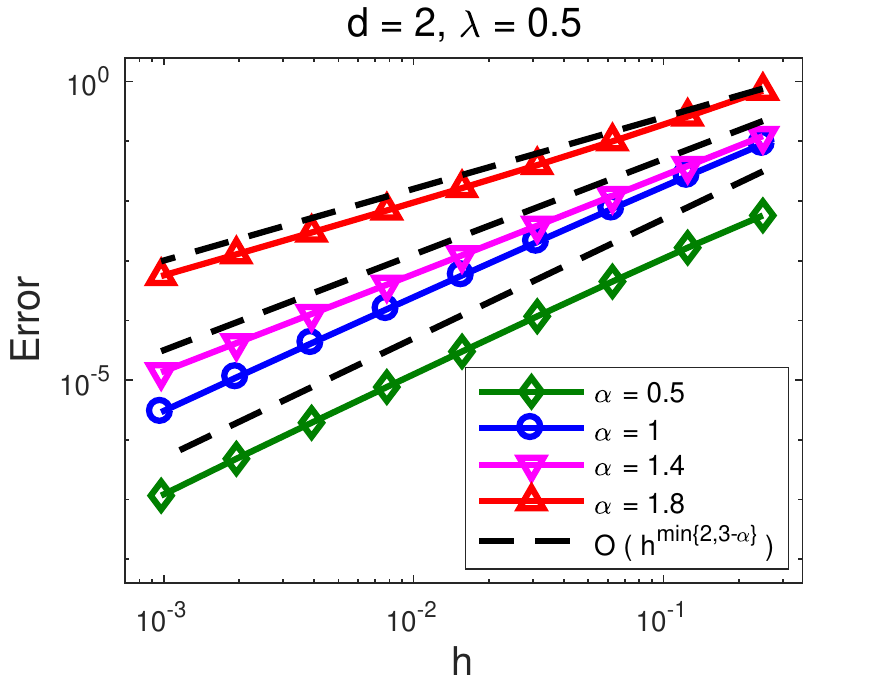}}
\caption{Comparison of numerical errors in approximating $(-\Dt+\lambda)^{\fl{\ap}{2}}u$.
(a)\&(c):  Including $e^{-\lambda|\bxi|}$ in the weight function as that in (\ref{TfL2});
(b)\&(d):  excluding $e^{-\lambda|\bxi|}$ from the weight function.}\label{Fig1}
\end{figure}
To demonstrate it,   we compare in Fig. \ref{Fig1} numerical errors of two methods obtained by, respectively, including and excluding  the damping term from the weight function, where we choose 
$u = (1-x^2)_+^{3+\fl{\ap}{2}}$, $\lambda = 0.5$ and $\gm = 2$.
Fig. \ref{Fig1} (a)\&(c) show that including the damping term $e^{-\lambda|\bxi|}$ in the weight function leads to the second-order accuracy, independent of  $\ap$ and $d$.  
By contrast,  excluding it from the weight function yields a method with an $\ap$-dependent accuracy; see Fig. \ref{Fig1} (b)\&(d). 
Moreover, the numerical errors in Fig. \ref{Fig1} (b) \& (d) are much larger than those in Fig. \ref{Fig1} (a)\&(c)  with the same parameters.

\smallskip
Combining (\ref{scheme0}) with (\ref{Psi}),  we can obtain the finite difference approximation to the tempered integral fractional Laplacian $(-\Dt + \lambda)^{\fl{\ap}{2}}$. 
Without loss of generality, we assume that $N_x = N$, and choose $N_y$ and $N_z$  be the smallest integer such that  $a_2 + N_yh \geq b_2$ and $a_3 + N_zh \geq b_3$. 
Define the grid points $x_i = a_1 + ih$ for $0 \le i \le N_x$,  \,$y_j = a_2 + jh$ for $0 \le j \le N_y$, and $z_k = a_3 + kh$ for $0 \le k \le N_z$. 
For the convenience of the reader, we will summarize the finite difference scheme for $d = 1,2$ and $3$ in Sec. \ref{section2-1}--\ref{section2-3}, respectively.

\subsection{One-dimensional cases}
\label{section2-1}

In one-dimensional (i.e., $d = 1$) cases,
notice the definition of $\varphi_{1, \gm}$ in (\ref{Psi}), i.e., 
$\varphi_{1, \gm}(x, \xi) = [u(x-\xi) - 2u(x) + u(x+\xi)]/\xi^\gm$, 
which can be viewed as a (weighted) central difference approximation to operator $\p_{xx}$. 
Substituting it into  (\ref{scheme0}) and rearranging terms, we obtain the finite difference approximation to the 1D tempered integral fractional Laplacian $(-\Dt + \lambda)^{\fl{\ap}{2}}$ as
 \bea\label{fLh-1D}
(-\Delta+\lambda)_{h,\gamma}^{\fl{\alpha}{2}}u_{i}  = -c_{1}^{\alpha,\lambda}
\bigg(a_{0} u_{i}+\sum_{m=0}^{i-1}a_{m}u_{i-m}
+\sum_{m=0}^{N_x-1-i} a_{m}u_{i+m}\bigg), \qquad 1 \le i \le N_x-1.\quad 
\eea
Due to its nonlocality, the  approximation of $(-\Dt + \lambda)^{\fl{\ap}{2}}$ at point $x_i$ depends on all points in domain $\Og$. 
The coefficients $a_m$, depending on the splitting parameter $\gm \in (\ap, 2]$, are given by 
\bea\label{am}
a_{m} = \left\{\begin{array}{ll} 
\displaystyle -2\lambda^{\alpha}\breve{\Gamma}(-\alpha,\lambda L) -2\sum_{1 \le n \le N_x} a_{n}, &  \mbox{if \  $m = 0$},\\
\displaystyle \fl{\lambda^{-\nu}}{2 h^\gm} \Big(2\Gamma(\nu) - \breve{\Gamma}(\nu, 2\lambda h) - \breve{\Gamma}(\nu,  \lambda h)\Big), &\mbox{if \ $m = 1$ and $\gm = 2$},\\
\displaystyle  \fl{\lambda^{-\nu}}{2m^\gm h^\gm} \Big(\breve{\Gamma}(\nu, (m-1)\lambda h) - \breve{\Gamma}(\nu,  (m+1)\lambda h)\Big), \quad &\mbox{otherwise}\\
\end{array}\right. 
\eea
with $\nu = \gamma-\ap$, and $\breve{\Gamma}(a, b)$ represents the upper incomplete Gamma function.
In the special case of $\lambda = 0$,  (\ref{fLh-1D}) reduces to the finite difference method for the 1D fractional Laplacian $(-\Dt)^{\fl{\ap}{2}}$ in \cite{Duo2018, Duo2015}. 
 
Denote the vector ${\bf u} = (u_1, u_2, \cdots, u_{N_x-1})^T$.  
We can write (\ref{fLh-1D})--(\ref{am}) into matrix-vector form of $(-\Dt + \lambda)_{h, \gm}^{\fl{\ap}{2}}{\bf u} = A_1{\bf u}$, where $A_1$ is a positive definite symmetric Toeplitz matrix with its entries satisfying
\beas
A_{i,j} = A_{i+1, j+1} = a_{|i-j|}, \qquad  1\le i, j \le N_x-1.
\eeas
Hence, the matrix-vector multiplication $A_1{\bf u}$ can be efficiently computed via the one-dimensional FFT with computational costs of 
${\mathcal O}((N_x-1)\log(N_x-1))$ and storage of ${\mathcal O}(N_x-1)$.

\subsection{Two-dimensional cases}
\label{section2-2}
In two-dimensional (i.e., $d = 2$) cases, we denote $u_{ij} = u(x_i, y_j)$. 
Setting $d = 2$ in (\ref{scheme0}) and noticing the definition of $\varphi_{2,\gm}$ in (\ref{Psi}), we obtain the finite difference approximation to the 2D tempered integral fractional Laplacian $(-\Dt + \lambda)^{\fl{\ap}{2}}$ as: for $1 \le i \le N_x-1$ and $1 \le j \le N_y-1$, 
 \bea\label{fLh-2D}
&&(-\Delta+\lambda)_{h,\gamma}^{\fl{\alpha}{2}}u_{ij}  = -c_{2}^{\alpha,\lambda}
\bigg[a_{00} u_{ij}+\sum_{m=0}^{i-1}\bigg(\sum_{\substack{n=0 \\ m+n\neq 0}}^{j-1}a_{mn}u_{(i-m)(j-n)} + \sum_{n=1}^{N_y-1-j} a_{mn} u_{(i-m)(j+n)} \bigg)\qquad \  \  \nn\\
&&\hspace{3.5cm}  +\sum_{m=0}^{N_x-1-i} \bigg(\sum_{\substack{n = 0 \\ m+n\neq 0}}^{j-1}a_{mn}u_{(i+m)(j-n)}+\sum_{n=1}^{N_y-1-j}a_{mn}u_{(i+m)(j+n)}\bigg)\bigg]. \eea
Denote $\sigma(m, n)$ as the number of zeros among integers $m$ and $n$. 
Then the coefficients $a_{mn}$ are given by: 
\beas\label{amn}
a_{mn} = \left\{\begin{array}{ll}  
\displaystyle \frac{2^{\sigma(m, n)}}{4|\bxi_{mn}|^{\gamma}}
\bigg(\int_{T_{mn}} e^{-\lambda|\bxi|}|\bxi|^{\gamma-(2+\alpha)}\,d\bxi
+ \bar{c}_{mn}\left\lfloor\fl{\gm}{2}\right\rfloor\int_{I_{00}} e^{-\lambda|\bxi|}|\bxi|^{\gamma-(2+\alpha)}\,d\bxi\bigg), & \mbox{if $m + n > 0$}, \\
\displaystyle -2\sum_{i=1}^{N} \big(a_{i0} + a_{0i}\big)
-4\sum_{i, j =1}^{N} a_{ij}
-4 \int_{D_2} e^{-\lambda|\bxi|} |\bxi|^{-(2+\ap)}\,d\xi d\eta, & \mbox{if $m=n=0$}.
\end{array}\right.
\eeas
where  the constant $\bar{c}_{01} = \bar{c}_{10} = -\bar{c}_{11} = 1$, and  $\bar{c}_{mn} \equiv 0$ for other $m, n > 0$.  

Denote the vector ${\bf u}_{x, j} = (u_{1, j}, u_{2, j},\ldots,u_{N_x-1, j})$ for $1\le j \le N_y-1$, and the block vector $
{\bf u} = ({\bf u}_{x,1}, {\bf u}_{x,2},\,\ldots,{\bf u}_{x,\,N_y-1})^T$. 
We can write the scheme (\ref{fLh-2D}) into matrix-vector form $(-\Delta + \lambda)^{\fl{\alpha}{2}}_{h,\gamma}{\bf u} = A_2{\bf u}$. 
Here, the matrix $A_2$ is a symmetric block Toeplitz matrix, defined as
\bea\label{A-2D}
{{A_2}}= 
\left(
\begin{array}{cccccc}
A_{x,0} & A_{x,1} & \ldots &  A_{x,N_y-3} & A_{x,N_y-2}  \\
A_{x,1}& A_{x,0} & A_{x,1}  &  \cdots & A_{x,N_y-3}  \\
\vdots & \ddots  & \ddots & \ddots & \vdots \\
A_{x,N_y-3} &  \ldots  & A_{x,1} & A_{x,0} & A_{x,1} \\
A_{x,N_y-2} & A_{{x},N_y-3}  & \ldots & A_{x,1} & A_{x,0}
\end{array}
\right) _{M\times M}
\eea
with $M = (N_x-1)(N_y-1)$ being the total number of unknowns, and each block $A_{x, j}$ (for $0 \le j \le N_y-2$) is a symmetric Toeplitz matrix with its entries defined as 
\beas
\big(A_{x, j}\big)_{i, k} = \big(A_{x, j}\big)_{i+1, k+1} = a_{|i-k|j}, \qquad 1 \le i, k \le N_x-1. 
\eeas
Since $A_2$ is a block-Toeplitz-Toeplitz-block matrix,  the matrix-vector multiplication $A_2{\bf u}$ can be computed efficiently via the two-dimensional FFT.

\subsection{Three-dimensional cases}
\label{section2-3}

In three-dimensional (i.e., $d = 3$) cases, we denote $u_{ijk} = u(x_i, y_j, z_k)$.  
Setting $d = 3$ in (\ref{scheme0}) and substituting $\varphi_{3,\gm}$ into it,  we obtain the finite difference approximation to the 3D tempered  integral fractional Laplacian $(-\Dt + \lambda)^{\fl{\ap}{2}}$ as: 
\begin{eqnarray}\label{fLh-3D}
&&(-\Delta+\lambda)_{h,\gamma}^{\frac{\alpha}{2}}u_{ijk}
= -c_{3}^{\alpha,\lambda}\bigg[a_{000}\,u_{ijk} + 
\sum_{p=0, 1}\bigg(\sum_{m\in S_i^p}a_{m00}\,u_{[i+(-1)^pm]jk} 
+ \sum_{n\in S_j^p}a_{0n0}u_{i[j+(-1)^pn]k} \nn\\
&&\hspace{1.6cm} + \sum_{s\in S_k^p}a_{00s}u_{ij[k+(-1)^ps]} \bigg)
+\sum_{p, q=0, 1}\bigg(\sum_{s\in S_k^q}\sum_{n\in S_j^p}a_{0ns}u_{i[j+(-1)^pn][k+(-1)^qs]}  \nn\\
&&\hspace{1.6cm}+ \sum_{s\in S_k^q}\sum_{m\in S_i^p}a_{m0s}u_{[i+(-1)^pm]j[k+(-1)^qs]}
+ \sum_{n\in S_j^q}\sum_{m\in S_i^p}a_{mn0}u_{[i+(-1)^pm][j+(-1)^qn]\,k} \bigg) \nn\\
&&\hspace{1.6cm}+ \sum_{p,q,r=0}^{1}\sum_{s\in S_k^r}\sum_{n\in S_j^q}\sum_{m\in S_i^p}a_{mns}u_{[i+(-1)^pm][j+(-1)^qn][k+(-1)^rs]} \bigg], 
\end{eqnarray}
for $1 \le i \le N_x -1$, $1 \le j \le N_y -1$, and $1 \le k \le N_z-1$, where the index sets 
\beas
&&S_i^p = \big\{l\,|\, l\in{\mathbb N}, \  1 \le i + (-1)^pl \le N_x-1\big\},\\
&&S_j^p = \big\{l\,|\, l\in{\mathbb N}, \  1 \le j + (-1)^pl \le N_y-1\big\}, \\
&&S_k^p = \big\{l\,|\, l\in{\mathbb N}, \  1 \le k + (-1)^pl \le N_z-1\big\},\qquad p = 0, \,\mbox{or} \ 1.
\eeas
The coefficients $a_{mns}$ are given by: 
\beas
a_{mns} = \left\{\begin{array}{ll}
\displaystyle -2\sum_{i=1}^{N}\big(a_{i00} + a_{0i0} + a_{00i}\big) 
-4\sum_{i,j=1}^{N}\big(a_{0ij} + a_{i0j} + a_{ij0}\big) & \\
\displaystyle \hspace{4.5cm}- 8\sum_{i,j,k=1}^{N}a_{ijk} -8\int_{D_3} e^{-\lambda|\bxi|}|\bxi|^{-(3+\ap)}\,d\bxi, & \mbox{if $ m=n=s=0$},\\
\displaystyle \frac{2^{\sigma(m, n, s)}}{8|\bxi_{mns}|^{\gamma}}\bigg(\int_{T_{mns}} e^{-\lambda|\bxi|}|\bxi|^{\gamma-(3+\alpha)}\,d\bxi - \bar{c}_{mns}\left\lfloor \fl{\gm}{2}\right\rfloor\int_{I_{000}}e^{-\lambda|\bxi|} |\bxi|^{\gamma-(3+\alpha)}\,d\bxi\bigg), & \mbox{otherwise},
\end{array}\right.
\eeas
where $\sigma(m, n, s)$ denotes the number zeros among integers $m, n$ and $s$, and $N = \max\{N_x,  N_y, N_z\}$. 
For $m,n,s \leq 1$, the constant $\bar{c}_{mns} = -\fl{5}{3}$ if $\sigma(m, n, s) = 2$;  $\bar{c}_{mns} = 1$ if $\sigma(m, n, s) = 0$ or $1$. In other cases, i.e., if one of $m,n,s>1$, $\bar{c}_{mns} = 0$.

Denote the vector ${\bf u}_{x, j, k} = \big({u}_{1jk},  \,\ldots,\, u_{(N_x-1)jk}\big)$,  the block vector ${\bf u}_{x,y,k} = \big({\bf u}_{x,1,k},  \,\ldots,\,{\bf u}_{x, N_y-1, k}\big)$, and then ${\bf u} = \big({\bf u}_{x,y,1},   \, \ldots,\, {\bf u}_{x,y,N_z-1}\big)^T$. 
The matrix-vector form of (\ref{fLh-3D}) is given by $(-\Dt + \lambda)_{h, \gm}^{\fl{\ap}{2}}{\bf u} = A_3 {\bf u}$, where the matrix $A_3$ is defined as: 
\begin{eqnarray*}\label{A-3D}
{{{A}_3}}= 
\left(
\begin{array}{cccccc}
A_{x,y,0} & A_{x,y,1} & \ldots &  A_{x,y,N_z-3} & A_{x,y,N_z-2}  \\
A_{x,y,1}& A_{x,y,0} & A_{x,y,1}  &  \cdots & A_{x,y,N_z-3}  \\
\vdots & \ddots  & \ddots & \ddots & \vdots \\
A_{x,y,N_z-3} &  \ldots  & A_{x,y,1} & A_{x,y,0} & A_{x,y,1} \\
A_{x,y, N_z-2} & A_{x,y,N_z-3}  & \ldots & A_{x,y,1} & A_{x,y,0}
\end{array}
\right).
\end{eqnarray*}
For $k = 0,1,\dots,N_z-2$, the block matrix
\begin{eqnarray*}\label{Ak}
{{A}}_{x,y,k} = 
\left(
\begin{array}{cccccc}
A_{x,0,k} & A_{x,1,k} & \ldots &  A_{x,N_y-3,k} & A_{x,N_y-2,k}  \\
A_{x,1,k} & A_{x,0,k} & A_{x,1,k}  &  \cdots & A_{x,N_y-3,k}  \\
\vdots & \ddots  & \ddots & \ddots & \vdots \\
A_{x,N_y-3,k} &  \ldots  & A_{x,1,k} & A_{x,0,k} & A_{x,1,k} \\
A_{x,N_y-2,k} & A_{x,N_y-3,k}  & \ldots & A_{x,1,k} & A_{x,0,k}
\end{array}
\right), \nn
\end{eqnarray*}
and each block $A_{x, j, k}$  is a symmetric Toeplitz matrix with its entries defined as 
\beas
\big(A_{x, j, k}\big)_{i, l} = \big(A_{x, j}\big)_{i+1, l+1} = a_{|i-l|jk}, \qquad 1 \le i, l \le N_x-1. 
\eeas
%
Similarly, the matrix-vector product $A_3\bu$ can be efficiently computed via the three-dimensional FFT. 


\section{Error analysis}
\setcounter{equation}{0}
\label{section3}

In this section, we study local truncation errors of our method in approximating the tempered fractional Laplacian $(-\Dt + \lambda)^{\fl{\ap}{2}}$. 
Without loss of generality, we will provide detailed error estimates for the one-dimensional cases, and the generalization to two and three dimensions can be done  by following similar lines. 
For  $k \in {\mathbb N}$ and $0 \le s \le 1$, we denote $C^{k, s}({\mathbb R^d})$ as the space that consists of function $u: {\mathbb R}^d \to {\mathbb R}$ with continuous derivatives of order less or equal to $k$, and their $k$-th (partial) derivatives are uniformly H\"older continuous with exponent $s$. 
Let $w: [a, b] \to {\mathbb R}$ be a non-negative integrable function. 
We define
\begin{eqnarray}\label{Tht0}
\Theta_{[a,\,b]}^{(m)}(x) =
\int_{a}^{x}w(\xi)\,\frac{(x-{\xi})^{m}}{m!}\,d{\xi}
+\int_{b}^{x}w(\xi)\,\frac{(x-{\xi})^{m}}{m!}\,d{\xi}, \quad \ \, x \in [a,\,b], \quad m \in {\mathbb N}, 
\end{eqnarray}
which can be viewed as an extension of the generalized Peano kernel function. 
It is easy to show that for any $m \in {\mathbb N}$, there is
\bea\label{prop-tht}
\big|\Theta_{[a, b]}^{(m)}(x)\big| \le C(b-a)^m \int_a^b \big|w(\xi)\big| d\xi,\qquad x\in[a, b].
\eea
The main technique used in our proof is the extension of weighted Montgomery identity \cite{Khan2013, Duo2018}. 
For the convenience of the reader, we will review it in the following lemma. 

\begin{lemma}[Extension of the weighted Montgomery identity \cite{Duo2018}]\label{Montgomery}
Let $w, f: [a,\,b]\rightarrow \mathbb{R}$ be integrable functions.
If  the derivative $f^{(n)}$ exists and integrable for  $n \in {\mathbb N}$,  we have
\begin{eqnarray}
&&\int_{a}^{b} \Big(2f(x)- f(a) - f(b) \Big)w(x)\,dx  = (-1)^n \int_a^b \Theta_{[a,\,b]}^{(n-1)}(x)f^{(n)}(x)\,dx\nn\\
&& \hspace{6cm}+\sum_{k=2}^{n}(-1)^{k-1}
\left(\Theta_{[a,\,b]}^{(k-1)}(b)f^{(k-1)}(b)-\Theta_{[a,\,b]}^{(k-1)}(a)f^{(k-1)}(a)\right). \qquad\qquad\qquad\qquad  \nn
\end{eqnarray}
\end{lemma}
If $w \equiv 1$, Lemma \ref{Montgomery} gives the error estimates of the conventional trapezoidal rule. 
Next, we will present some often used properties of the central difference quotient $\varphi_{1,\gm}$. 
For notational simplicity, we will omit ${ x}$ and let $\varphi_{1, \gamma}(\xi) := \varphi_{1, \gamma}(x, \xi)$, and also denote $\varphi^{(n)}_{1, \gm}(\xi) := \p_\xi^n\varphi_{1, \gm}(x, \xi)$.
\begin{lemma}\label{lemma1-2D}
Let $0 < s \le 1$ and $\xi \in {\mathbb R}^{+}$.
\begin{itemize}\itemsep -0pt
\item[(i)]  If $u \in C^{m, s}(\mathbb{R})$ for $m = 0, 1$,  there exists a constant $C > 0$, such that
$\big|\varphi_{1, 0}(\xi)\big| \leq C\xi^{s+m}$. 
\item[(ii)] If $u \in C^{1, s}(\mathbb{R})$,  then the derivative $\varphi_{1, \gm}'$ exists.  Moreover,  there is a constant $C > 0$, such that 
$\big|\varphi'_{1, \gamma}(\xi)\big| \leq C{\xi}^{s-\gamma}$ for  any $\gm \in (\ap, 2]$.
\item[(iii)] If $u \in C^{m, s}(\mathbb{R})$ for $m = 2, 3$, the derivatives $\varphi'_{1,2}$, $\varphi''_{1,2}$ exist. Moreover, there exists a constant $C > 0$, such that
$\big|\varphi'_{1, 2}(\xi)\big|\leq C{\xi}^{s-(3-m)}$ and
$\big|\varphi''_{1,2}(\xi)\big|\leq C{\xi}^{s-(4-m)}$.
\end{itemize}
\end{lemma}
The proof of Lemma \ref{lemma1-2D} can be done by directly applying Taylor's theorem.
In the following,  we will present our main results in Theorems \ref{thm1} and \ref{thm2} for $\ap \in (0, 1)$ and $\ap \in [1, 2)$, respectively.  
Define the local truncation error as:
\bea\label{localerror}
e_{\alpha,\gamma}^{h}(\bx) = (-\Dt+\lambda)^{\fl{\ap}{2}} u(\bx) - (-\Dt+\lambda)^{\fl{\ap}{2}}_{h,\gamma} u(\bx),\qquad \mbox{for} \ \ \bx\in\Og.
\eea
\begin{theorem}[{\bf Error estimates for $0 < \ap < 1$}]\label{thm1}
Let $(-\Dt + \lambda)_{h,\gamma}^{\fl{\ap}{2}}$ be the finite difference approximation of the operator $(-\Dt + \lambda)^{\fl{\ap}{2}}$, with $h$ a small mesh size.  
Suppose that  $u$ has finite support on the domain $\Og \subset {\mathbb R}^d$. 
For $0 < \veps \le 1-\ap$, there exists a  constant $C > 0$ independent of $h$, such that
\begin{enumerate}\itemsep -1pt
\item[](i) if $u \in C^{0,\,\ap + \veps}(\bar{\Og})$ and $\gm \in (\ap, 2]$, the local truncation error satisfies $\big\|e_{\alpha,\gamma}^{h}(\bx)\big\|_\infty \le Ch^\veps$. 
\item[](ii) if $u \in C^{2,\,\ap + \veps}(\bar{\Og})$ and $\gm = 2$, the local truncation error satisfies $\big\|e_{\alpha, 2}^{h}(\bx)\big\|_\infty \le Ch^2$.
\end{enumerate}
\end{theorem}
\begin{proof}
Here, we will focus on the proof for $d = 1$.  
From (\ref{split1})--(\ref{I-2D}) and (\ref{I00-final}) with $d = 1$, we obtain
\bea\label{error}
e_{\alpha,\gamma}^{h}(x) &=&-\fl{c_{1}^{\alpha,\lambda}}{2}
\bigg(\int_{0}^{\xi_1}\Big(2\varphi_{1, \gamma}(\xi) 
- c_{1}^\gm\varphi_{1, \gm}(\xi_{1})\Big)
w_{\lambda,\gm}(\xi)d\xi \nn\\
&&+ \sum_{i=1}^{N-1} \int_{\xi_{i}}^{\xi_{i+1}} \Big( 2\varphi_{1, \gamma}(\xi) - \varphi_{1, \gamma}\big(\xi_{i}\big)- \varphi_{1, \gamma}\big(\xi_{i+1}\big)\Big) w_{\lambda,\gm}(\xi) d\xi \bigg] = -\fl{c_1^{\alpha,\lambda}}{2} (I+II), \quad 
\eea
where  constant $c_1^{\gamma}$ is defined after (\ref{I00-final}). 
We will then prove Cases (i) and (ii) separately. 
\bigskip 

\noindent{\bf Case (i)  (For $u\in C^{0,\,\ap+\veps}({\mathbb R})$): }  
Formally, we can rewrite  $\varphi_{1, \gamma}(\xi) = \xi^{-\gm}\varphi_{1, 0}(\xi)$. 
Using the triangle inequality and Lemma \ref{lemma1-2D} (i) with $m = 0$ and $s = \ap+\veps$, we  obtain 
\bea\label{thm11-termI}
|\,I \, | &=& \bigg|\int_0^h
\Big(2\varphi_{1, 0}(\xi)\xi^{-(1+\alpha)} - c_{1}^\gm\,\varphi_{1, 0}(\xi_{1})\,\xi_{1}^{-\gm} \xi^{\gamma-(1+\alpha)}\Big)e^{-\lambda \xi}\,d\xi \bigg|\nn\\
&\le& C\bigg(\int_0^h\big|\varphi_{1, 0}(\xi)\big| \xi^{-(1+\alpha)}\,d\xi + h^{-\gm}\int_0^h \big|\varphi_{1, 0}(\xi_{1})\big|\xi^{\gamma-(1+\alpha)}d\xi \bigg)\quad \nn\\
&\le& C\bigg(\int_0^h \xi^{\veps-1}d\xi + h^{\ap+\veps-\gm}\int_0^h \xi^{\gm-(1+\ap)}d\xi \bigg)  \le \, Ch^{\varepsilon}.
\eea
For term $II$, we first add and subtract {$\big(\varphi_{1, 0}(\xi_i) + \varphi_{1, 0}(\xi_{i+1})\big)\xi^{-(1+\ap)}e^{-\lambda\xi}$} and then use the triangle inequality, Taylor's  theorem, and Lemma \ref{lemma1-2D} (i) with $m = 0$ and  $s = a+\veps$, to obtain
\beas
|\,II\,| &\leq& \sum_{i=1}^{N-1} \int_{\xi_i}^{\xi_{i+1}} \big| 2\varphi_{1, 0}(\xi) - \varphi_{1, 0}\big(\xi_{i}\big)- \varphi_{1, 0}\big(\xi_{i+1}\big)\big| \xi^{-(1+\alpha)} d\xi\nn\\
&&+\sum_{i=1}^{N-1} \int_{\xi_i}^{\xi_{i+1}} \sum_{m = 0,1}\big|\xi^{-\gm}-\xi_{i+m}^{-\gm}\big| \big|\varphi_{1, 0}\big(\xi_{i+m}\big)\big|\xi^{\gamma-(1+\alpha)} d\xi \nn\\
&\le& C\bigg( \sum_{i=1}^{N-1}\int_{\xi_i}^{\xi_{i+1}} \sum_{m = 0, 1} \Big(|u(x-\xi) - u(x-{\xi_{i+m}})| + |u(x+\xi) - u(x+{\xi_{i+m}})|\Big)\xi^{-(1+\ap)} d\xi\nn\\
&&+ \sum_{i=1}^{N-1}\int_{\xi_i}^{\xi_{i+1}}  h \Big(\max_{\eta\in[\xi_i, \xi]} \eta^{-(\gm+1)} +  \max_{\zeta\in[\xi, \xi_{i+1}]} \zeta^{-(\gm+1)}\Big) |\xi_{i+m}|^{\ap+\veps}\xi^{\gamma-(1+\alpha)} d\xi \bigg) \nn\\
&\leq& C\bigg(\sum_{i=1}^{N-1} \int_{\xi_i}^{\xi_{i+1}}  h^{\alpha+\varepsilon} \xi^{-(1+\alpha)}\,d\xi 
+ \sum_{i=1}^{N-1} \int_{\xi_i}^{\xi_{i+1}} h \big(\xi^{-(\gm+1)}\big)\big(\xi^{\alpha+\varepsilon}\big)\big(\xi^{\gamma-(1+\alpha)}\big)\, d\xi \bigg)\qquad\qquad\qquad \nn\\
\label{thm11-termII}
&=& C\bigg(h^{\ap + \veps}\int_h^L \xi^{-(1+\ap)}d\xi  + h\int_h^L \xi^{\veps - 2}d\xi \bigg)\, \le  \, Ch^{\varepsilon}, \qquad\qquad\qquad\qquad\qquad\qquad\qquad\ \ 
\eeas
where we use the fact that for $i \ge 1$, if $\xi \in [\xi_i, \xi_{i+1}]$, then $\xi_{i+1} \le 2\xi_i \le 2\xi$. 
Combining (\ref{error})--(\ref{thm11-termII}) leads to $\big\|e_{\ap, \gm}^h(x)\big\|_\infty \le Ch^\veps$ immediately. \\

\noindent{\bf Case (ii) (For $u\in C^{2,\,\ap+\veps}({\mathbb R})$): }  
Setting $\gm = 2$ and using Lemma \ref{Montgomery} with $n = 2$ to (\ref{error}), we get 
\begin{eqnarray} \label{Error2}
&&e_{\alpha, 2}^{h}(x) = -\fl{c_{1}^{\alpha,\lambda}}{2}
\bigg(2 \int_{0}^{\xi_1} 
\Big(\varphi_{1, 2}(\xi)-\varphi_{1, 2}(\xi_1)\Big)e^{-\lambda\xi}\xi^{1-\alpha}\,d\xi +\sum_{i=1}^{N-1} \int_{\xi_i}^{\xi_{i+1}} \Theta_{[\xi_i,\xi_{i+1}]}^{(1)} (\xi)\varphi''_{1, 2}(\xi) \,d\xi\nn\\
&&\hspace{1cm}-\sum_{i=1}^{N-1}
\Big(\Theta_{[\xi_i,\xi_{i+1}]}^{(1)}(\xi_{i+1})\varphi'_{1, 2}(\xi_{i+1})-\Theta_{[\xi_i,\xi_{i+1}]}^{(1)}(\xi_{i})\varphi'_{1, 2}(\xi_{i})\Big)\bigg) =-\fl{c_{1}^{\alpha,\lambda}}{2}\big(\bar{I} + \bar{II} +\bar{III}\big).\qquad \qquad 
\end{eqnarray}
By Taylor's theorem and Lemma \ref{lemma1-2D} (iii) with $m = 2$ and $s = \ap+\veps$, we obtain
\begin{eqnarray}\label{term1}
|\,\bar{I}\,| =  2\,\bigg|\int_0^h\bigg(\int_h^{\xi}\varphi'_{1, 2}(\widetilde{\xi})\,d\widetilde{\xi}\bigg)e^{-\lambda\xi} \xi^{1-\alpha}d\xi \bigg| &\le& C\int_0^h\bigg(\int_{\xi}^h|\varphi'_{1, 2}(\widetilde{\xi})|\,d\widetilde{\xi}\bigg)\xi^{1-\alpha}d\xi \nn\\
&\leq& C \bigg(\int_0^{h}\widetilde{\xi}^{\alpha+\varepsilon-1}\,d\widetilde{\xi}\bigg)\bigg(\int_0^h\xi^{1-\alpha}d\xi\bigg) \leq  Ch^{2+\varepsilon}.\qquad\quad
\eea

For term $\bar{II}$, using the triangle inequality,  property (\ref{prop-tht}),  and Lemma \ref{lemma1-2D} (iii) with $m = 2$ and $s = \alpha+\veps$, we obtain 
\begin{eqnarray}\label{term2}
|\,\bar{II}\,| &\leq&   \sum_{i = 1}^{N-1} \int_{\xi_i}^{\xi_{i+1}} \big|\Theta_{[\xi_i,\xi_{i+1}]}^{(1)} (\xi)\big|\,\big|\varphi''_{1, 2}(\xi)\big| d\xi \le  C\sum_{i=1}^{N-1}\bigg(h\int_{\xi_{i}}^{\xi_{i+1}} e^{-\lambda \xi} \xi ^{1-\alpha} d\xi  \bigg)
\bigg(\int_{\xi_{i}}^{\xi_{i+1}}  \xi^{\ap+\veps-2} \,d\xi \bigg) \nn\\
&\leq & Ch\sum_{i=1}^{N-1}\bigg( \int_{\xi_{i}}^{\xi_{i+1}}  \xi_{i+1}^{-1}\xi ^{1-\alpha} d\xi  \bigg)
\bigg(\int_{\xi_{i}}^{\xi_{i+1}}  \xi_{i+1}\xi^{\ap+\veps-2} \,d\xi \bigg) \nn\\
&\leq & Ch\sum_{i=1}^{N-1}\bigg( \int_{\xi_{i}}^{\xi_{i+1}}\xi ^{-\alpha} d\xi  \bigg)
\bigg(\int_{\xi_{i}}^{\xi_{i+1}} \xi^{\ap+\veps-1} \,d\xi \bigg) \leq  Ch^2\sum_{i=1}^{N-1}\int_{\xi_{i}}^{\xi_{i+1}}  \xi^{\veps-1} d\xi \,
\le \, C h^{2}, 
\end{eqnarray}
where the third inequality is obtained by using the same property as  in obtaining (\ref{thm11-termII}), and the {second inequality from the end} is by the Chebyshev integral inequality. 

For term $\bar{III}$, we first rewrite it as
\beas
&&\bar{III}=-\sum_{i=2}^{N-1} \Big(\Theta_{[\xi_{i-1}, \xi_i]}^{(1)}(\xi_{i})
-\Theta_{[\xi_{i}, \xi_{i+1}]}^{(1)}(\xi_{i})\Big)\varphi'_{1, 2}(\xi_{i})
+\Theta_{[\xi_{1}, \xi_{2}]}^{(1)}(\xi_{1})\varphi'_{1, 2}(\xi_{1})
-\Theta_{[\xi_{N-1},\xi_{N}]}^{(1)}(\xi_{N})\varphi'_{1, 2}(\xi_{N})\qquad\qquad\nn\\
&&\hspace{0.7cm} =  \bar{III}_1 + \bar{III}_2 + \bar{III}_3.
\eeas
To estimate term $\bar{III}_1$, we define an auxiliary function
$G(x) =\int_{x}^{\xi_{i}} e^{-\lambda\xi}\, \xi^{1-\alpha}(\xi_i - {\xi})\,d\xi$.
Noticing the definitions of $\Theta$ and $G$, and by Taylor's theorem, we  obtain 
\begin{eqnarray}\label{eq100}
\big|\Theta_{[\xi_{i-1},\xi_i]}^{(1)}(\xi_{i})
-\Theta_{[\xi_{i},\xi_{i+1}]}^{(1)}(\xi_{i})\big| 
= \big|G(\xi_{i-1}) - G(\xi_{i+1})\big| 
\leq Ch^{3}\max_{\widetilde{\xi}\in[\xi_{i-1},\xi_{i+1}]}|G'''(\widetilde{\xi})| \le Ch^3\xi_i^{-\ap}, \quad 
\end{eqnarray}
where we use the fact that $G(\xi_i) = G'(\xi_i) = 0$ and $\max_{\widetilde{\xi}\in[\xi_{i-1},\xi_{i+1}]}|G'''(\widetilde{\xi})| \leq C \xi_i^{-\alpha}$. 
Using the triangle inequality, Lemma \ref{lemma1-2D} (iii) with $m = 2$ and $s = \ap + \veps$, and (\ref{eq100}), we obtain      
\beas\label{term3-1}
|\, \bar{III}_1 \,| 
&\le&\sum_{i=2}^{N-1} \big|\Theta_{[\xi_{i-1},\xi_i]}^{(1)}(\xi_{i})
-\Theta_{[\xi_{i},\xi_{i+1}]}^{(1)}(\xi_{i})\big|\big|\varphi'_{1, 2}(\xi_{i})\big|\nn\\
&\le& C h^{3} \sum_{i=2}^{N-1}\xi_i^{-\alpha} \xi_i^{\alpha+\veps-1} 
\le  Ch^2 \int_h^L \xi^{\veps-1}d\xi  \le Ch^{2+\veps}.\qquad \qquad \qquad\qquad  
\eeas
Using Lemma \ref{lemma1-2D} (iii) with $m = 2$ and $s = \ap+\veps$ and the property (\ref{prop-tht}) to term $\bar{III}_2$, we obtain
\beas\label{term3-2}
|\, \bar{III}_2 \,|
= \big|\varphi'_{1, 2}(\xi_{1})\big|  \big|\Theta_{[\xi_{1},\xi_{2}]}^{(1)}(\xi_{1})\big| \le  C h^{\ap + \veps -1} \bigg(h \int_{\xi_1}^{\xi_2}e^{-\lambda \xi}\xi^{1-\alpha}\,d\xi\,\bigg) \le Ch^{\alpha+\veps} \Big(\xi_2^{2-\alpha}-\xi_1^{2-\alpha}\Big)
   \le Ch^{2+\veps}.\qquad \qquad \qquad\qquad 
\eeas
Following similar lines and using Taylor's theorem, we get 
\beas\label{term3-3}
|\, \bar{III}_3 \,| =  \big|\varphi'_{1, 2}(\xi_{N})\big| \big|\Theta_{[\xi_{N-1},\xi_{N}]}^{(1)}(\xi_{N})\big|
\leq C \xi_N^{\ap +\veps -1}\bigg(h\int_{\xi_{N-1}}^{\xi_{N}}e^{-\lambda \xi}\xi^{1-\alpha}\,d\xi\bigg) \le Ch \,\xi_{N}^{\alpha+\veps-1} \Big(\xi_{N}^{2-\alpha}-\xi_{N-1}^{2-\alpha}\Big)
   \le Ch^{2}.\qquad\qquad\qquad
\eeas
Combining the above estimates on $\bar{III}_1$, $\bar{III}_2$ and $\bar{III}_3$, we obtain $|\, \bar{III} \,| \le Ch^{2}$.
Hence, the error bounds in Case (ii) is obtained by combining the estimates on terms $\bar{I}$, $\bar{II}$ and $\bar{III}$. 
\end{proof}
Theorem \ref{thm1} indicates that for $\ap \in (0, 1)$, our method is consistent if $u \in C^{0, \ap+\veps}(\bar{\Og})$ with small $\veps > 0$, independent of the splitting parameter $\gm$.  
Furthermore, if choosing $\gm = 2$, our method has the second-order accuracy for $u \in C^{2, \ap+\veps}(\bar{\Og})$. 
The above conclusions hold for any dimension $d \ge 1$ and $\lambda \ge 0$.  
Note that if $\lambda = 0$, Theorem \ref{thm1} (ii) improves the results in \cite{Duo2018} for the fractional Laplacian $(-\Dt)^{\fl{\ap}{2}}$, i.e., proving the second-order accuracy with a much less regularity requirement. 
 
\begin{theorem}[{\bf Error estimates for $1 \le \ap < 2$}]\label{thm2}
Let $(-\Dt + \lambda)_{h,\gamma}^{\fl{\ap}{2}}$ be the finite difference approximation of the operator $(-\Dt + \lambda)^{\fl{\ap}{2}}$, with $h$ a small mesh size.  
Suppose that  $u$ has finite support on the domain $\Og \subset {\mathbb R}^d$. 
For $0 < \veps  {\, \leq \,} 2-\ap$, there exists a constant $C > 0$ independent of $h$, such that 
\begin{enumerate}\itemsep -1pt
\item[](i) If $u \in C^{1,\,\ap -1 + \veps}(\bar{\Og})$ and $\gm\in(\ap, 2]$, the local truncation error 
satisfies $\big\|e_{\ap, \gm}^h(\bx)\big\|_\infty \le Ch^\veps$. 
\item[](ii) If $u \in C^{3,\,\ap -1 + \veps}(\bar{\Og})$ and $\gm =  2$, the local truncation error 
satisfies $\big\|e_{\ap, 2}^h(\bx)\big\|_\infty \le Ch^2$.
\end{enumerate}
\end{theorem} 

\begin{proof} 
Again, we will  focus on the proof for $d = 1$ and divide our discussion into two parts. \vskip 6pt

\noindent{\bf Case (i) (For $u\in C^{1,\,\ap-1+\veps}({\mathbb R})$): } 
Using Lemma \ref{Montgomery} with $n = 1$ to the error function (\ref{error}), we  get 
\bea\label{error1}
e_{\alpha,\gamma}^{h}(x) &=&-\fl{c_{1}^{\alpha,\lambda}}{2}
\bigg(\int_{0}^{\xi_1}\Big(2\varphi_{1, \gamma}(\xi) 
- c_{1}^\gm\varphi_{1, \gm}(\xi_{1})\Big)
w_{\lambda,\gm}(\xi)d\xi -\sum _{i=1}^{N-1}
\int_{\xi_i}^{\xi_{i+1}} \Theta_{[\xi_i,\xi_{i+1}]}^{(0)} (\xi)\, \varphi'_{1, \gamma}(\xi) \,d\xi \bigg)\nn\\
&=& -\fl{c_{1}^{\alpha,\lambda}}{2} \big(I + II\big).
\eea
For term $I$, using the triangle inequality and then Lemma \ref{lemma1-2D} (i) with $m = 1$ and $s = \ap-1+\veps$ yields 
\bea\label{termI}
|\,I \, | &=& \bigg|\int_0^h
\Big(2\varphi_{1, 0}(\xi)\xi^{-(1+\alpha)} - c_{1}^\gm\,\varphi_{1, 0}(\xi_{1})\,\xi_{1}^{-\gm} \xi^{\gamma-(1+\alpha)}\Big)e^{-\lambda \xi}\,d\xi \bigg|\nn\\
&\le& C\bigg(\int_0^h\big|\varphi_{1, 0}(\xi)\big| \xi^{-(1+\alpha)}\,d\xi + \int_0^h \big|\varphi_{1, 0}(\xi_{1})\big|\,\xi_{1}^{-\gm} \xi^{\gamma-(1+\alpha)}d\xi \bigg)\quad \nn\\
&\le& C\bigg(\int_0^h \xi^{\ap+\veps}\xi^{-(1+\ap)}d\xi + h^{\ap+\veps-\gm}\int_0^h \xi^{\gm-(1+\ap)}d\xi \bigg)  \le \, Ch^{\varepsilon}.\qquad\qquad\qquad 
\eea
By triangle inequality, Chebyshev integral inequality  and Lemma \ref{lemma1-2D} (ii) with  $s = \ap-1+\veps$, we get
\bea\label{termII}
|\,II\,| \le \sum _{i=1}^{N-1}
\int_{\xi_i}^{\xi_{i+1}} \big|\Theta_{[\xi_i,\xi_{i+1}]}^{(0)} (\xi)\big| \big|\varphi'_{1, \gamma}(\xi)\big| d\xi &\le& C\sum _{i=1}^{N-1}
\bigg(\int_{\xi_{i}}^{\xi_{i+1}} \xi ^{\gamma-(1+\ap)}\,d\xi 
\bigg)
\bigg(\int_{\xi_{i}}^{\xi_{i+1}} \xi^{\ap+\veps-1-\gamma}\,d\xi \bigg)\nn\\
&\leq& Ch\sum _{i=1}^{N-1}
\int_{\xi_{i}}^{\xi_{i+1}} \xi ^{\veps-2} \,d\xi \ \le \ Ch^{\veps}.
\eea
Combining (\ref{error1})--(\ref{termII}) yields the error estimate in Case (i).\\

\noindent{\bf Case (ii)  (For $u\in C^{3,\,\ap-1+\veps}({\mathbb R})$): } 
Starting from  (\ref{Error2}), we can obtain the same estimates for terms $\bar{I}$ and $\bar{III}$, where instead Lemma \ref{lemma1-2D} {(iii)} with $m = 3$ and $s = \ap + \veps - 1$ is used. 
While for term $\bar{II}$, we use the triangle inequality, property (\ref{prop-tht}), Lemma \ref{lemma1-2D} {(iii)} with $m = 3$ and $s = \ap+\veps-1$, and Chebyshev integral inequality to obtain
\bea
|\,\bar{II}\,| &\leq&   \sum_{i = 1}^{N-1} \int_{\xi_i}^{\xi_{i+1}} \big|\Theta_{[\xi_i,\xi_{i+1}]}^{(1)} (\xi)\big|\,\big|\varphi''_{1, 2}(\xi)\big| d\xi \le  C\sum_{i=1}^{N-1}\bigg(h\int_{\xi_{i}}^{\xi_{i+1}} e^{-\lambda \xi} \xi ^{1-\alpha} d\xi  \bigg)
\bigg(\int_{\xi_{i}}^{\xi_{i+1}}  \xi^{\ap+\veps-2} \,d\xi \bigg) \nn\\
&\le&Ch\sum_{i=1}^{N-1}\bigg(
\int_{\xi_{i}}^{\xi_{i+1}} \xi ^{1-\alpha} d\xi  \bigg)
\bigg(\int_{\xi_{i}}^{\xi_{i+1}}  \xi^{\ap+\veps-2} \,d\xi \bigg) \le Ch^2\sum_{i=1}^{N-1}\int_{\xi_{i}}^{\xi_{i+1}}  \xi^{\veps-1} d\xi \,
\le \, C h^{2}.
\end{eqnarray}
Hence, Case (ii) is proved immediately.
\end{proof}
Comparing Theorems \ref{thm1} and \ref{thm2}, it shows that to obtain the same error estimates, the smoothness requirements of function $u$ is higher, if $\ap \ge 1$.  
Choosing $\veps = 1+\lfloor\ap\rfloor - \ap$ in Theorems \ref{thm1} (i) and \ref{thm2} (i) immediately leads to the following corollary: 

\begin{corollary}
Suppose that $u \in C^{\lfloor\ap\rfloor, 1}(\bar{\Og})$ has finite support on domain $\Og \in {\mathbb R}^d$.  For any $\gm \in (\ap, 2]$, the local truncation error satisfies $\big\|e_{\ap, \gm}^h\big\|_\infty \le Ch^{1+\lfloor\ap\rfloor - \ap}$.
\end{corollary}

\begin{figure}[htb!]
\centerline{
(a)\includegraphics[height=4.660cm,width=6.46cm]
{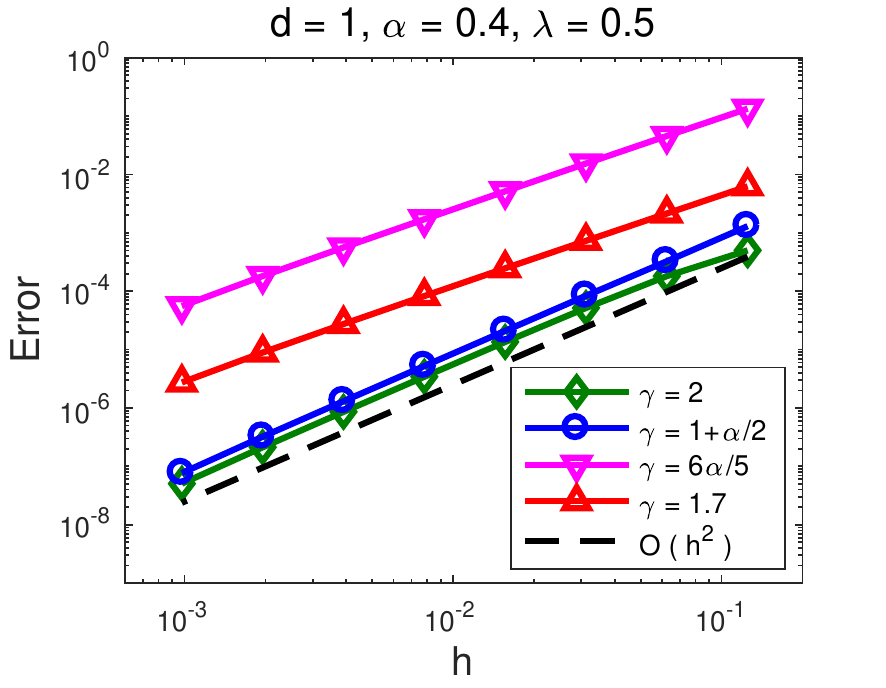}
(b)\includegraphics[height=4.660cm,width=6.46cm]{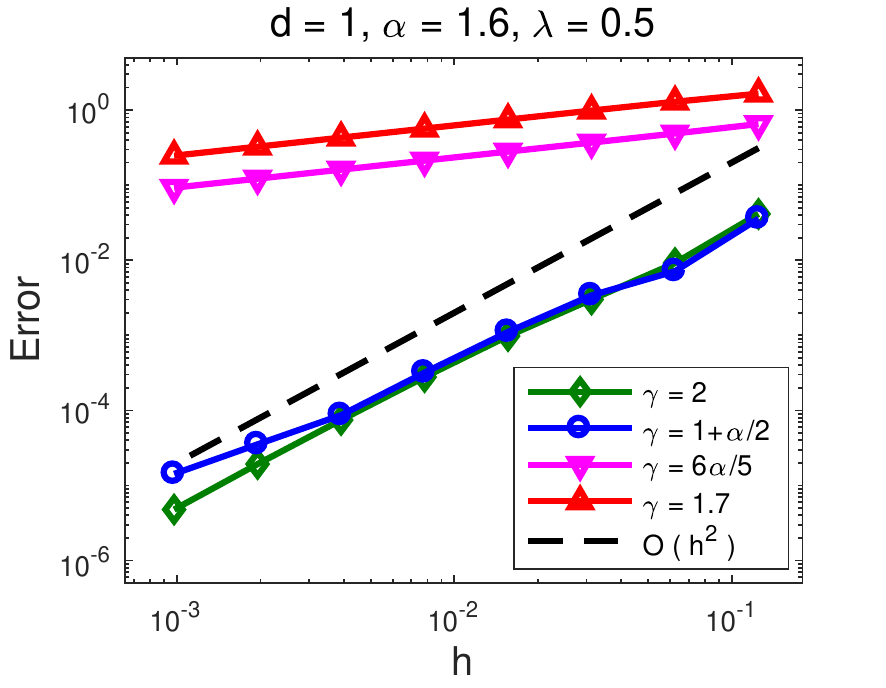}}
\vspace{-1mm}
\centerline{(c) \includegraphics[height=4.660cm,width=6.46cm]{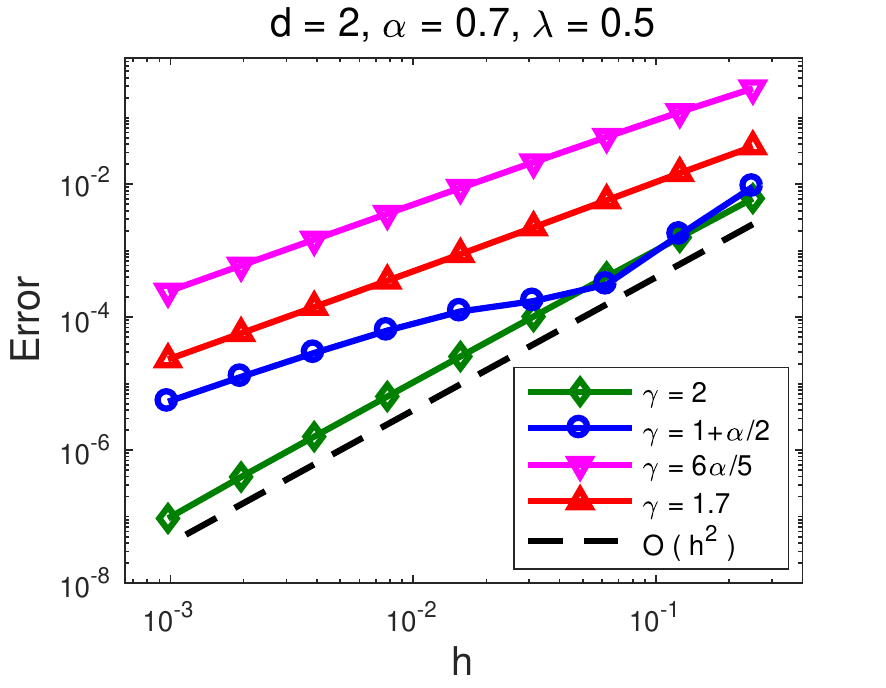}
(d)\includegraphics[height=4.660cm,width=6.46cm]{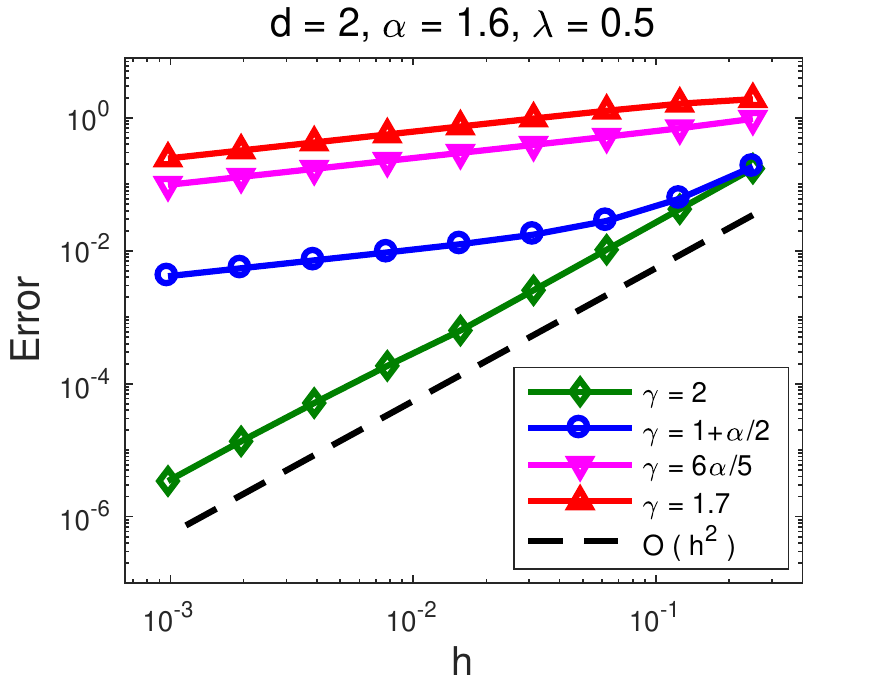}}
\caption{Illustration of the effect of splitting parameter $\gm$, where $u \in C^{2+\lfloor\ap\rfloor, 1}(\bar{\Og})$ and $\lambda = 0.5$.}\label{Fig3}
\end{figure} 
\begin{remark}[Effect of splitting parameter $\gm$] \label{remark3-1}
Theorems \ref{thm1} (ii) and \ref{thm2} (ii) show that the splitting parameter $\gm$ plays an important role in the accuracy of our methods.
\vspace{-.1cm}
\begin{enumerate}\itemsep -4pt
\item[(i)] In one-dimensional cases, there exist two optimal choices of $\gm$ that yield accuracy of ${\mathcal O}(h^2)$:  (i) $\gm = 2$ for any $\ap \in (0, 2)$ and $\lambda \ge 0$; (ii) $\gm = 1+\fl{\ap}{2}$ for $\ap \in (0, 1]$ or $\lambda = 0$; see Fig. \ref{Fig3} (a){\rm\&}(b).

\item[(ii)]  In two- or three-dimensional cases, only the optimal choice $\gm = 2$ leads to   accuracy of ${\mathcal O}(h^2)$; see illustration in Fig. \ref{Fig3} (c){\rm \&}(d). 
\item[(iii)] For other choices of $\gm \in (\ap, 2)$ but not mentioned in (i) and (ii),  our studies show that the accuracy is $\ap$-dependent, i.e., ${\mathcal O}(h^{2-\ap})$.
\end{enumerate}
\end{remark}


\vspace{-3mm}
\section{Numerical accuracy}
\setcounter{equation}{0}
\label{section4}
 
In this section, we test the accuracy of our methods in approximating the operator $(-\Dt + \lambda)^{\fl{\ap}{2}}$ and  solving the fractional  Poisson problems.  
Unless otherwise stated we will use the splitting parameter $\gm = 2$ in all numerical simulations.  
\subsection{Accuracy in approximating $(-\Dt + \lambda)^{\fl{\ap}{2}}$}
\label{section4-1}

In Remarks \ref{remark2-1} and \ref{remark3-1}, we have briefly discussed the accuracy of our methods under different conditions. 
Next, we will carry out further studies to understand their performance  in discretizing  $(-\Dt + \lambda)^{\fl{\ap}{2}}$. 
Here, we  focus on the errors $\|e_{\ap, 2}^h\|_\infty$ and $\|e_{\ap, 2}^h\|_2$ with the error function $e_{\ap, \gm}^h$ defined in (\ref{localerror}).  
Since the analytical solution of $(-\Dt + \lambda)^{\fl{\ap}{2}}u$ remains unknown,  we will use numerical solutions with very fine mesh size, i.e., $h = 1/4096$,  as the ``exact" solutions. 

\bigskip
\noindent{\bf Example 1 (Cases for $\alpha <1$). }
We consider the function
\bea\label{fun2}
u(x, y) = [(1-x^2)_+(1-y^2)_+]^p, \qquad (x, y) \in {\mathbb R}^2, \quad p > 0, 
\eea  
i.e., $u = [(1-x^2)(1-y^2)]^p$ for $(x, y) \in (-1, 1)^2$, otherwise $u = 0$ if $(x, y) \notin (-1, 1)^2$. 
In Fig. \ref{Fig4-1}, we present numerical errors $\|e_{\ap, 2}^h\|_\infty$ and $\|e_{\ap, 2}^h\|_2$ for various $p$, where order lines are included for easy comparison. 
\begin{figure}[htb!]
\centerline{
(a) \includegraphics[height=4.660cm,width=6.46cm]{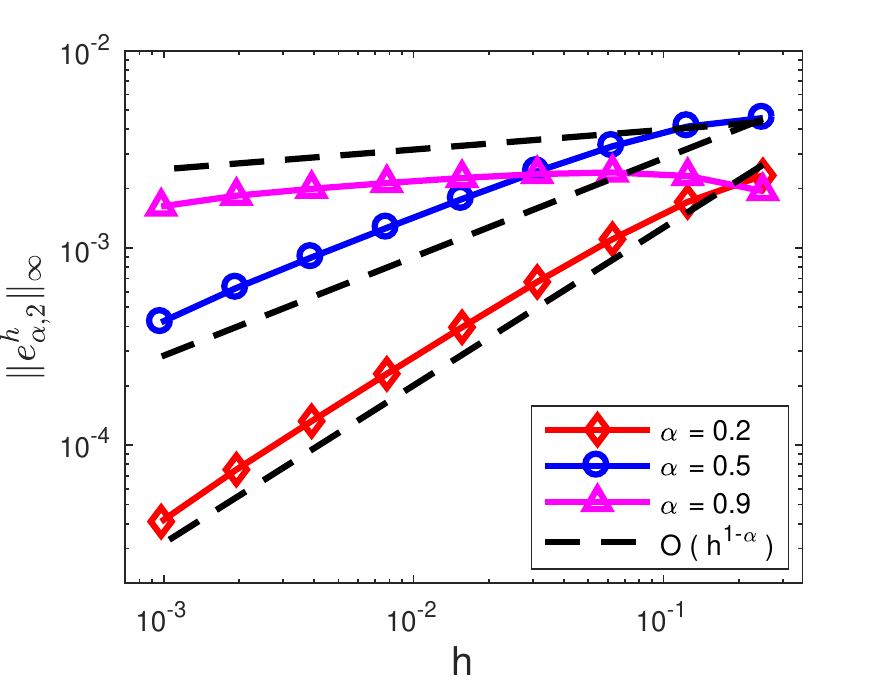}
(b) \includegraphics[height=4.660cm,width=6.46cm]{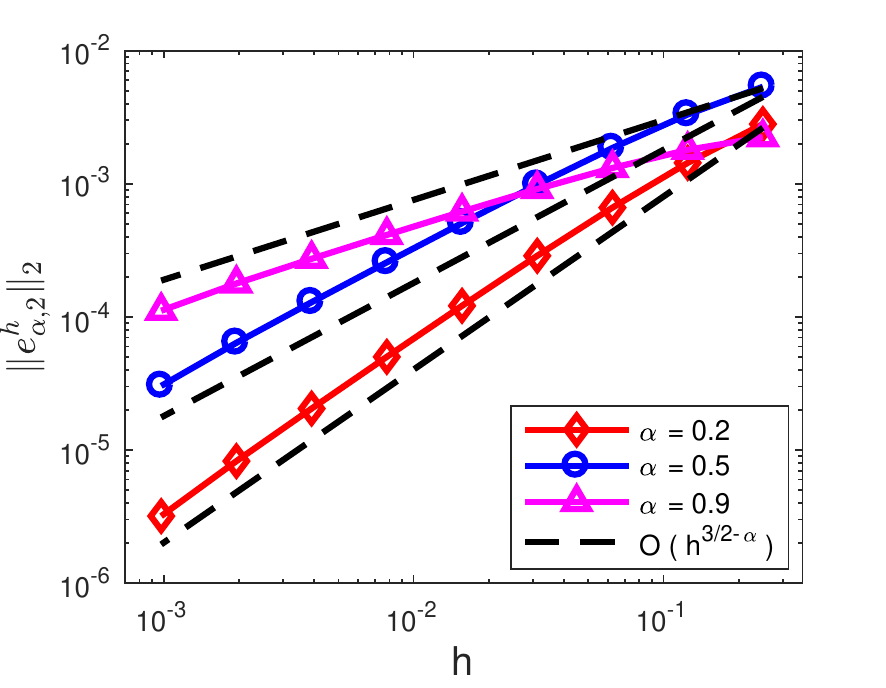}}
\vspace{-1mm}
\centerline{
(c) \includegraphics[height=4.660cm,width=6.46cm]{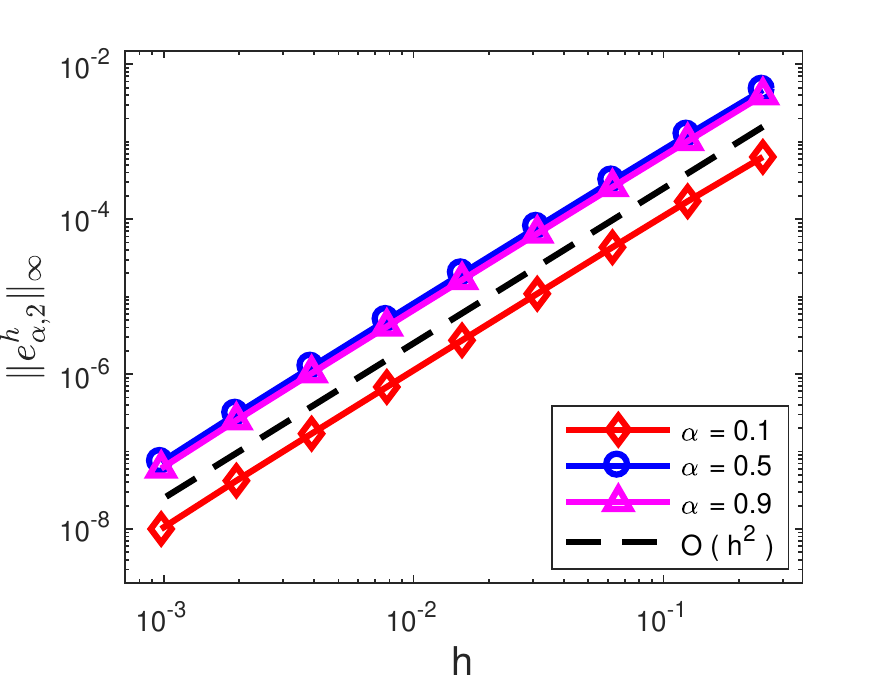}
(d) \includegraphics[height=4.660cm,width=6.46cm]{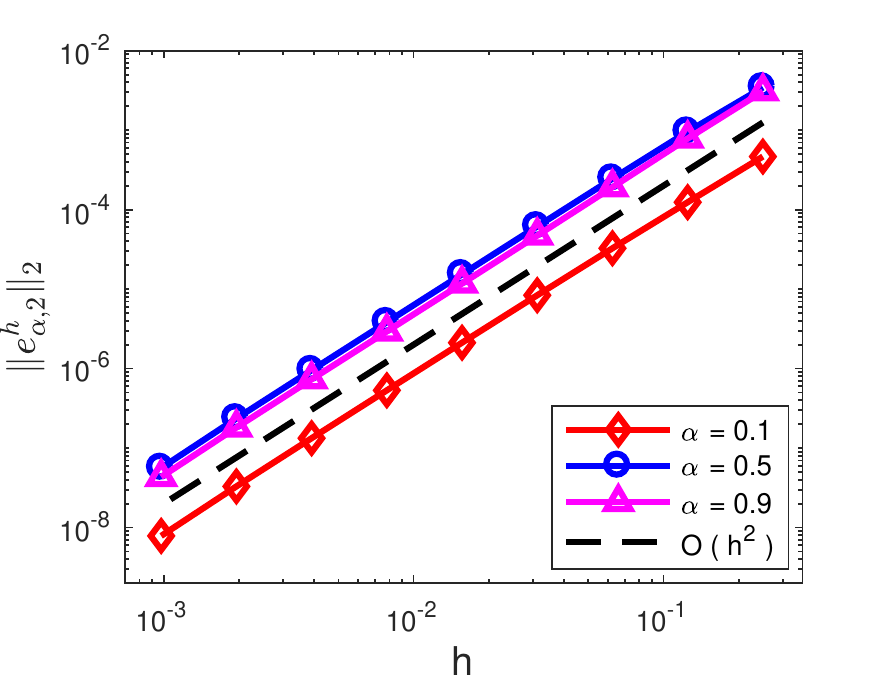}}
\caption{Numerical errors in approximating the operator $(-\Dt + \lambda)^{\fl{\ap}{2}}$ with $\lambda = 0.5$, where  $u$ is defined in {(\ref{fun2})} with $p = 1$ ((a) \& (b)) and $p = 2+\ap + \veps$ and $\veps = 0.05$ ((c) \& (d)).}\label{Fig4-1}
\end{figure}
In Fig. \ref{Fig4-1} (a) \& (b), we choose $p  =1$, i.e., $u \in C^{0,1}({\mathbb R}^2)$. 
It shows that an accuracy of ${\mathcal O}(h^{1-\ap})$ is achieved for all $\ap \in (0, 1)$, confirming our analytical results in Theorem \ref{thm1} (i). 
Additionally, we find that the accuracy in $2$-norm is ${\mathcal O}(h^{\fl{3}{2}-\ap})$,  $1/2$ order higher than that of the $\infty$-norm. 
On the other hand, Fig. \ref{Fig4-1} (c) \& (d) shows numerical errors for  $p = 2+\ap + \veps$, i.e., $u \in C^{2,\alpha+\varepsilon}(\mathbb{R}^2)$,  where $\varepsilon = 0.05$. 
It is clear that our method has the second order of accuracy in both $\infty$- and $2$-norm, independent of parameters $\ap$ and $\lambda$. 
Moreover, our extensive simulations show that  numerical errors are more sensitive to the power $\ap$, but remain almost the same for different damping parameter $\lambda$. 

\begin{figure}[htb!]
\centerline{
(a) \includegraphics[height=4.660cm,width=6.46cm]{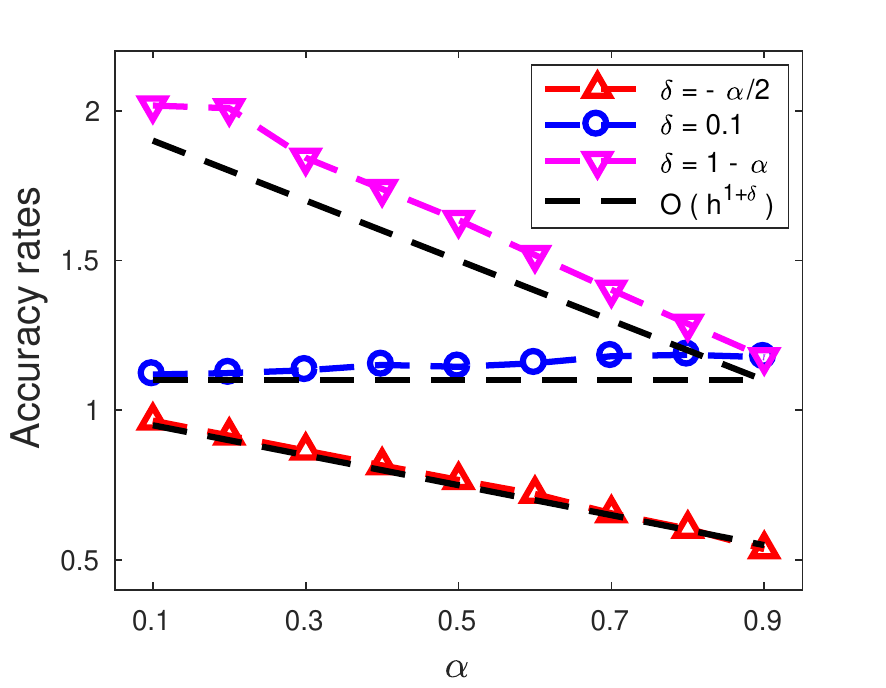}
(b) \includegraphics[height=4.660cm,width=6.46cm]{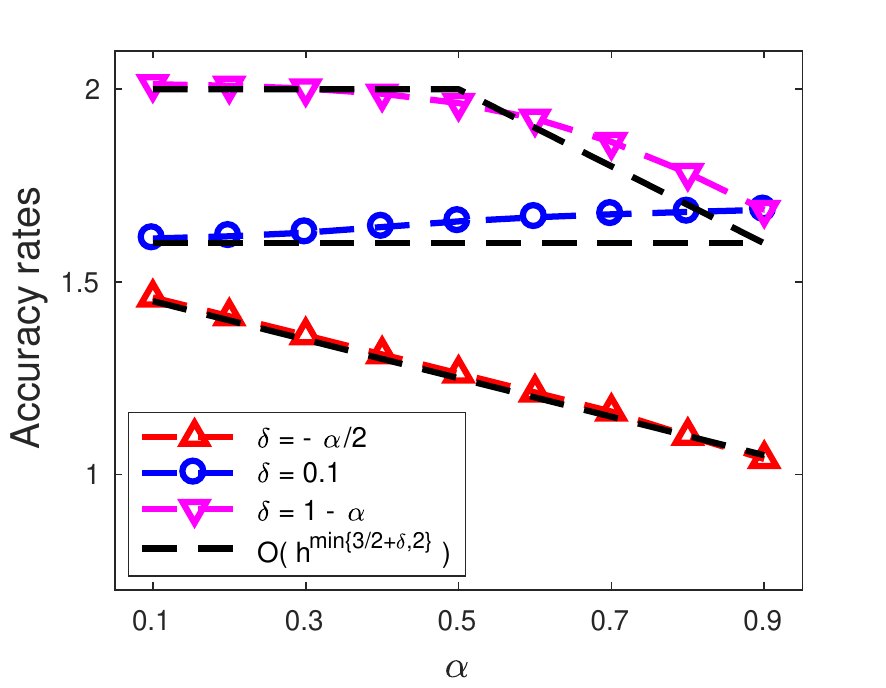}}
\caption{Accuracy rate in $\infty$-norm (a) and $2$-norm (b) of our method in approximating the operator $(-\Dt + \lambda)^\fl{\ap}{2}$ with $\lambda = 0.5$, where $u$ is defined in (\ref{fun2}) with $p = \lceil \ap \rceil + \ap + \dt$. }\label{Fig4-2}
\end{figure} 
Theorem \ref{thm1} predicts the accuracy of our method under the conditions of  $u \in C^{0, \ap+\veps}(\bar{\Og})$ and $u \in C^{2, \ap+\veps}(\bar{\Og})$.
To gain further insights, we study its accuracy if $u \in C^{1,\,\ap + \dt}(\bar{\Og})$ with $-\ap \le \dt \le 1-\ap$, i.e., choosing $u$ in (\ref{fun2}) with $p = \lceil{\ap}\rceil+\ap+\delta$, where $\lceil \cdot \rceil$ denotes the ceiling function.
Fig. \ref{Fig4-2} (a)  shows that our method has the accuracy of ${\mathcal O}(h^{1+\dt})$ in $\infty$-norm. 
Moreover, Fig. \ref{Fig4-2} (b) shows that the accuracy rate in $2$-norm is ${\mathcal O}(h^{\min\{\fl{5}{2}-\ap, 2\}})$.
We remark that in the case of $\dt = -\fl{\ap}{2}$, the accuracy is ${\mathcal O}(h^{1-\fl{\ap}{2}})$, consistent with the error estimates in \cite[Theorem 3.1]{Duo2018} for the fractional Laplacian $(-\Dt)^{\fl{\ap}{2}}$, i.e., $\lambda = 0$. 

%
\bigskip
\noindent{\bf Example 2 (Cases with $\alpha \geq 1$). } 
Again, we consider function $u$ in (\ref{fun2}) and carry out studies for $\ap \ge 1$. 
Fig. \ref{Fig4-3} shows  numerical errors for $\lambda = 0.5$, where $p = 2$ in (a) \& (b) and $p = 2 + \ap + \veps$ with $\veps = 0.05$ in (c) \& (d). 
From Fig. \ref{Fig4-3} (a) \& (c), we find that for $u \in C^{1,1}({\mathbb R}^2)$, our method has an accuracy of ${\mathcal O}(h^{2-\ap})$ in $\infty$-norm, confirming the conclusion in Theorem \ref{thm2} (i) with $\veps = 2 -\ap$. 
Moreover, the accuracy in $2$-norm is $\fl{1}{2}$-order higher, i.e., ${\mathcal O}(h^{\fl{3}{2}-\ap})$.
\begin{figure}[htb!]
\centerline{
(a) \includegraphics[height=4.660cm,width=6.46cm]{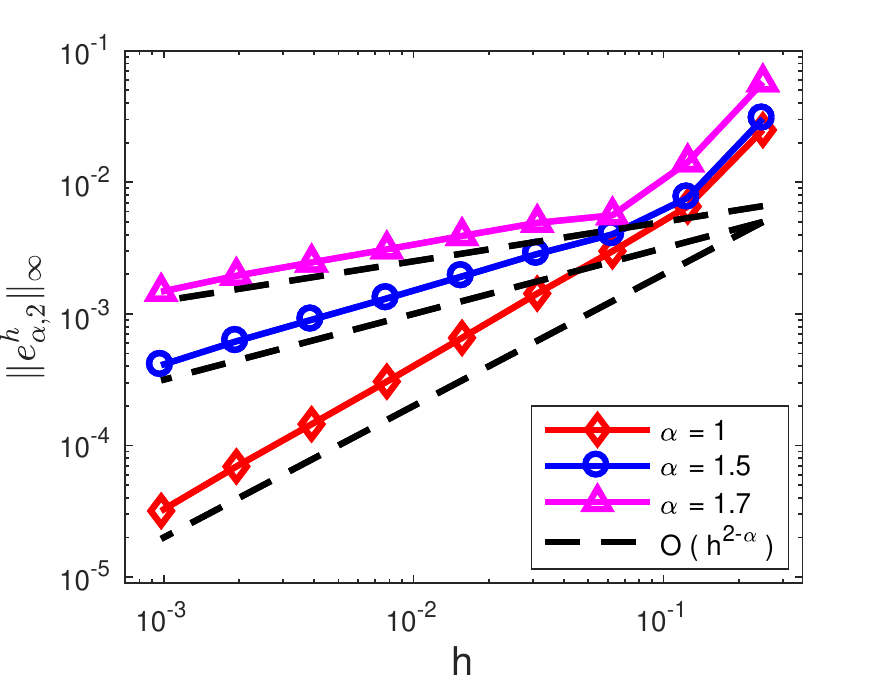}
(b) \includegraphics[height=4.660cm,width=6.46cm]{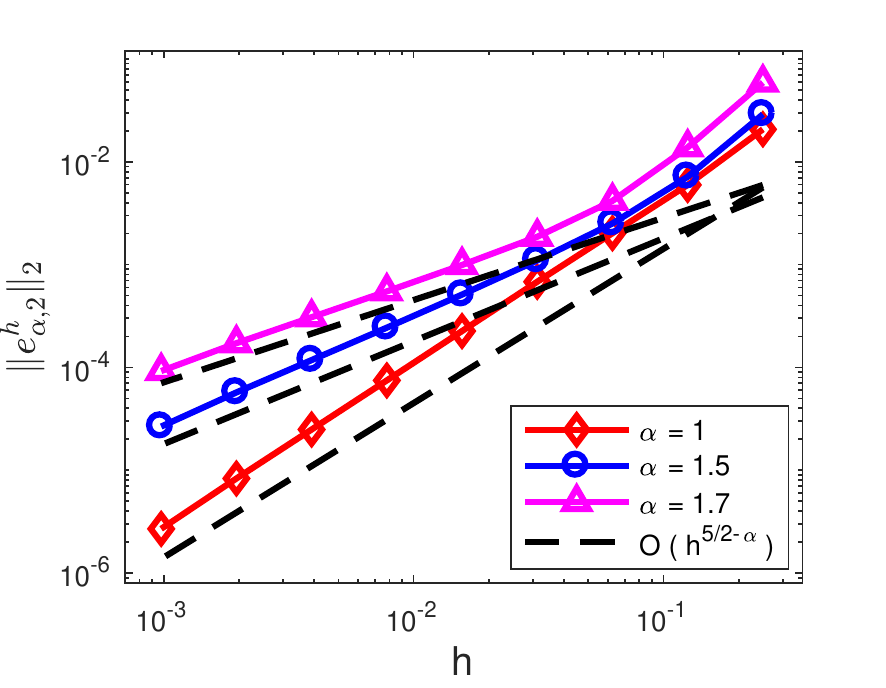}}
\vspace{-1mm}
\centerline{(c) \includegraphics[height=4.660cm,width=6.46cm]{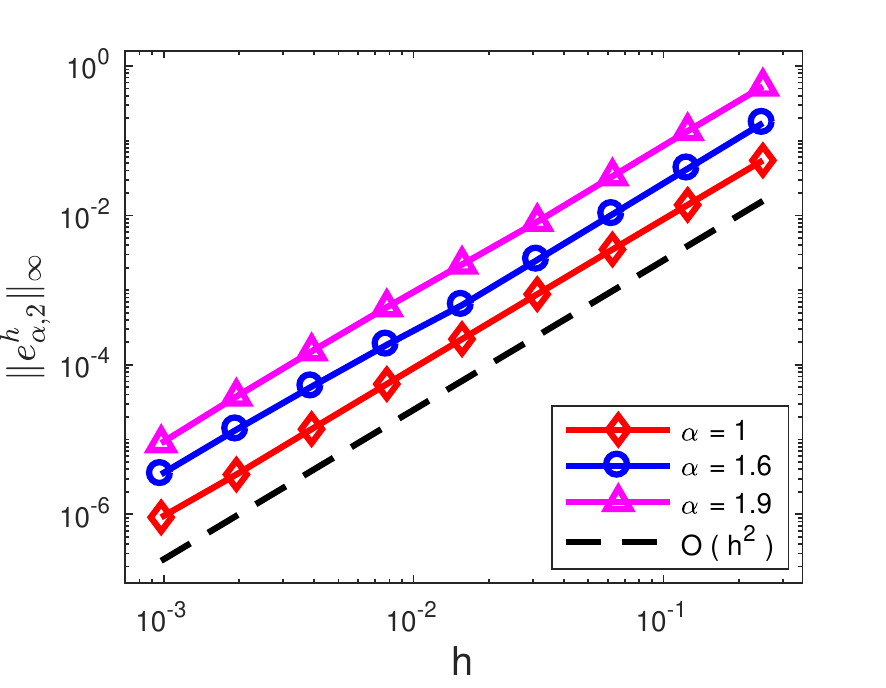}
(d) \includegraphics[height=4.660cm,width=6.46cm]{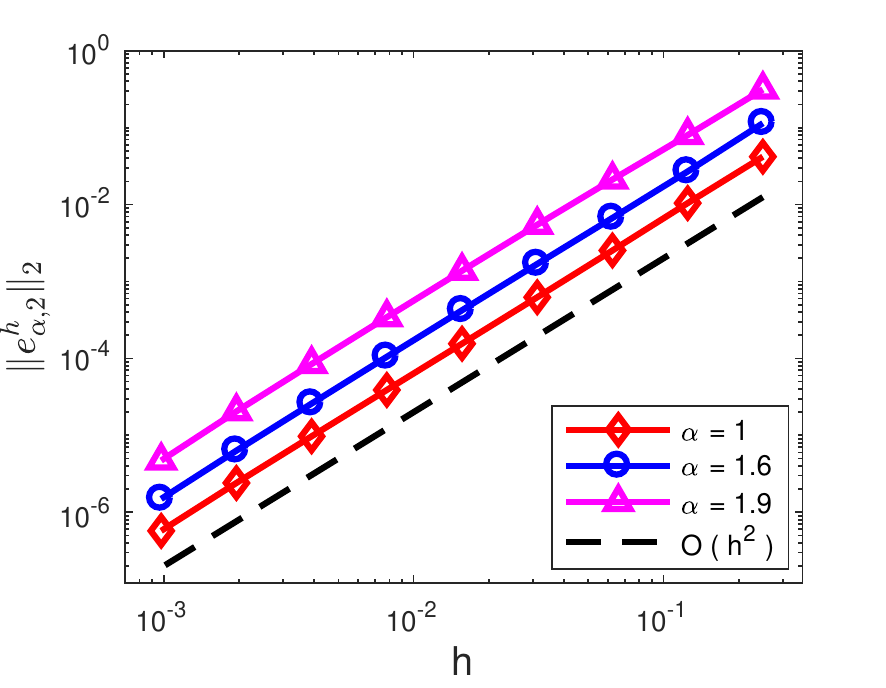}}
\caption{Numerical errors in approximating the operator $(-\Dt + \lambda)^{\fl{\ap}{2}}$ with $\lambda = 0.5$, where  $u$ is defined in (\ref{fun2}) with $p = 2$ ((a) \& (b)) and $p = 2+\ap + \veps$ and $\veps = 0.05$ ((c) \& (d)).}\label{Fig4-3}
\end{figure}
Similar to the case of $\ap < 1$,  the larger the value of $\ap$, the bigger the numerical errors. 
On the other hand, Fig. \ref{Fig4-3} (c) \& (d) not only verify our conclusions in Theorem \ref{thm2} (ii)  but also suggest that the accuracy of our method in $2$-norm is also ${\mathcal O}(h^2)$.

\subsection{Accuracy in solving fractional Poisson problems}
\label{section4-2}

In the following, we test the performance of our method in solving the fractional Poisson problem:  
\bea\label{Poisson}
(-\Dt + \lambda)^{\fl{\ap}{2}}u = f(\bx), &\ & \bx \in (-1, 1)^d,\\ 
\label{BCP}
u(\bx) = 0, && \bx \in {\mathbb R}^d\backslash (-1, 1)^d.
\eea
Our extensive studies show that the same conclusions can be obtained from solving any $d$-dimensional ($d \ge 1$) Poisson problem (\ref{Poisson})--(\ref{BCP}).  
For the purpose of brevity, we will thus focus on the examples with $d = 2$.
Here, the numerical errors are computed as $\|e_u\|_p = \|u - u_h\|_p$, 
where $u$ and $u_h$ denote the exact and numerical solutions of (\ref{Poisson})--(\ref{BCP}), respectively.

\vskip 8pt
\noindent{\bf Example 3.}  We consider a benchmark example  in \cite[Example 3]{Sun0018} -- the fractional Poisson problem (\ref{Poisson})--(\ref{BCP}) with exact solution $u$ defined as in (\ref{fun2}) with $p = 2$.
It is easy to verify that the solution satisfies $u \in C^{1,1}({\mathbb R}^2)$.
In practice, the function $f$ in (\ref{Poisson})  is prepared numerically with a fine mesh size $h_x = h_y = 2^{-12}$, i.e., computing $f = (-\Dt+\lambda)^{\fl{\ap}{2}}u$ with exact $u$. 
\begin{table}[htb!]
\begin{center}
\begin{tabular}{|c|l r r r r r r|}
\hline
\diagbox[width=3em]{$\ap$}{$h$} && 1/16 & 1/32 & 1/64 & 1/128 & 1/256  & 1/512 \\
\hline
\multirow{4}{*}{$0.5$} & \multirow{2}{*}{$\|e_u\|_\infty$}&  6.117E-4 &   1.498E-4 & 3.699E-5 & 9.176E-6 &  2.278E-6 & 5.613E-7  \\
& & c.r. & 2.0295 & 2.0179 & 2.0114 & 2.0103 & 2.0206  \\
\cline{2-8}
&  \multirow{2}{*}{$\|e_u\|_2$}& 4.642E-4 & 1.141E-4 & 2.842E-5 &  7.106E-6 & 1.774E-6 & 4.387E-7  \\
& & c.r. & 2.0249 & 2.0046 & 2.0000 & 2.0024 & 2.0155  \\
\hline
\multirow{4}{*}{$1$} & \multirow{2}{*}{$\|e_u\|_\infty$}& 1.001E-3 &  2.389E-4 & 5.785E-5 & 1.414E-5 & 3.471E-6 & 8.487E-7 \\
& & c.r. &2.0661 &2.0462 &2.0328 &2.0261 &2.0321  \\
\cline{2-8}
&  \multirow{2}{*}{$\|e_u\|_2$}&  7.961E-4 & 1.864E-4 & 4.522E-5 &  1.117E-5 &  2.774E-6 & 6.852E-7  \\
& & c.r. &2.0945 &2.0434 &2.0179 &2.0092 &2.0172  \\
\hline
\multirow{4}{*}{$1.5$} & \multirow{2}{*}{$\|e_u\|_\infty$}& 1.669E-3 &  3.893E-4 &  9.177E-5 & 2.184E-5 & 5.229E-6 & 1.250E-6  \\
& & c.r. & 2.1004 & 2.0848 & 2.0714 & 2.0622 & 2.0645 \\
\cline{2-8}
&  \multirow{2}{*}{$\|e_u\|_2$}&  1.506E-3 & 3.361E-4 & 7.660E-5 &  1.783E-5 &  4.218E-6 & 1.004E-6  \\
& & c.r. & 2.1642 & 2.1332 & 2.1032 & 2.0794 & 2.0705 \\
\hline
\end{tabular}
\caption{Numerical errors in solving the 2D Poisson problem (\ref{Poisson})--(\ref{BCP}) with $\lambda = 0.5$, where $f$ is chosen such that the exact solution is $u(\bx) = (1-x^2)_+^2(1-y^2)_+^2$.}
\label{Tab1}
\end{center}\vspace{-5mm}
\end{table}
In Table \ref{Tab1}, we present numerical errors $\|e_u\|_\infty$ and $\|e_u\|_2$ for various $\ap$, where $\lambda = 0.5$.
It shows that even though the solution $u \in C^{1, 1}({\mathbb R}^2)$, our method achieves the  accuracy of ${\mathcal O}(h^2)$, {\it uniformly} for any $\ap \in (0, 2)$. 
In other words, to obtain the second-order accuracy,  the regularity  that is required on the solution of fractional Poisson problems is much lower than that required in approximating the operator $(-\Dt+\lambda)^{\fl{\ap}{2}}$ in Theorems \ref{thm1} and \ref{thm2}. 
This observation is consistent with the central difference scheme for the classical Laplace operator $\Dt$. 
Additionally, Fig. \ref{Fig7-1} compares the numerical errors for different $\lambda$, where $\ap = 0.5$ and $1.5$. 
It shows that for small $\ap$, the smaller the parameter $\lambda$,  the less the numerical errors, and numerical errors for $\lambda = 0$ are minimized. 
However,  numerical errors become insensitive to $\lambda$ as $\ap$ increases. 
\begin{figure}[htb!]
\centerline{
(a) \includegraphics[height=4.66cm,width=6.46cm]{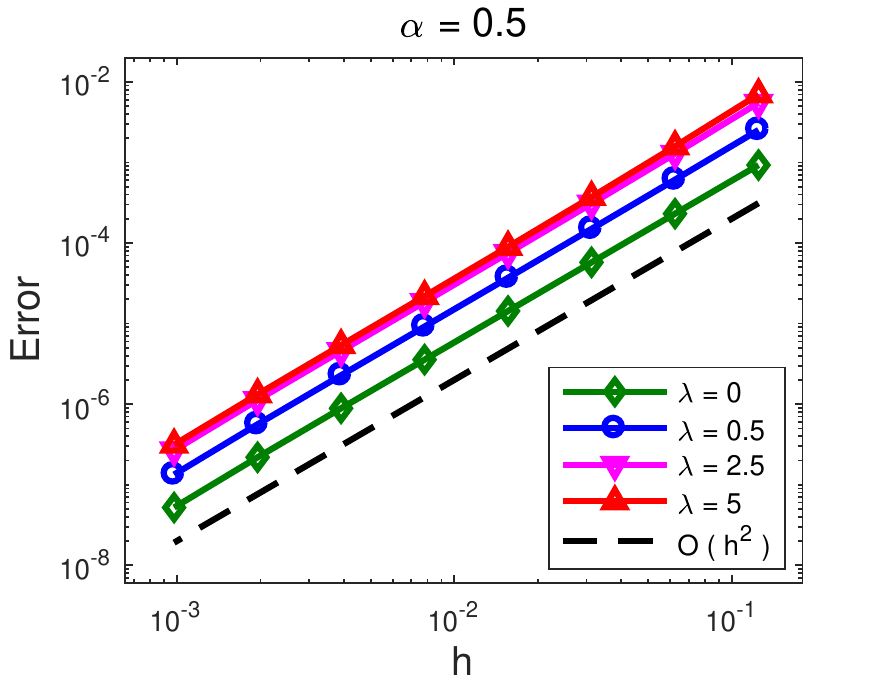}
(b) \includegraphics[height=4.66cm,width=6.46cm]{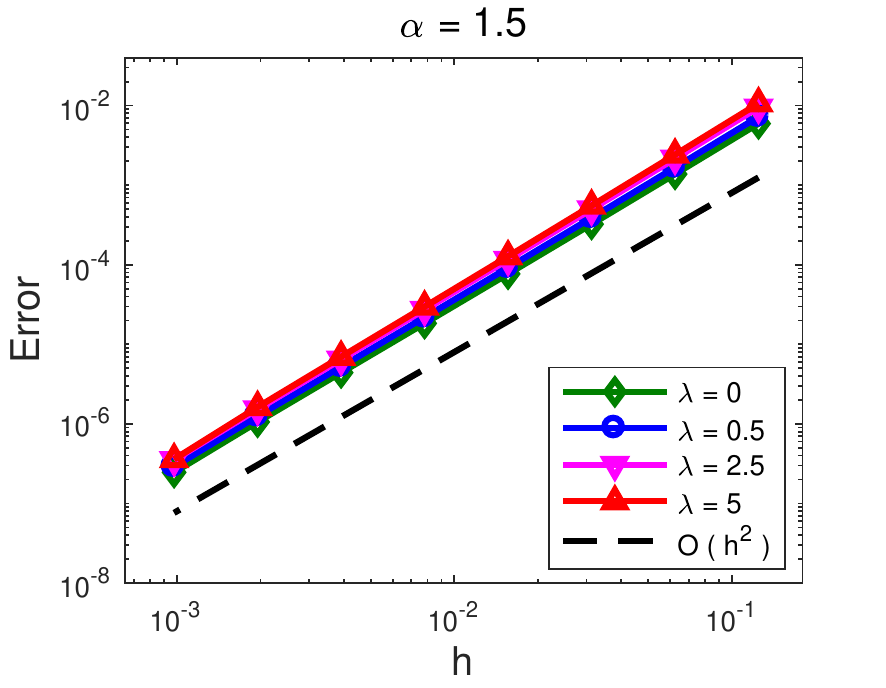}}
\caption{Numerical errors $\|e_u\|_\infty$ in solving the fractional Poisson problem in Example 3.}  \label{Fig7-1}
\end{figure}

Moreover, we compare the errors in Table \ref{Tab1} with those in \cite[Table 8]{Sun0018} to further study the performance of our methods. 
We find that: (i) The method in \cite{Sun0018}  has accuracy of ${\mathcal O}(h^{2-\ap})$. 
In contrast, our accuracy is ${\mathcal O}(h^{2})$ for any $\ap \in (0, 2)$. 
(ii) For fixed $\ap$, $\lambda$ and $h$, numerical errors of our method are much smaller than those  in \cite[Table 8]{Sun0018}. 
For example, when $\ap = 1.5$, $\lambda = 0.5$,  and $h = 1/256$, our method yields $\|e_u\|_\infty = 5.229$E-6, but a much larger error $\|e_u\|_\infty =5.1048$E-2 from the method in \cite[Table 8]{Sun0018}.  
As discussed in Remark \ref{remark2-1}, including of the damping term $e^{-\lambda|\bxi|}$ in  the weight function is crucial in the design of an accurate finite difference method for the tempered fractional Laplacian.
The above comparisons further confirm our conclusions in  Remark \ref{remark2-1}. 

From the above results and our extensive studies, we conclude that to obtain the second-order accuracy in solving fractional Poisson problems, our method requires the solution $u\in C^{1,1}(\bar{\Og})$  {\it at most}.
In the following, we will test the performance of our method in solving the Poisson problem (\ref{Poisson})--(\ref{BCP}), when its solution has lower regularity than that in Example 3. 

\bigskip
\noindent{\bf Example 4.}  We solve the problem (\ref{Poisson})--(\ref{BCP}) with the exact solution $u(\bx) = (1-x^2)_+(1-y^2)_+$,
i.e., the solution $u \in C^{0, 1}({\mathbb R}^2)$.  
The right hand side function $f$ is computed in the same manner as that in Example 3. 
Table \ref{Tab2} shows numerical errors $\|e_u\|_\infty$ and $\|e_u\|_2$ for various $\ap$, where $\lambda = 0.5$.
\begin{table}[htb!]
\begin{center}
\begin{tabular}{|c|c| r r r r r r|}
\hline
\diagbox[width=3em]{$\ap$}{$h$} && 1/16 & 1/32 & 1/64 & 1/128 & 1/256  & 1/512 \\
\hline
\multirow{4}{*}{$0.5$} & \multirow{2}{*}{$\|e\|_\infty$}& 2.692E-3 &  1.284E-3 & 6.072E-4 & 2.881E-4 & 1.376E-4 & 6.569E-5 \\
& & c.r. & 1.0677 & 1.0806 & 1.0753 & 1.0661 & 1.0669  \\
&  \multirow{2}{*}{$\|e\|_2$}& 2.732E-3 & 1.020E-3 & 3.622E-4 &  1.248E-4 &  4.225E-5 & 1.406E-5  \\
& & c.r. & 1.4213 & 1.4939 & 1.5368 & 1.5631 & 1.5870  \\
\hline
\multirow{4}{*}{$1$} & \multirow{2}{*}{$\|e\|_\infty$}& 2.688E-3 &    1.370E-3 & 6.836E-4 & 3.387E-4 & 1.673E-4 & 8.197E-5 \\
& & c.r. &0.9722 &1.0032 &1.0131 &1.0171 &1.0296 \\
&  \multirow{2}{*}{$\|e\|_2$}& 3.188E-3 &  1.370E-3 & 5.546E-4 &  2.162E-4 &  8.210E-5 & 3.040E-5  \\
& & c.r. &1.2188 &1.3045 &1.3591 &1.3968 &1.4331  \\
\hline
\multirow{4}{*}{$1.8$} & \multirow{2}{*}{$\|e\|_\infty$}& 8.247E-4 &   4.620E-4 & 2.467E-4  & 1.272E-4 & 6.444E-5 & 3.202E-5 \\
& & c.r. & 0.8358 & 0.9054 & 0.9557 & 0.9809 & 1.0088  \\
&  \multirow{2}{*}{$\|e\|_2$}& 1.315E-3 & 7.216E-4 & 3.718E-4 &  1.840E-4 &  8.870E-5 & 4.167E-5  \\
& & c.r. & 0.8660 & 0.9568 & 1.0144 & 1.0531 & 1.0897 \\
\hline
\end{tabular}
\caption{Numerical errors in solving the 2D Poisson problem (\ref{Poisson})--(\ref{BCP}) with $\lambda = 0.5$, where $f$ is chosen such that the exact solution is $u(\bx) = (1-x^2)_+(1-y^2)_+$.
}
\label{Tab2}
\end{center}
\end{table}
From it,  we find that the accuracy  in $\infty$-norm is ${\mathcal O}(h)$, independent of the values of $\ap$ and $\lambda$. 
The accuracy rate in 2-norm is higher -- the smaller the value of $\alpha$, the higher the accuracy rate in $2$-norm.  
As $\ap \to 2^-$, the accuracy in $2$-norm becomes ${\mathcal O}(h)$. 
From Table \ref{Tab2} and our extensive simulations, we find that if the solution satisfies $u \in C^{0, s}(\bar{\Og})$ for $s\in(0, 1]$, our method has the accuracy in $\infty$-norm ${\mathcal O}(h^s)$ in solving the fractional Poisson problem. 

\bigskip
\noindent{\bf Example 5.}  We solve the problem (\ref{Poisson})--(\ref{BCP}) with $f = 1$. 
In this case, the regularity of the solution is much lower.  
Table \ref{Tab3} shows numerical errors $\|e_u\|_\infty$ and $\|e_u\|_2$ for various $\ap$, where $\lambda = 0.5$. 
Even though the regularity of the solution in this case is much lower than those in Examples 3--4, our method approximates it with a reasonable accuracy.
\begin{table}[htb!]
\begin{center}
\begin{tabular}{|c|l r r r r r r|}
\hline
\diagbox[width=3em]{$\ap$}{$h$} && 1/16 & 1/32 & 1/64 & 1/128 & 1/256  & 1/512 \\
\hline
\multirow{4}{*}{$0.5$} & \multirow{2}{*}{$\|e_u\|_\infty$}&  1.737E-1 &  1.312E-1 &  1.013E-1 &  7.987E-2 &  6.401E-2 &  5.196E-2 \\
& & c.r. & 0.4048 & 0.3724 & 0.3434 & 0.3195 & 0.3008 \\
&  \multirow{2}{*}{$\|e_u\|_2$}& 1.935E-1 & 1.110E-1 & 6.248E-2 &  3.488E-2 &  1.944E-2 &  1.086E-2 \\
& & c.r. & 0.8016 & 0.8289 & 0.8412 & 0.8436 & 0.8396 \\
\hline
\multirow{4}{*}{$1$} & \multirow{2}{*}{$\|e_u\|_\infty$}& 2.634E-2 & 1.798E-2 & 1.238E-2 & 8.593E-3 & 6.000E-3 &  4.209E-3 \\
& & c.r. & 0.5509 & 0.5381 & 0.5269 & 0.5181 & 0.5117 \\
&  \multirow{2}{*}{$\|e_u\|_2$}& 3.380E-2 & 1.875E-2 & 1.014E-2 & 5.392E-3 & 2.836E-3 & 1.481E-3  \\
& & c.r. & 0.8504 & 0.8868 & 0.9109 & 0.9269 & 0.9377 \\
\hline
\multirow{4}{*}{$1.8$} & \multirow{2}{*}{$\|e_u\|_\infty$}& 2.917E-3 & 1.560E-3 & 8.433E-4 & 4.545E-4 & 2.442E-4 & 1.310E-4 \\
& & c.r. & 0.9030 & 0.8874 & 0.8917 & 0.8961 & 0.8985 \\
&  \multirow{2}{*}{$\|e_u\|_2$}& 5.271E-3 & 2.711E-3 & 1.390E-3 & 7.080E-4 & 3.584E-4 & 1.806E-4  \\
& & c.r. & 0.9590 & 0.9637 & 0.9735 & 0.9824 & 0.9888  \\
\hline
\end{tabular}
\caption{Numerical errors in solving the 2D Poisson problem (\ref{Poisson})--(\ref{BCP}) with $\lambda = 0.5$ and $f = 1$.}
\label{Tab3}
\end{center}\vspace{-5mm}
\end{table}
It shows that the accuracy in $\infty$-norm is ${\mathcal O}(h^{\fl{\ap}{2}})$ for $\ap \in (0, 2)$, while ${\mathcal O}(h^{\min\{1, \fl{1}{2}+\fl{\ap}{2}\}})$ in 2-norm, that is, the $2$-norm errors for $\ap \ge 1$ are uniformly ${\mathcal O}(h)$.  
%

\section{Applications to tempered fractional PDEs}
\setcounter{equation}{0}
\label{section5}

In this section, we apply our methods to solve various fractional problems with the tempered operator $(-\Dt + \lambda)^{\fl{\ap}{2}}$, so as to study the effects of the fractional power $\ap$ and damping constant $\lambda$ on their solutions. 
In the following applications,  the spatial discretization is done by our finite difference methods, and the temporal discretization is realized by the Crank--Nicolson method. 
In practice, fast algorithms via the fast Fourier transform are used for efficient computations, and at each time step the computational costs are ${\mathcal O}(M\log M)$ with $M$ the number of spatial unknowns.

\subsection{Fractional Allen--Cahn equation}
\label{section5-2}

Consider the two-dimensional tempered fractional Allen--Cahn equation of the form: 
\begin{eqnarray}\label{AC-1}
	\p_tu(\bx, t) = -(-\Delta + \lambda)^{\fl{\alpha}{2}}u - \frac{1}{\varepsilon^\ap}\,u(u^2-1), \ && \bx \in \Omega = (0, 1)^2, \quad t>0,\qquad\qquad \\ \label{AC-2}
	u(\bx, t) = -1, \  && \bx \in \Omega^{c}, \quad t\geq 0, 
\end{eqnarray}
where $\varepsilon > 0$ describes the diffuse interface width. 
In the special case with $\lambda = 0$ and $\ap \to 2^-$,   (\ref{AC-1})--(\ref{AC-2}) reduces to the well-known classical Allen--Cahn equation -- one of the most popular phase field models in materials science and fluid dynamics. 
Here, we study the coalescence of two ``kissing" bubbles, a benchmark problem in the phase field models. 
We take the initial condition as 
\beas\label{u0-two_bubble}
	u(\bx, 0) = 
		1- \tanh\bigg(\frac{|\bx - \bx_1| - {0.12}}{{\varepsilon}} \bigg) - \tanh\bigg(\frac{|\bx - \bx_2| - {0.12}}{{\varepsilon}} \bigg), \qquad \bx \in \Og
\eeas
with $\bx_1$ and $\bx_2$ denoting the initial center of  two bubbles, which are chosen such that 
two bubbles are initially osculating or ``kissing". 
Note that the boundary condition in (\ref{AC-2}) is nonzero constant. 
Letting $\bar{u} = u + 1$, we can rewrite the problem  (\ref{AC-1}) as an equation of $\bar{u}$ with the extended homogeneous boundary conditions, so that our method can be directly applied.
In our simulations, we choose the mesh size $h = 1/1024$ and the time step $\tau = 0.0005$.  
\begin{figure}[htb!]
\centerline{
\includegraphics[height=3.36cm,width=3.480cm]{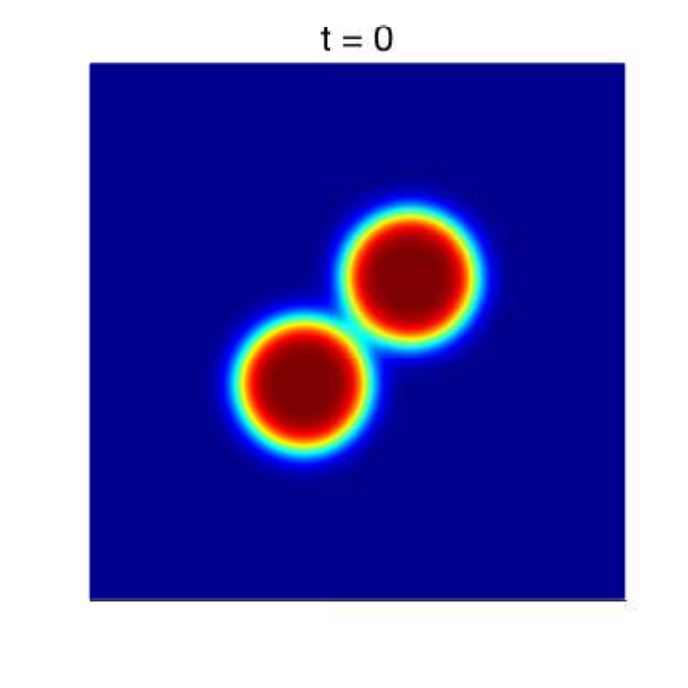}\hspace{-4mm}
\includegraphics[height=3.36cm,width=3.480cm]{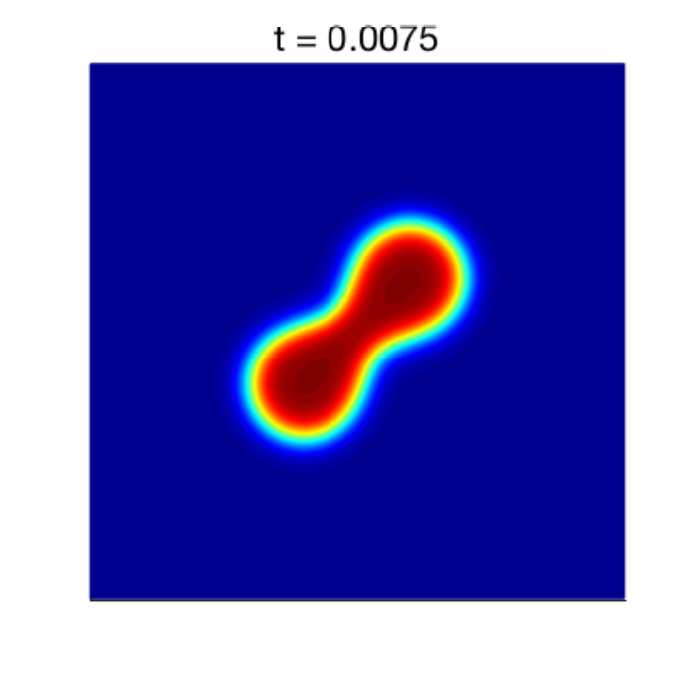}\hspace{-4mm}
\includegraphics[height=3.36cm,width=3.480cm]{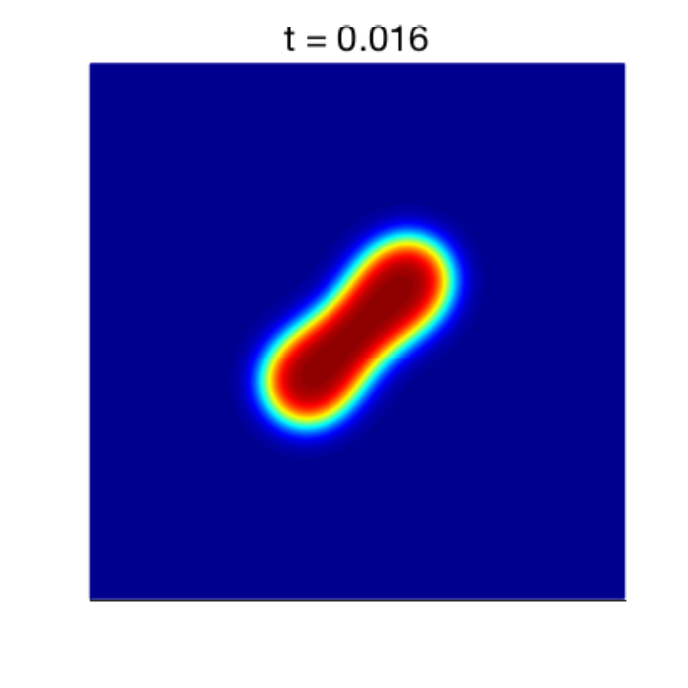}\hspace{-4mm}
\includegraphics[height=3.36cm,width=3.480cm]{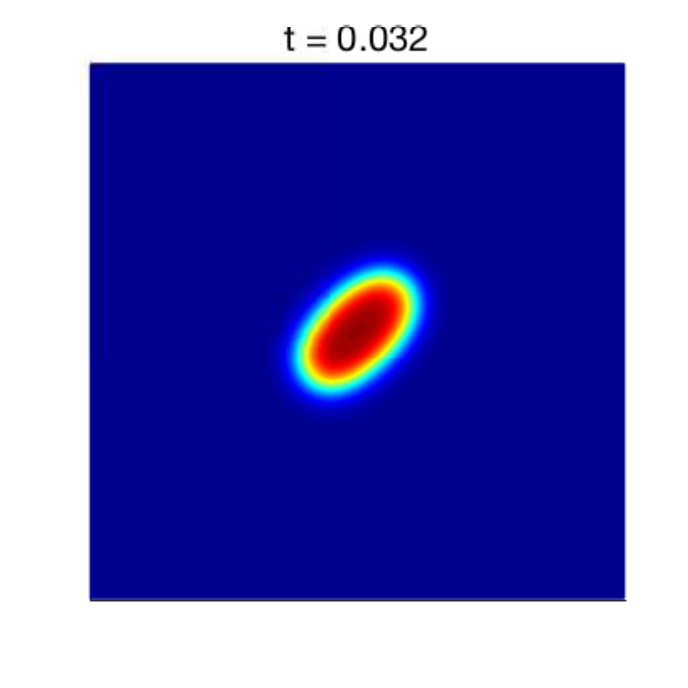}}
\vspace{-5mm}
\centerline{
\includegraphics[height=3.36cm,width=3.480cm]{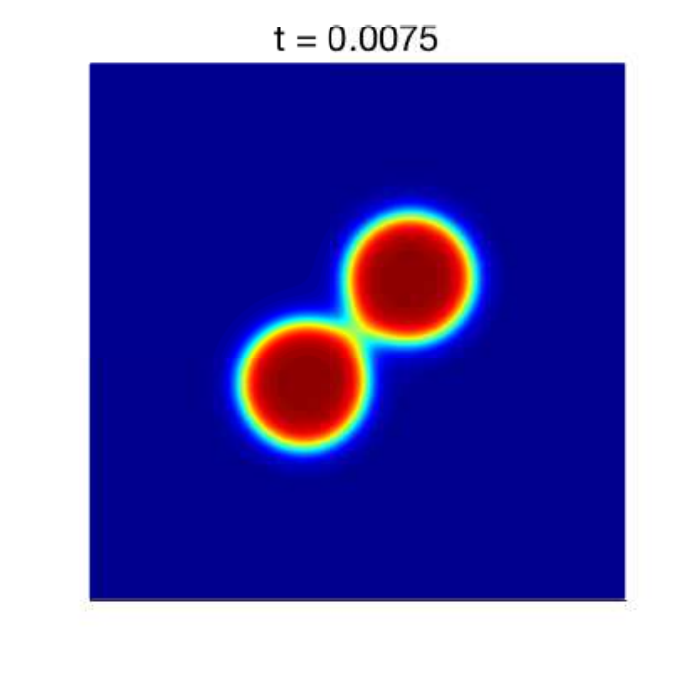}\hspace{-4mm}
\includegraphics[height=3.36cm,width=3.480cm]{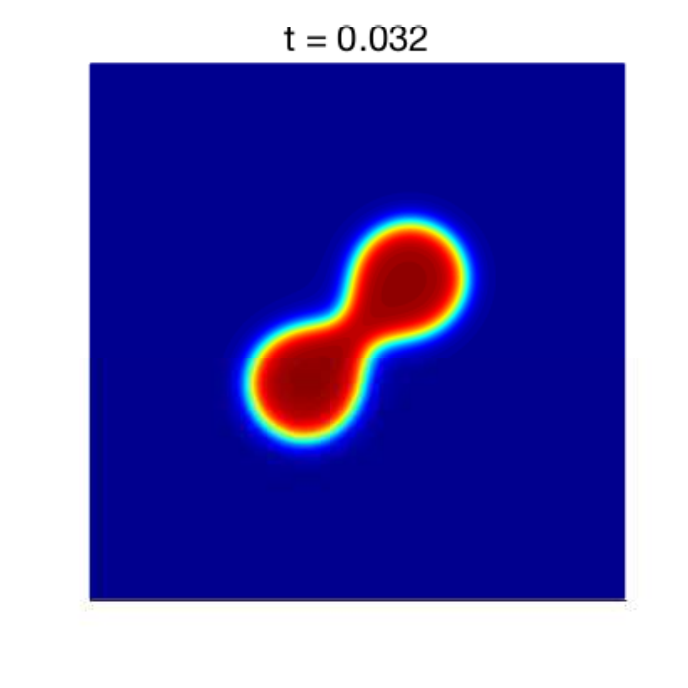}\hspace{-4mm}
\includegraphics[height=3.36cm,width=3.480cm]{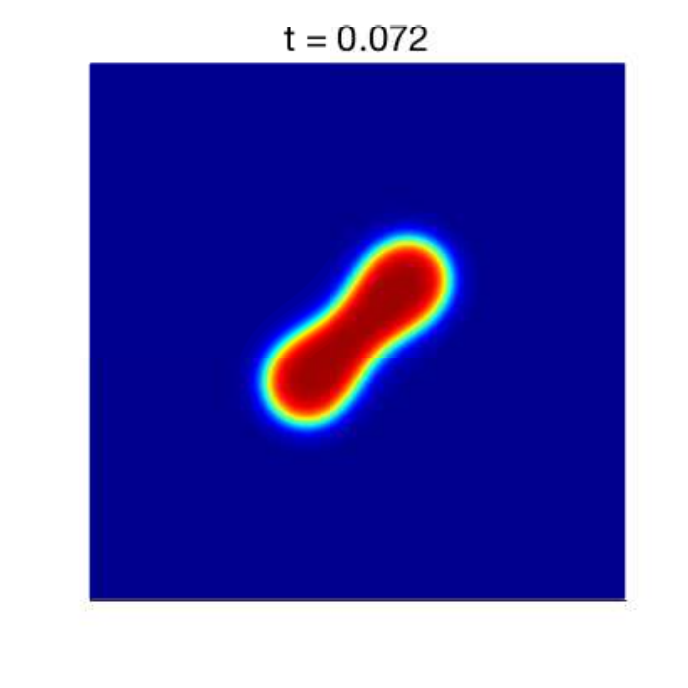}\hspace{-4mm}
\includegraphics[height=3.36cm,width=3.480cm]{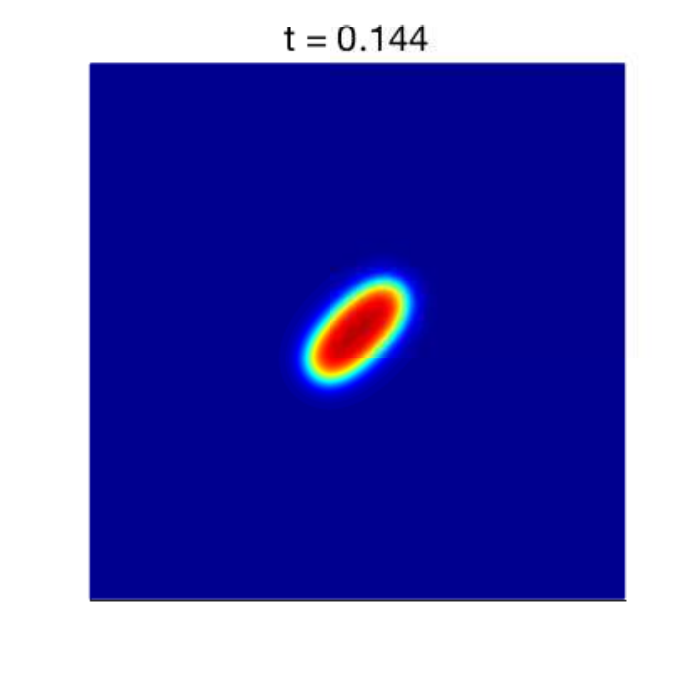}}
\vspace{-5mm}
\centerline{
\includegraphics[height=3.36cm,width=3.480cm]{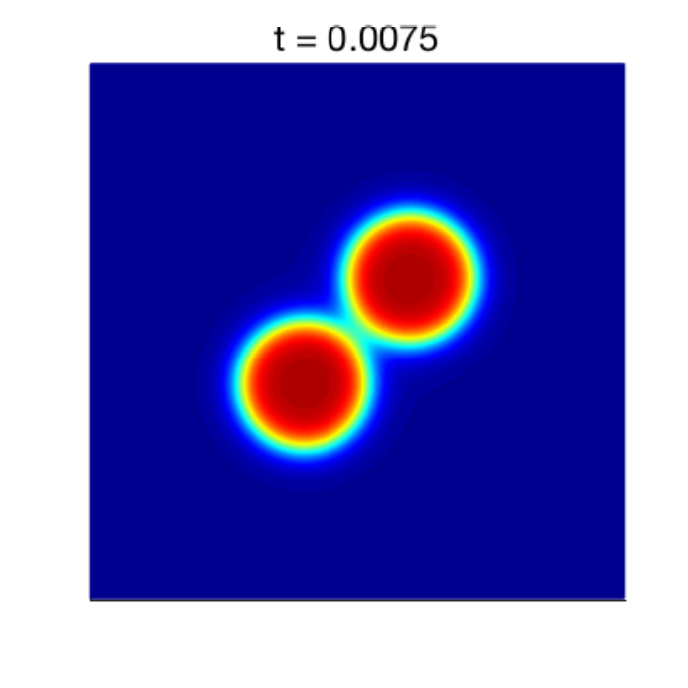}\hspace{-4mm}
\includegraphics[height=3.36cm,width=3.480cm]{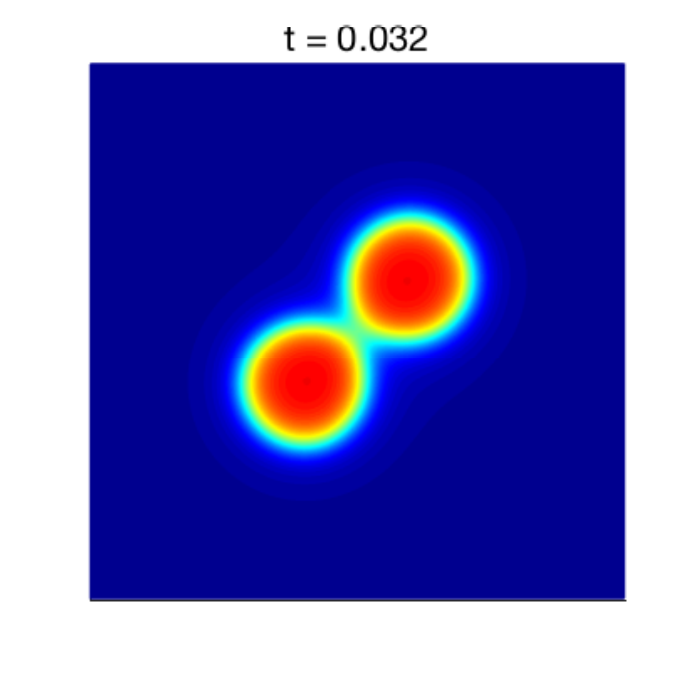}\hspace{-4mm}
\includegraphics[height=3.36cm,width=3.480cm]{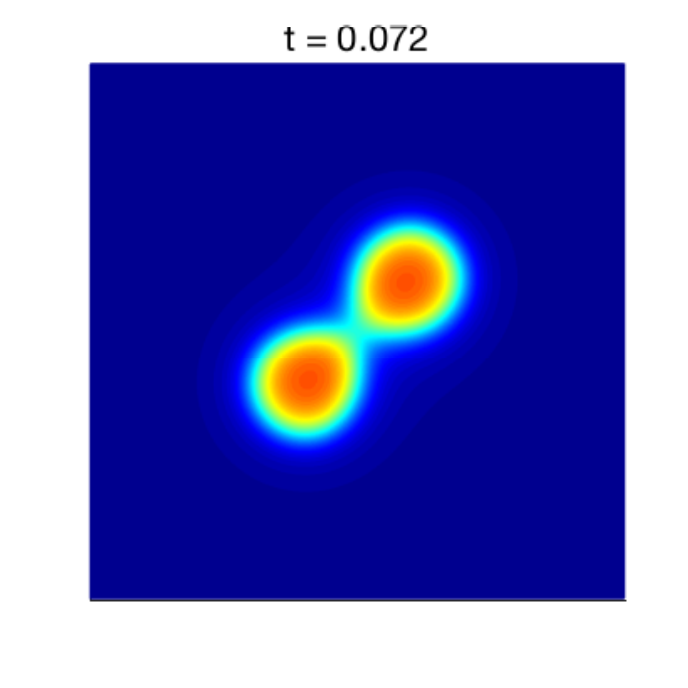}\hspace{-4mm}
\includegraphics[height=3.36cm,width=3.480cm]{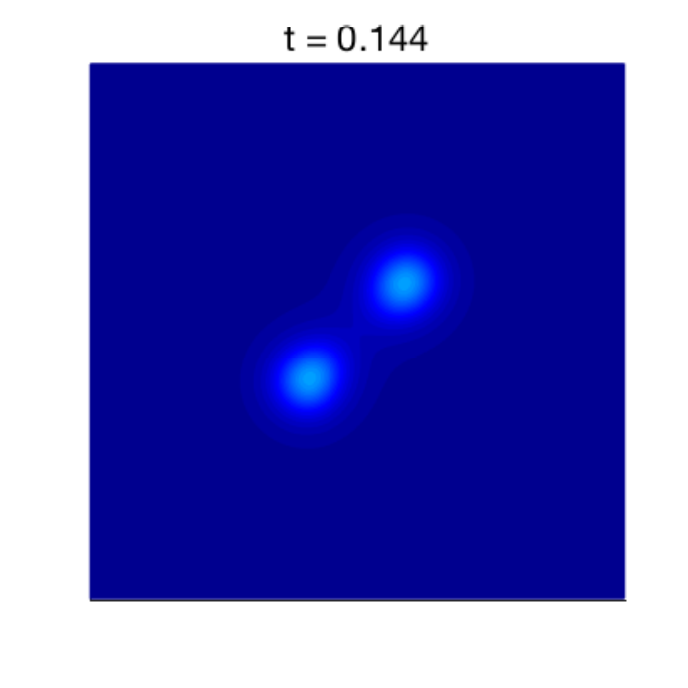}}
\vspace{-5mm}
\caption{Dynamics of the two kissing bubbles in the two-dimensional tempered fractional Allen--Cahn equation with $\varepsilon = 0.03$ and $\lambda = 0.2$.  
From top to bottom:  $\alpha = 1.9, 1.5, 1$. }\label{Fig-AC1}
\end{figure}

Figs. \ref{Fig-AC1}--\ref{Fig-AC2} show the time evolution of two bubbles in (\ref{AC-1})--(\ref{AC-2}) with $\varepsilon = 0.03$,  for various $\ap$ and $\lambda$.
Initially,  two bubbles are centered at $\bx_1 = (0.4, 0.4)$ and $\bx_2 = (0.6, 0.6)$, respectively.
It is well known that in the classical Allen--Cahn equation, two bubbles first coalesce into one, and then this newly formed bubble shrinks and  eventually disappear.
In Fig.  \ref{Fig-AC1} with fixed $\lambda = 0.2$, we find that the dynamics of two bubbles  are similar to the behaviors in the classical Allen--Cahn equation. 
With $\ap$ decreasing, the merging and shrinking of the bubbles becomes much slower (cf. Fig. \ref{Fig-AC1} for $\ap = 1.9$ and $1$).  
When further reducing $\ap$ (e.g., $\ap = 1$),  the two bubbles never merge completely. 

In Fig. \ref{Fig-AC2} with fixed $\ap = 1.8$, the effects of the damping term on the dynamics of two bubbles are studied for various $\lambda$.
\begin{figure}[htb!]
\centerline{
\includegraphics[height=3.36cm,width=3.48cm]{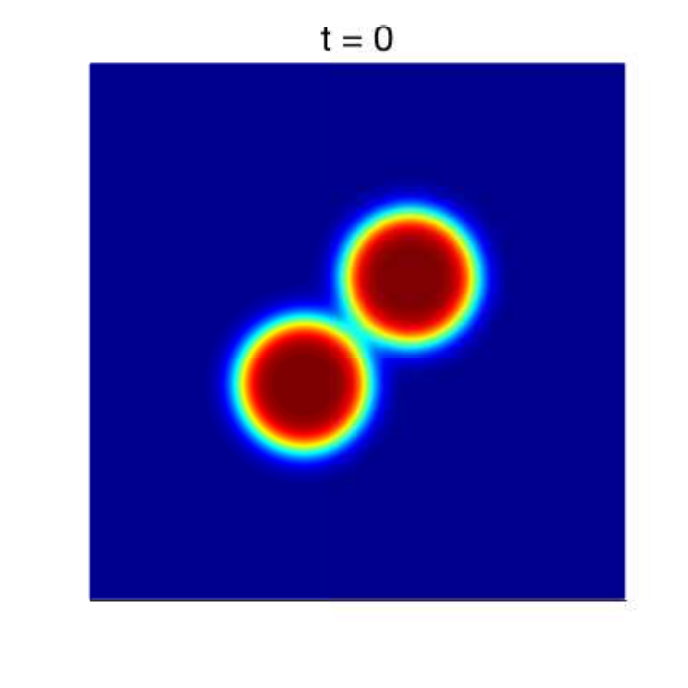}
\hspace{-4mm}
\includegraphics[height=3.36cm,width=3.480cm]{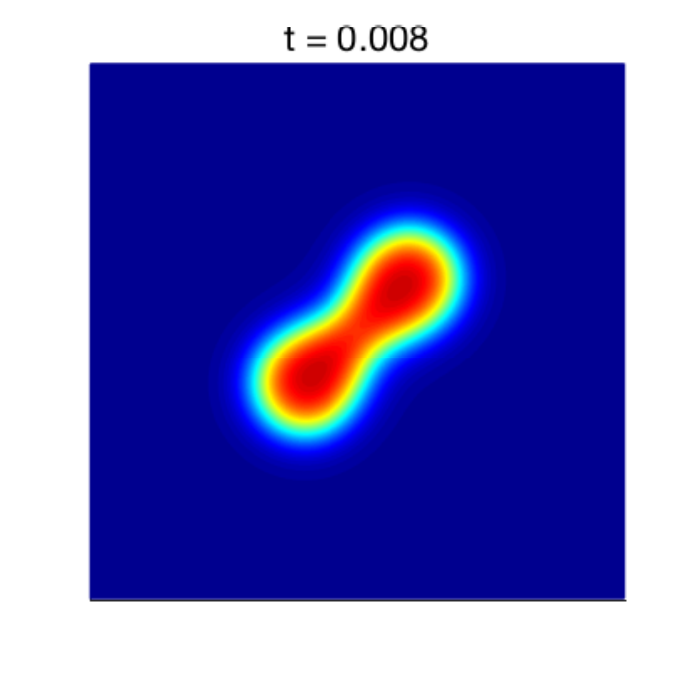}
\hspace{-4mm}
\includegraphics[height=3.36cm,width=3.480cm]{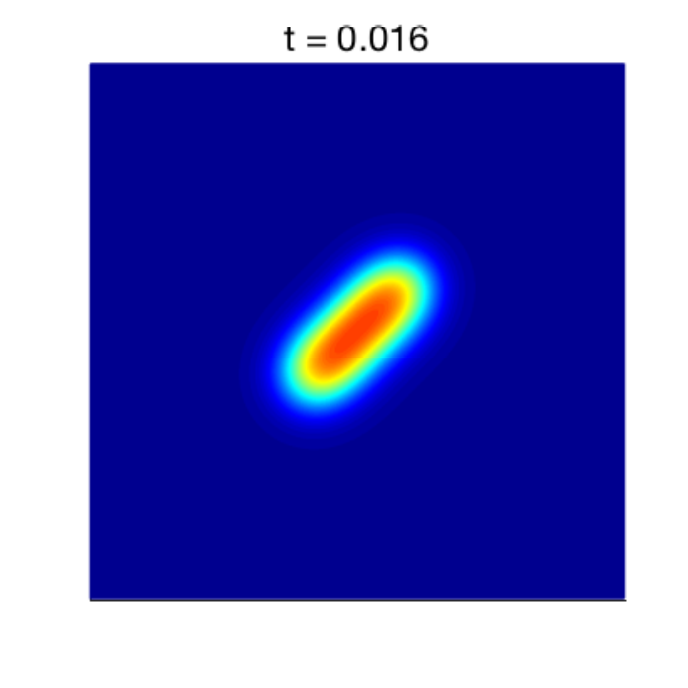}
\hspace{-4mm}
\includegraphics[height=3.36cm,width=3.480cm]{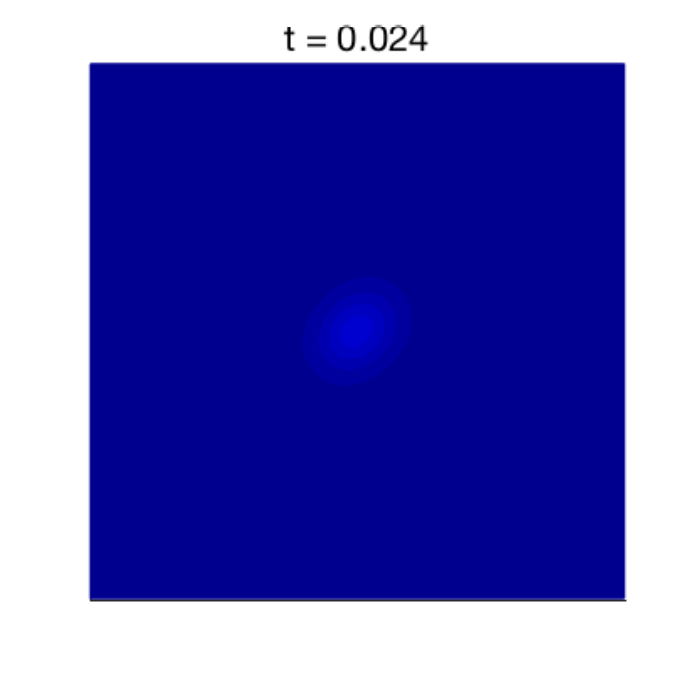}}
\vspace{-5mm}
\centerline{
\includegraphics[height=3.36cm,width=3.480cm]{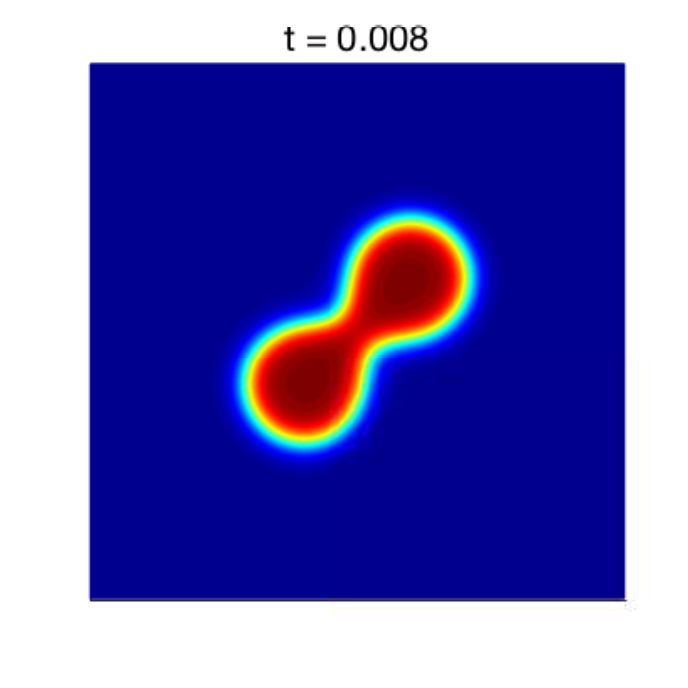}
\hspace{-4mm}
\includegraphics[height=3.36cm,width=3.480cm]{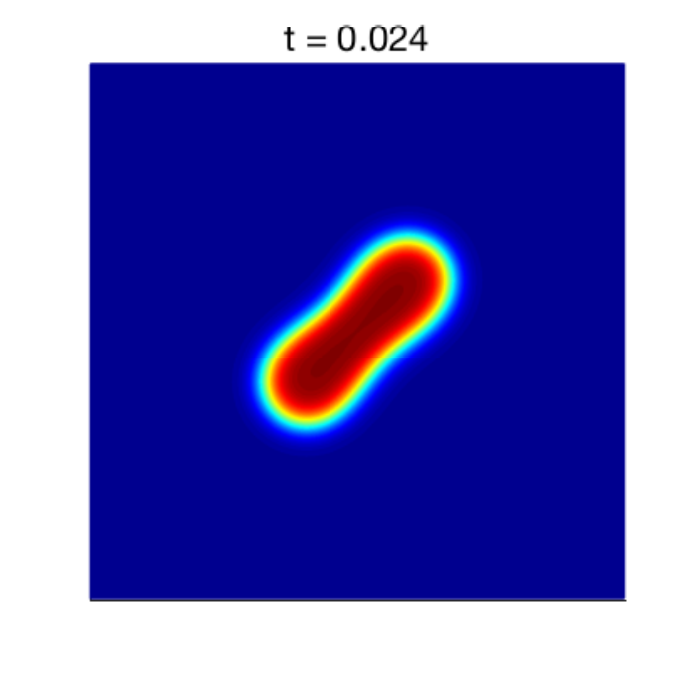}
\hspace{-4mm}
\includegraphics[height=3.36cm,width=3.480cm]{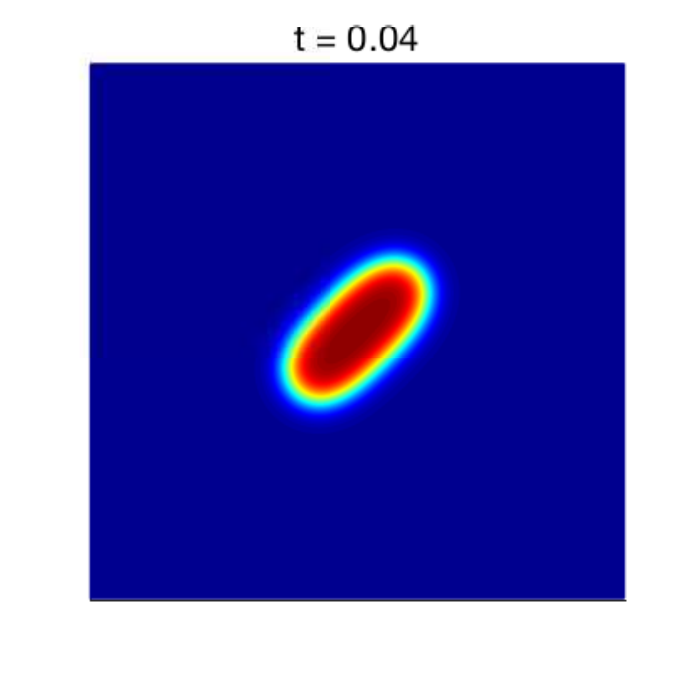}
\hspace{-4mm}
\includegraphics[height=3.36cm,width=3.480cm]{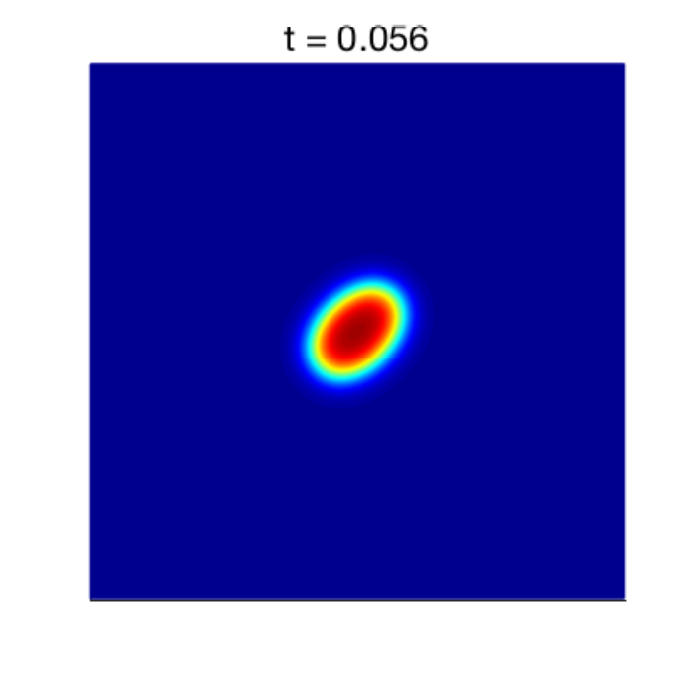}}
\vspace{-5mm}
\caption{Dynamics of the two kissing bubbles in the two-dimensional tempered fractional Allen--Cahn equation with $\varepsilon = 0.03$ and $\ap = 1.8$.  
From top to bottom:  $\lambda = 0, 5$.}\label{Fig-AC2}
\end{figure}
It shows that including the tempered term $e^{-\lambda|\bx - {\bf y}|}$ reduces the long-range interactions in the fractional Laplacian, so as to  slow down the evolution of two bubbles. 
Moreover, the larger the parameter $\lambda$, the slower the evolution, and consequently it takes much longer time for the bubbles to vanish for a larger $\lambda$. 

\subsection{Fractional Gray--Scott equations}
\label{section5-3}

Consider the  fractional Gray--Scott equations of the following form:
\bea\label{GS1}
&&u_t = -\kappa_1(-\Dt + \lambda)^{\fl{\ap}{2}} u - uv^2 + a (1-u), \\
\label{GS2}
&&v_t = -\kappa_2(-\Dt + \lambda)^{\fl{\ap}{2}} v + uv^2 - (a+b)v, 
\eea
where $u$ and $v$ denote the concentration of two species, respectively,  $\kappa_1$ and $\kappa_2$ are diffusion coefficients, $a$ is the feed rate, and $b$ is the depletion rate. 
Here, we take $\kappa_1 = 2 \times 10^{-5}$,  $\kappa_2 = 10^{-5}$, $a = 0.04$, and $b = 0.065$.
Let the domain $\Og = (0, 2.5)^d$.
The system (\ref{GS1})--(\ref{GS2}) admits a trivial solution: $(u, v)  \equiv (1, 0)$.  
We choose the initial condition as  $(u, v) = (1, 0)$ with a perturbation at this  center of the domain, i.e., $(u, v) = (0.5, 0.25)$ {for $\bx \in [1.201,1.299]^2$ in 2D and $\bx \in [1.152,1.348]^3$ in 3D.} 
The boundary conditions of (\ref{GS1})--(\ref{GS2}) are:  $u(\bx, t) = 1$ and $v(\bx, t) = 0$,  for $\bx \in\Og^c$ and $t \ge 0$.

Figs. \ref{Fig-GS1}--\ref{Fig-GS2} show  the pattern formation in the 2D  fractional Gray--Scott equation for various $\ap$ and $\lambda$. 
In our simulations, we choose  $N_x = N_y =  1024$, and time step $\tau = 0.5$. 
It shows that the pattern starts to emerge from the initial perturbation area, and if $\lambda$ is small, it quickly propagates to the boundary of the domain.  
In the classical Gray--Scott equation, a spot pattern was observed for this parameter regime (referred to as pattern-$\lambda$ in \cite{Pearson1993}). 
By contrast, the pattern formation in the fractional cases is more exotic, which significantly depends on the parameter $\ap$ and $\lambda$.  

\begin{figure}[htb!]
\centerline{
\includegraphics[height=3.36cm,width=3.60cm]{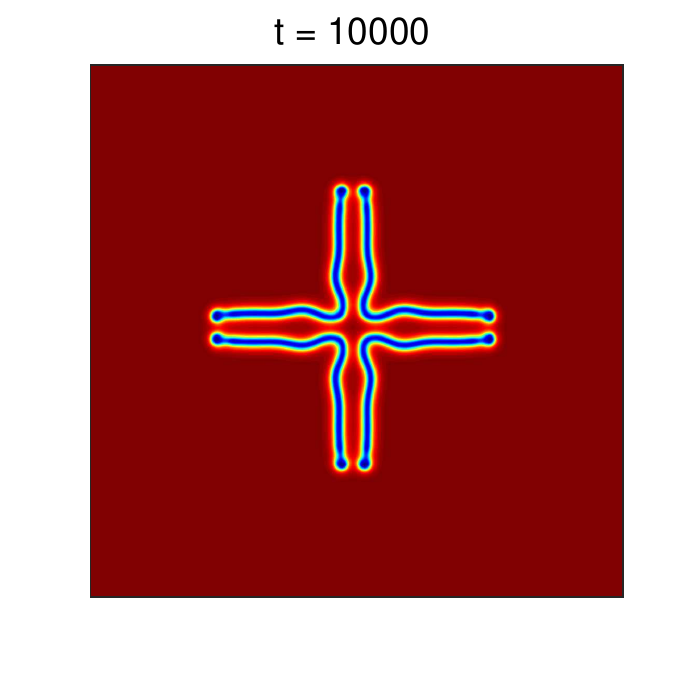}
\hspace{-6mm}
\includegraphics[height=3.36cm,width=3.6cm]{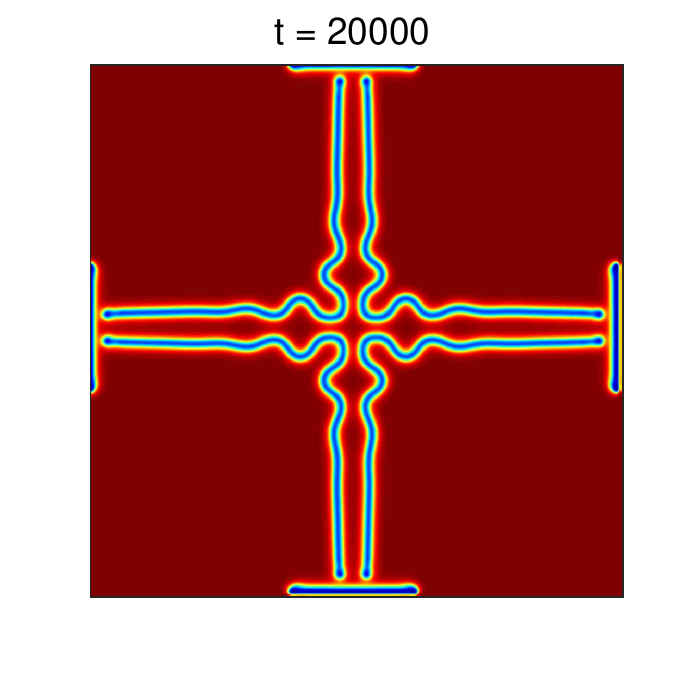}
\hspace{-6mm}
\includegraphics[height=3.36cm,width=3.6cm]{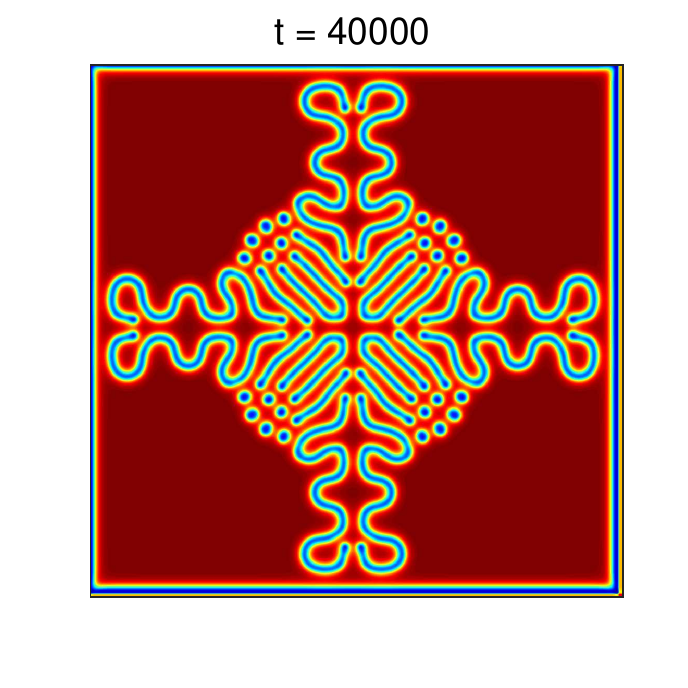}
\hspace{-6mm}
\includegraphics[height=3.36cm,width=3.6cm]{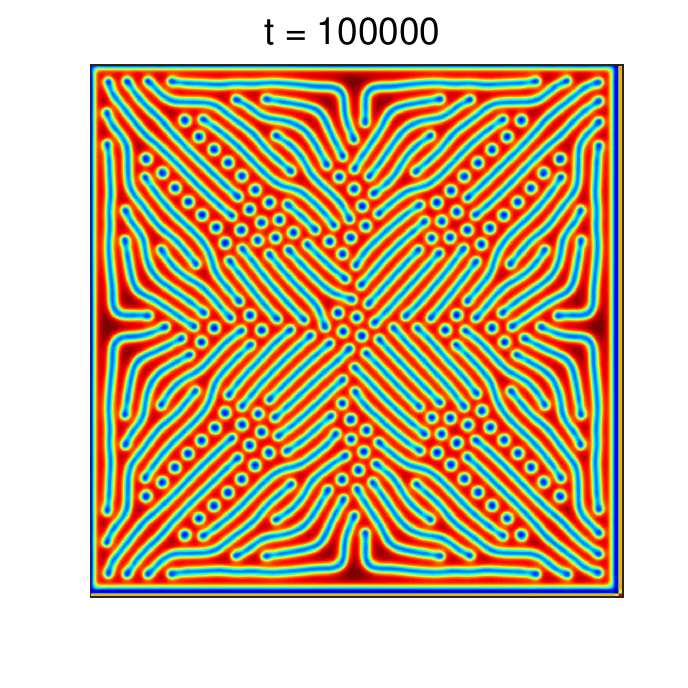}}
\vspace{-5mm}
\centerline{
\includegraphics[height=3.36cm,width=3.6cm]{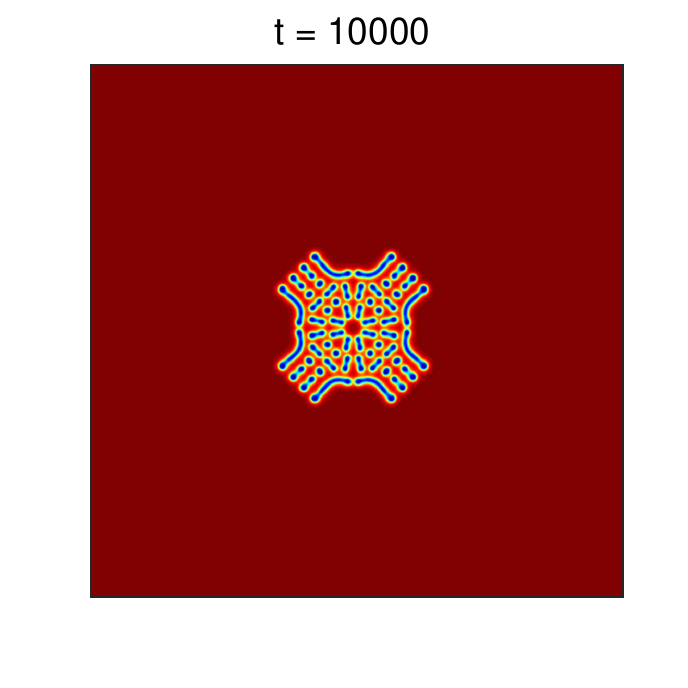}
\hspace{-6mm}
\includegraphics[height=3.36cm,width=3.6cm]{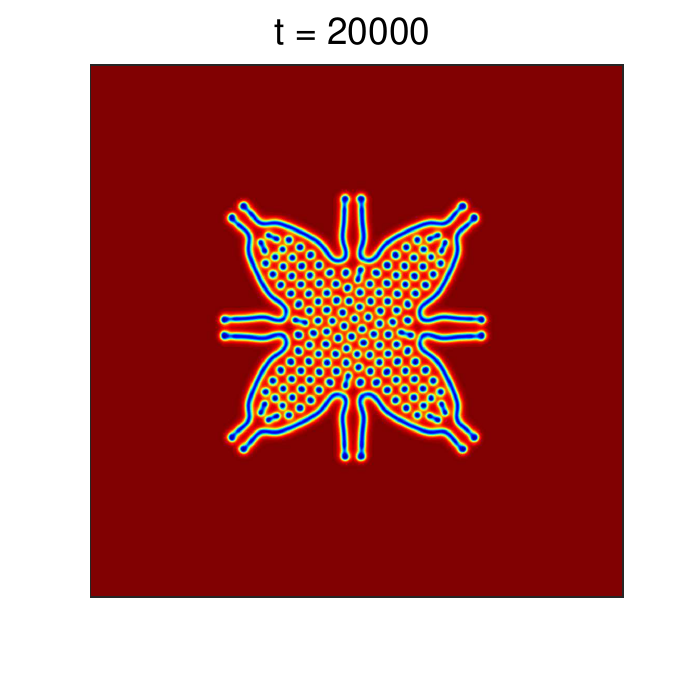}
\hspace{-6mm}
\includegraphics[height=3.36cm,width=3.6cm]{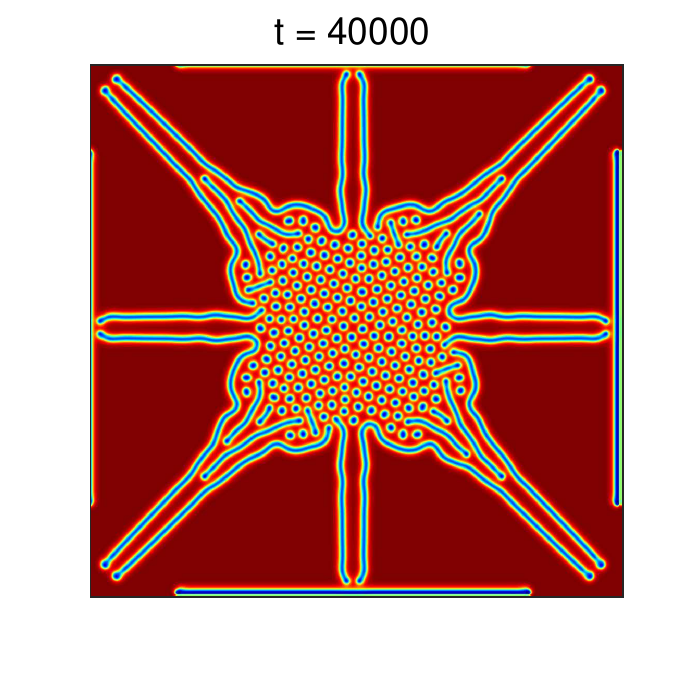}
\hspace{-6mm}
\includegraphics[height=3.36cm,width=3.6cm]{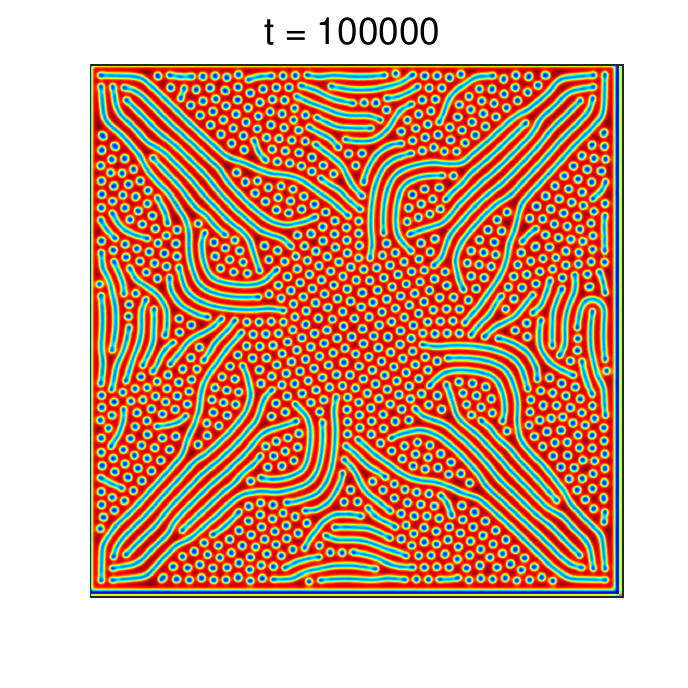}}
\vspace{-5mm}
\centerline{
\includegraphics[height=3.36cm,width=3.6cm]{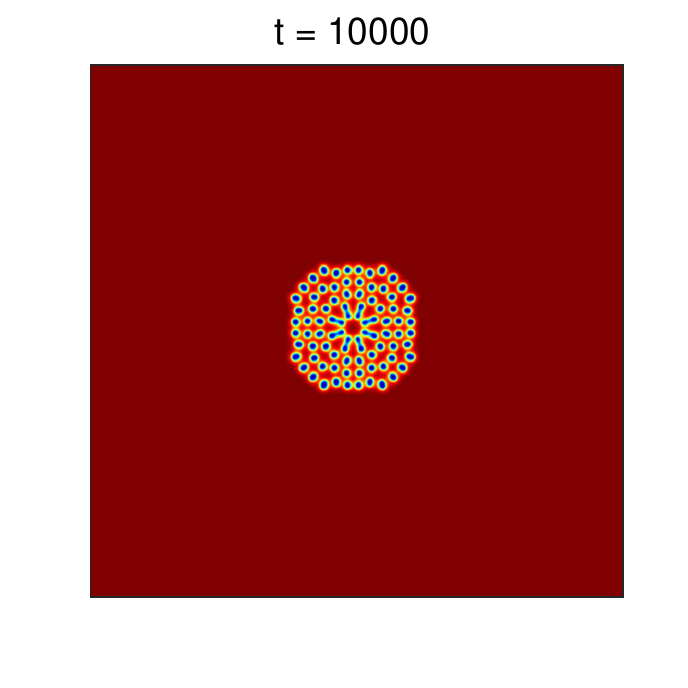}
\hspace{-6mm}
\includegraphics[height=3.36cm,width=3.6cm]{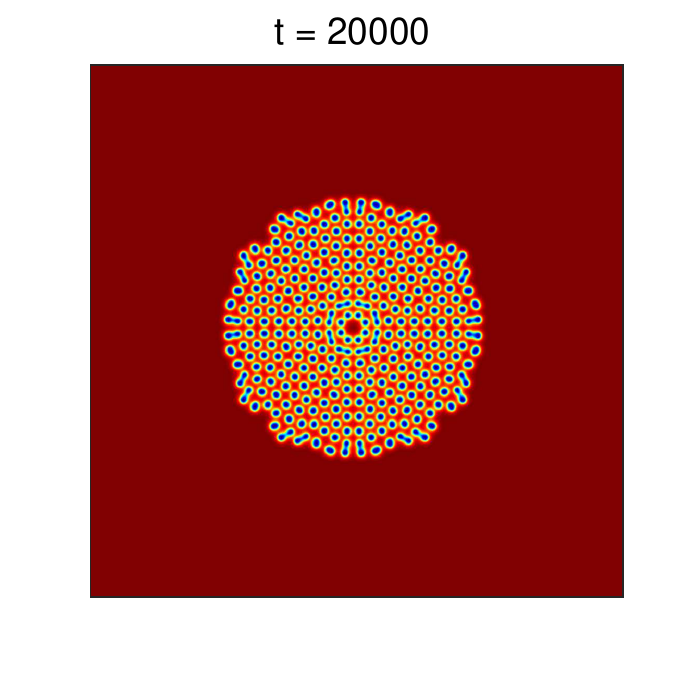}
\hspace{-6mm}
\includegraphics[height=3.36cm,width=3.6cm]{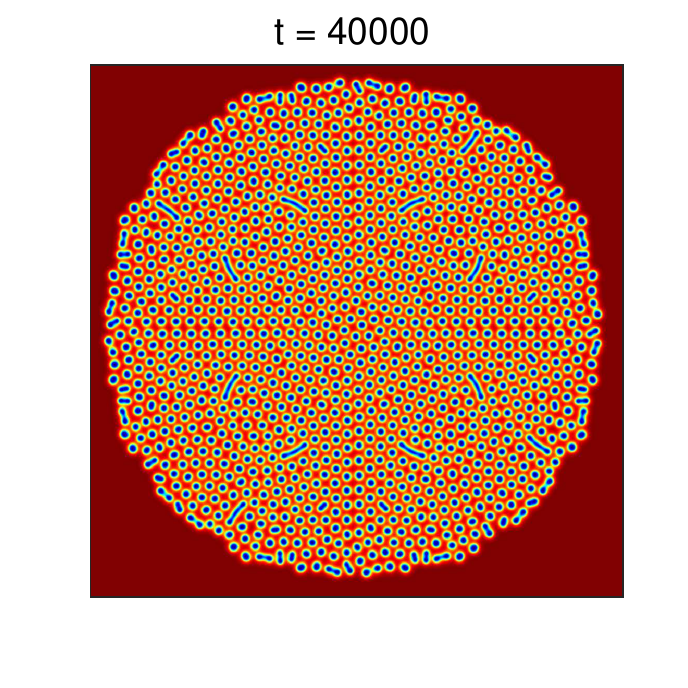}
\hspace{-6mm}
\includegraphics[height=3.36cm,width=3.6cm]{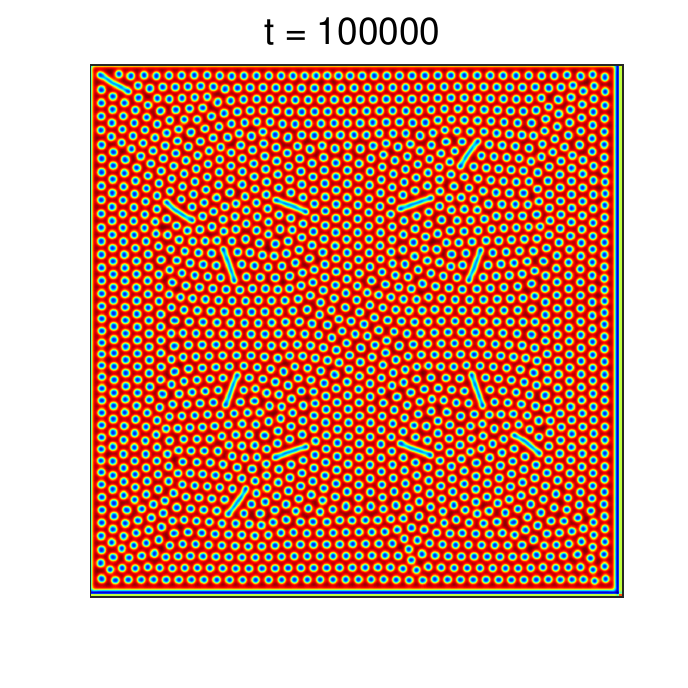}}
\vspace{-5mm}
\caption{Pattern formation in the 2D  Gray--Scott system with $\alpha = 1.8$ and $\lambda = 0, 5, 20$ (from top to bottom). }\label{Fig-GS1}
\end{figure}

In Fig. \ref{Fig-GS1} with fixed $\ap = 1.8$, a mixed pattern of spots and stripes is observed in the steady state of $\lambda = 0$.  
As $\lambda$ increases, the diffusion becomes slow, and the stripes quickly reduce. 
If $\lambda$ is large enough,  a pattern of spots is observed similar to the classical cases, but the structure is much finer due to the fractional dynamics.
In Fig. \ref{Fig-GS2}, we focus on the effects of the superdiffusive power $\ap$ by fixing $\lambda = 5$.  
For a larger $\ap$ (e.g. $\ap = 1.95$), a spot pattern is formed.
It is similar to the classical case in \cite{Pearson1993}, but the spot scale is much smaller. 
With $\ap$ decreasing, a pattern of mixed spots and stripes appears. 
The smaller the values of $\ap$, the more the stripes in the final pattern,  the finer the structure. 
These simulations show the effectiveness of our method in the study of pattern formations even with fine structures.
\begin{figure}[htb!]
\centerline{
\includegraphics[height=3.36cm,width=3.6cm]{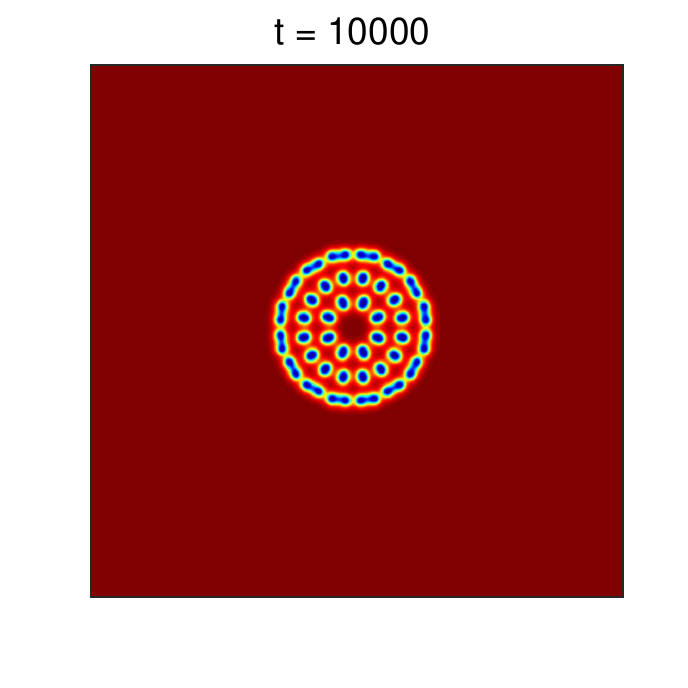}
\hspace{-6mm}
\includegraphics[height=3.36cm,width=3.6cm]{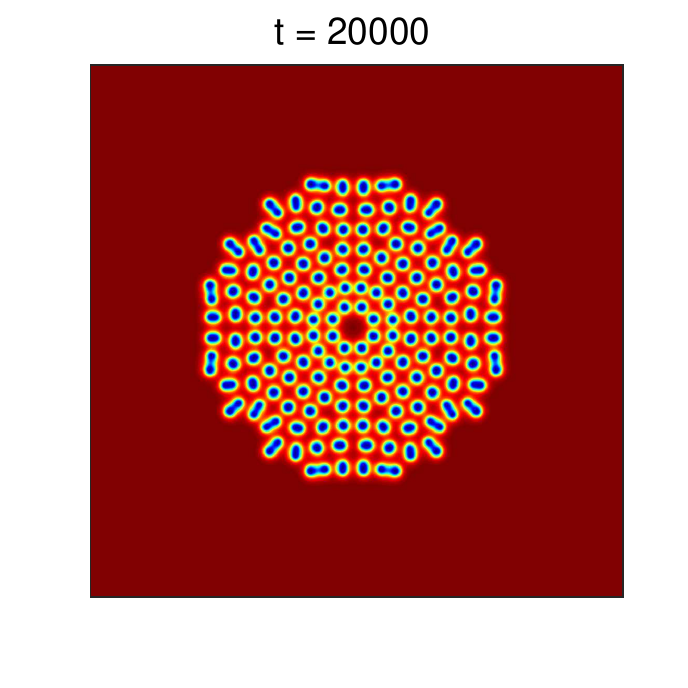}
\hspace{-6mm}
\includegraphics[height=3.36cm,width=3.6cm]{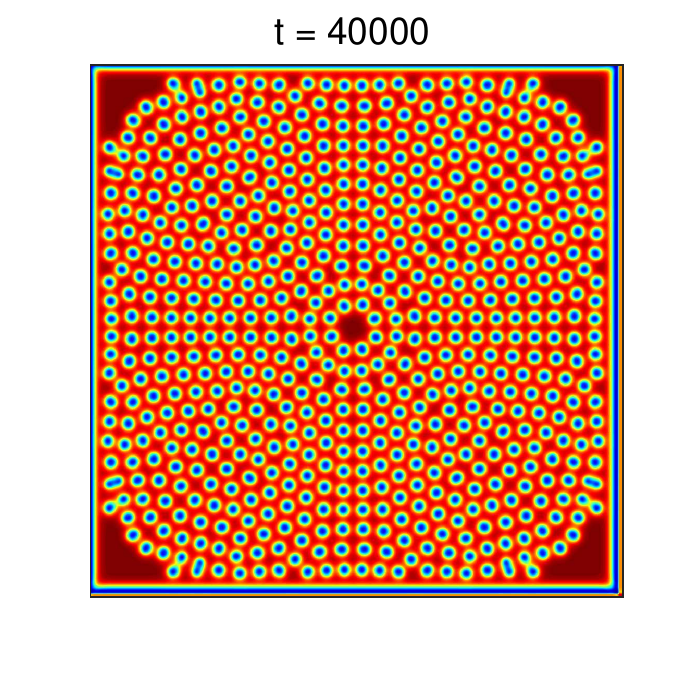}
\hspace{-6mm}
\includegraphics[height=3.36cm,width=3.6cm]{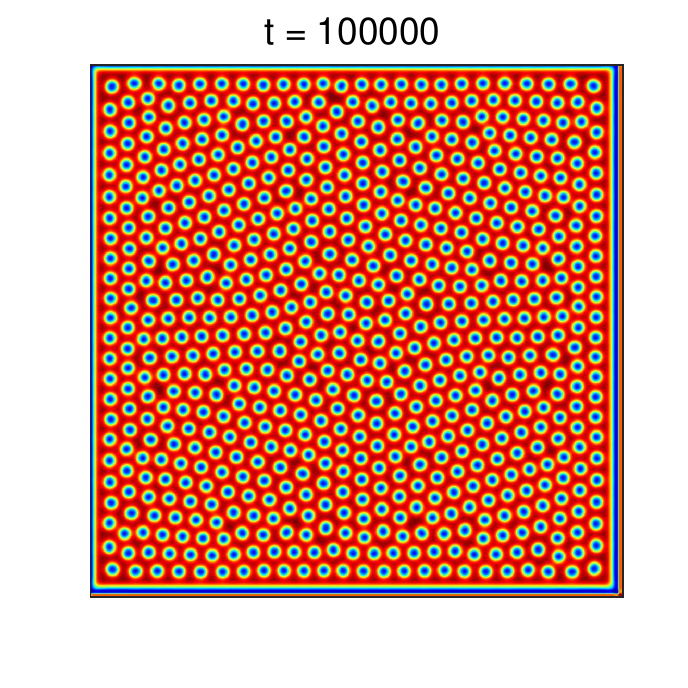}}
\vspace{-5mm}
\centerline{
\includegraphics[height=3.36cm,width=3.6cm]{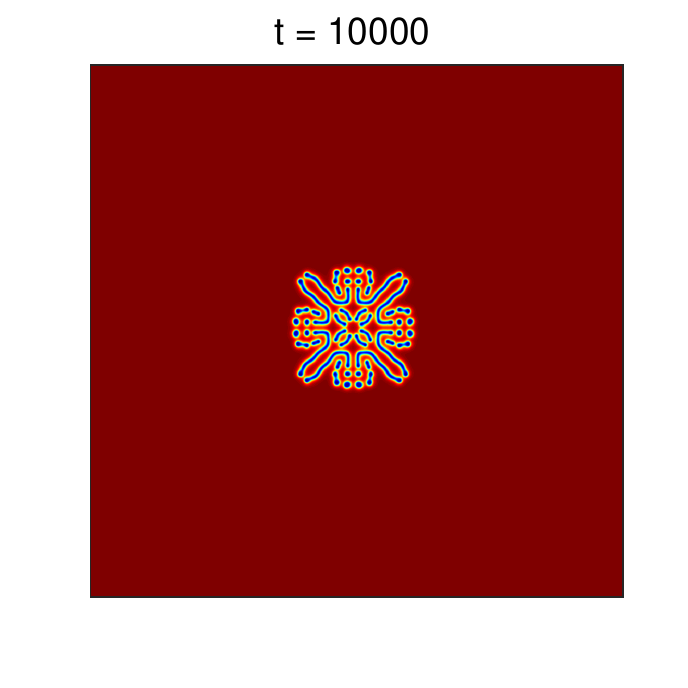}
\hspace{-6mm}
\includegraphics[height=3.36cm,width=3.6cm]{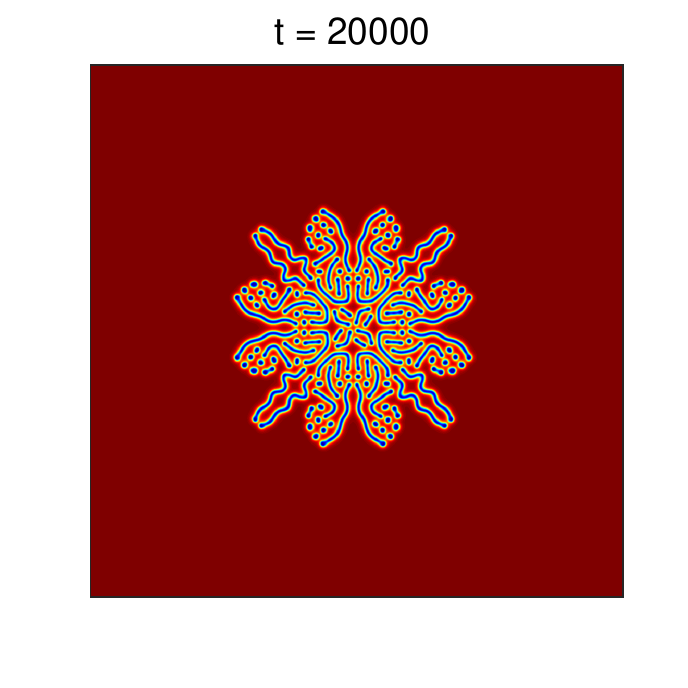}
\hspace{-6mm}
\includegraphics[height=3.36cm,width=3.6cm]{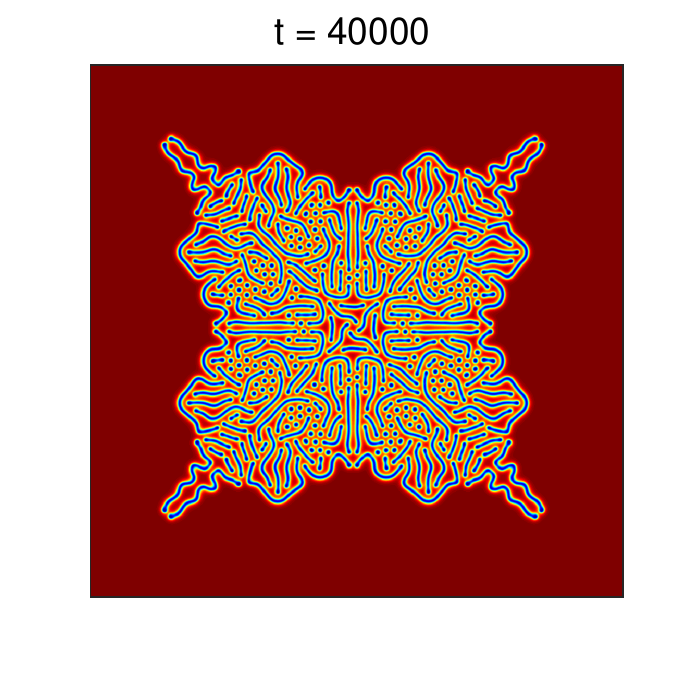}
\hspace{-6mm}
\includegraphics[height=3.36cm,width=3.6cm]{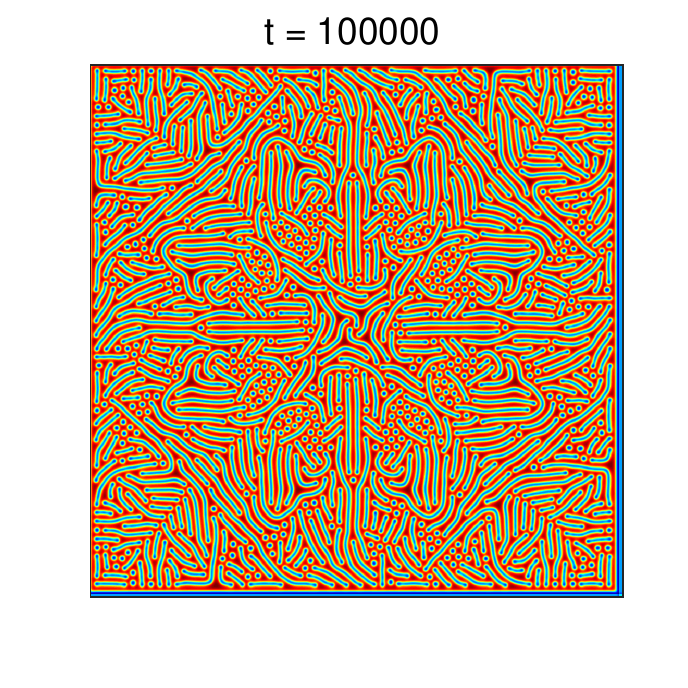}}
\vspace{-5mm}
\caption{Pattern formation in the 2D  Gray--Scott system with $\lambda= 5$, and $\alpha$ = 1.95 (top), 1.7 (bottom).}\label{Fig-GS2}
\end{figure}

\begin{figure}[htb!]
\centerline{
\includegraphics[height=4.45cm,width=4.10cm]{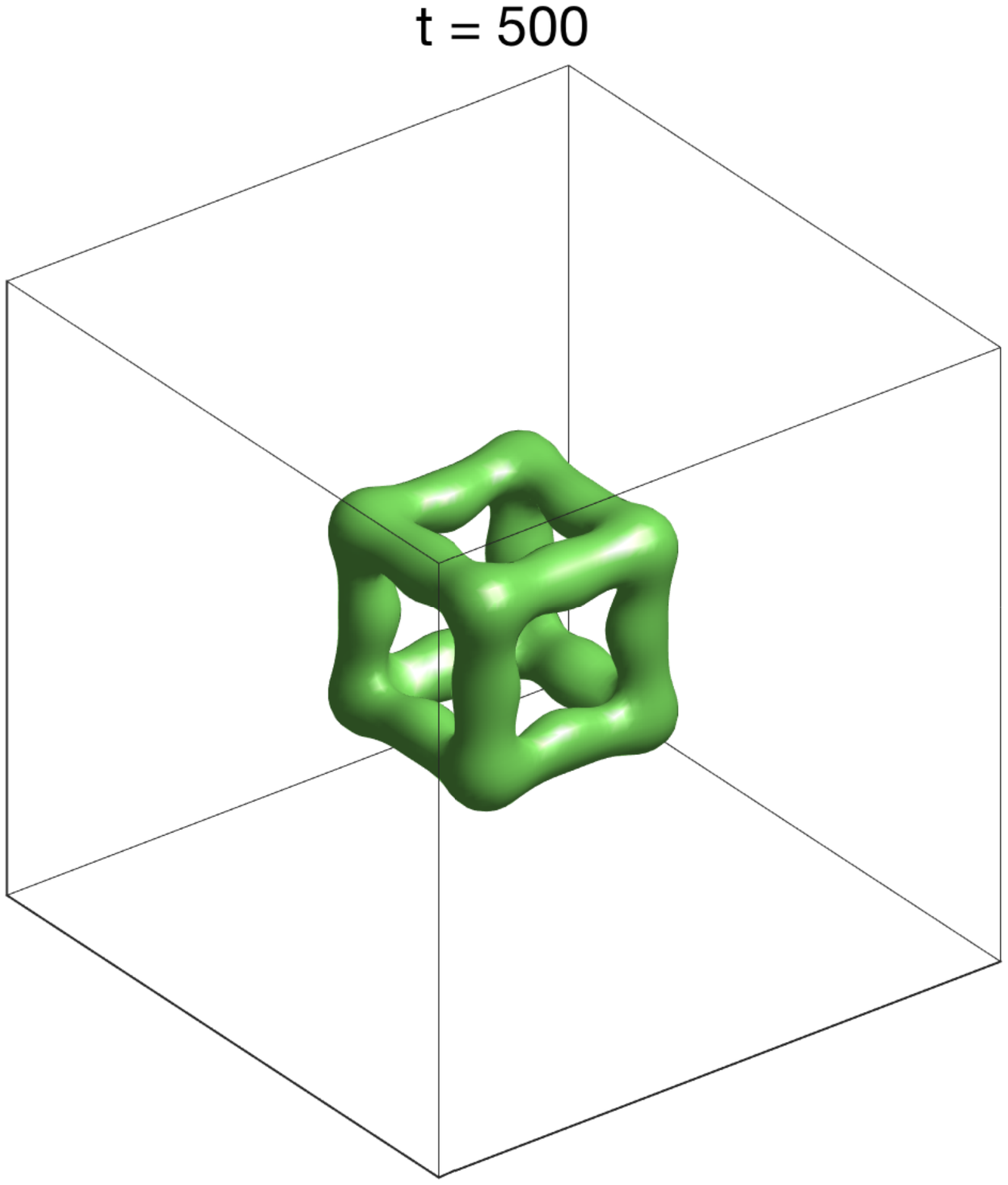}
\hspace{-5mm}
\includegraphics[height=4.45cm,width=4.10cm]{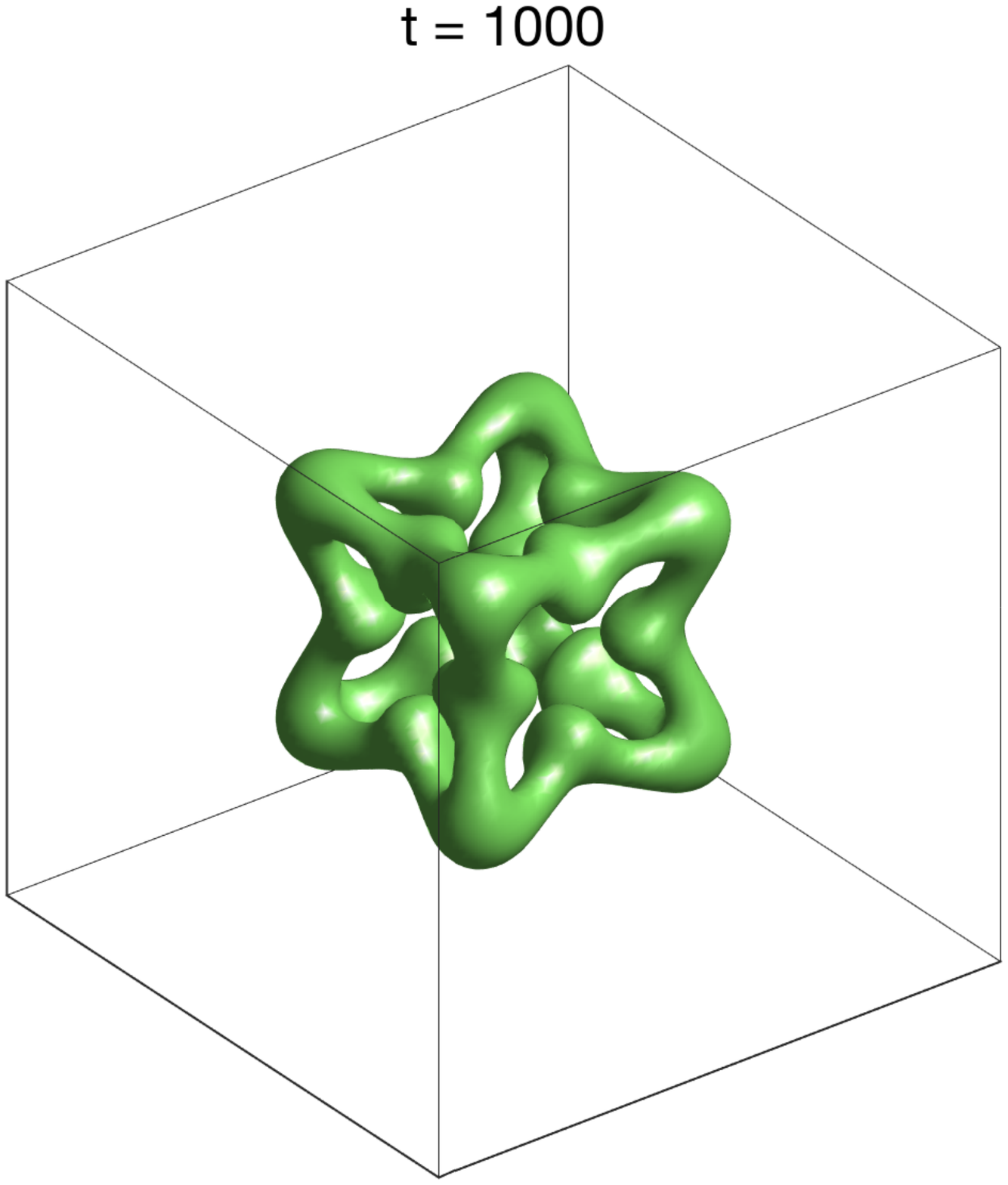}
\hspace{-5mm}
\includegraphics[height=4.45cm,width=4.10cm]{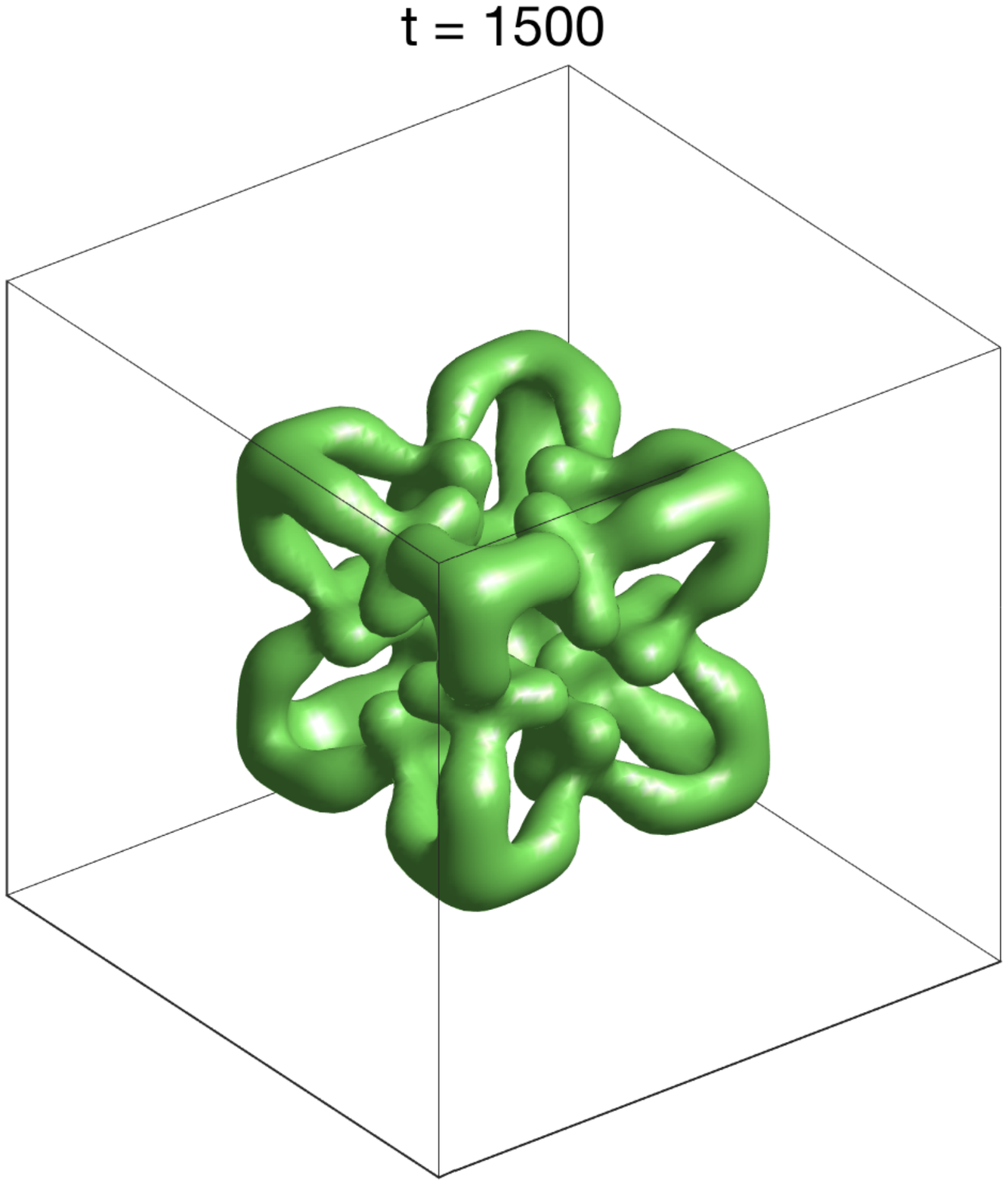}
\hspace{-5mm}
\includegraphics[height=4.45cm,width=4.10cm]{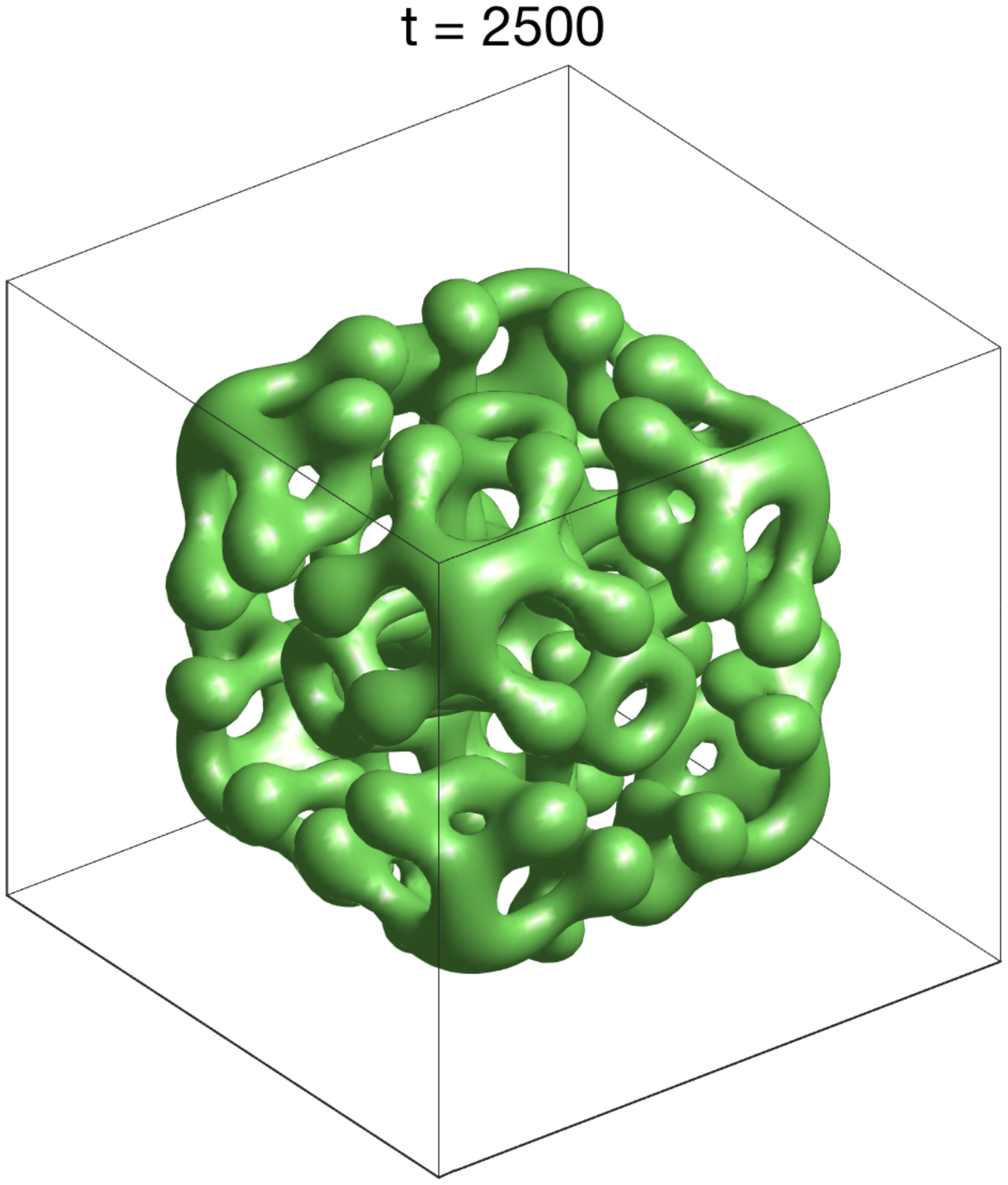}}
\vspace{-8mm}
\caption{Isosurface plots of $u = -0.5$ in the 3D  Gray--Scott  equation with $\alpha = 1.9$ and $\lambda = 0$. } \label{Fig6}
\end{figure}

Next, we further demonstrate the effectiveness of our method by studying the pattern formation in the 3D fractional Gray--Scott equations. 
To the best of our knowledge, so far no numerical results can be found  on the fractional PDEs with the 3D tempered fractional Laplacian, due to the considerable numerical challenges in discretizing the operator. 
Fig. \ref{Fig6} shows the isosurface plots of the component $u$ at different time $t$, where $\ap = 1.9$. 
For a better resolution, only the region of $[0.9,1.6]^3$ is displayed. 
It shows that the 3D fractional Gray-Scott model exhibits more exotic patterns than the 2D cases. 
Comparing to the 2D cases, the computations of 3D systems become more challenging, and our method and fast algorithms  are effective in the simulations. 

\section{Conclusions}
\label{section6}
\setcounter{equation}{0}

We proposed simple and accurate finite difference methods to discretize the $d$-dimensional ($d \ge 1$) tempered fractional Laplacian and provided detailed numerical analysis on their local truncation errors. 
Our analysis not only provides a sharp consistency conditions of our methods but also gives the accuracy  under various smoothness conditions. 
We showed that  the accuracy of our methods can be improved to ${\mathcal O}(h^2)$, independent of the fractional power $\ap$ and damping constant $\lambda$. 
Comparing to other existing methods \cite{Zhang0017, Sun0018}, our method can achieve higher accuracy with low regularity requirements, and are simpler to implement. 
The {multilevel} Toeplitz {stiffness} matrix enables us to develop fast algorithms for the efficient matrix-vector products with computational complexity of order $\mathcal{O}(M\log M)$ and memory storage $\mathcal{O}(M)$ with $M$ the total number of unknowns in space.

Extensive numerical examples were provided to verify the effectiveness of our methods. 
We numerically studied the accuracy of our method in solving tempered fractional Poisson problems and found that to achieve the second order of accuracy, it  only requires the solution $u \in C^{1,1}(\bar{\Og})$ for $\ap \in (0, 2)$. 
Moreover, extensive studies showed that if the solution $u \in C^{p, s}(\bar{\Og})$ for $p = 0, 1$ and $0 \le s \le 1$, our methods have the accuracy of ${\mathcal O}(h^{p+s})$ in solving fractional Poisson problems.
Finally,  the tempered effects were studied in the fractional Allen--Cahn equation and  Gray--Scott system. 
For example, the pattern formation in the tempered  Gray--Scott equation reveals the features of both classical and fractional Laplacian. 
More studies will be carried out in the future to further understand the coupling effects of the normal and anomalous diffusion in  the tempered fractional problems.

\bigskip
{\bf Acknowledgements}. 
This work was supported by the US National Science Foundation under grant number DMS-1620465.
\bibliographystyle{plain}

\end{document}